\documentclass[a4paper,10pt]{article}

\usepackage{graphicx}
\usepackage{latexsym}
\usepackage{amsthm}
\usepackage{enumerate}
\usepackage{autograph2,epic,latexsym,bezier,amsbsy,amssymb,color,psfrag,enumerate,amsfonts,amsmath,amscd,amssymb}
\usepackage[latin1]{inputenc}
\usepackage{GGGraphics,eepic,graphics,pstricks,wrapfig,epsfig}

\setloopdiam{6}\setprofcurve{4}

\newtheorem{defi}{Definition}[section]
\newtheorem{teo}[defi]{Theorem}
\newtheorem{thm}[defi]{Theorem}
\newtheorem{nota}[defi]{Notation}
\newtheorem{lem}[defi]{Lemma}
\newtheorem{prop}[defi]{Proposition}
\newtheorem{cor}[defi]{Corollary}

\newtheorem{rem}[defi]{Remark}

\newtheorem{es}[defi]{Example}

\begin{document}

\title{Schreier graphs of the Basilica group}

\author{D. D'Angeli, A. Donno, M. Matter and T. Nagnibeda}

\date{\today}

\maketitle

\begin{abstract}
With any self-similar action of a finitely generated group $G$ of
automorphisms of a regular rooted tree $T$ can be naturally
associated an infinite sequence of finite graphs
$\{\Gamma_n\}_{n\geq 1}$, where $\Gamma_n$ is the Schreier graph
of the action of $G$ on the $n$-th level of $T$. Moreover, the
action of $G$ on $\partial T$ gives rise to orbital Schreier
graphs $\Gamma_{\xi}$, $\xi\in \partial T$. Denoting by $\xi_n$
the prefix of length $n$ of the infinite ray $\xi$, the rooted
graph $(\Gamma_{\xi},\xi)$ is then the limit of the sequence of
finite rooted graphs $\{(\Gamma_n,\xi_n)\}_{n\geq 1}$ in the sense
of pointed Gromov-Hausdorff convergence. In this paper, we give a
complete classification (up to isomorphism) of the limit graphs
$(\Gamma_{\xi},\xi)$ associated with the Basilica group acting on
the binary tree, in terms of the infinite binary sequence
$\xi$.\footnote{This research has been supported by the Swiss
National Science Foundation Grant PP0022$_{-}$118946.}
\end{abstract}

\begin{center}{{\bf Mathematics Subject Classification (2010)}
Primary 20E08; Secondary 20F69, 05C63, 37E25.}
\end{center}

\section{Introduction}

Schreier graphs arise naturally from the action of a group on a set. In this paper, we consider groups acting by automorphisms on rooted trees. Let $T$ be a regular rooted tree and $G<Aut(T)$ be a finitely generated group of automorphisms of $T$. By fixing a finite set $S$ of generators of $G$, we get naturally a sequence $\{\Gamma_n\}_{n\geq 1}$ of finite left Schreier graphs of the action of $G$ on $T$. The vertex set of $\Gamma_n$ coincides with the set $L_n$ of vertices of the $n$-th level of $T$, and two vertices $v,v' \in L_n$ are connected by an edge if there exists $s\in S$ such that $s\cdot v= v'$. If $G$ is transitive on each level, then the graphs $\Gamma_n$ are connected.\\
\indent Similarly, the action of $G$ on the boundary $\partial T$ of the tree gives rise to an uncountable family of infinite orbital Schreier graphs $\{\Gamma_{\xi}\}_{\xi\in \partial T}$, with $V(\Gamma_{\xi}) = G\cdot\xi$. It turns out that these orbital Schreier graphs, viewed as rooted graphs $(\Gamma_\xi,\xi)$, are exactly the limits in the pointed Gromov-Hausdorff topology of finite Schreier graphs $(\Gamma_n,\xi_n)$ rooted at $\xi_n$, the prefix of length $n$ of $\xi$ (see Subsection \ref{subsecSchreierGraphs}). Moreover, if we choose the root of $\Gamma_n$ uniformly at random for all $n\geq 1$, then the uniform measure on the set $\{(\Gamma_\xi,\xi);\xi\in\partial T\}$ is the random weak limit (in the sense of \cite{BenSchr01}) of this sequence of random rooted graphs.\\
\indent Particularly interesting examples of such sequences of Schreier graphs come from the class of self-similar, or automata, groups.
In this paper, we examine in detail finite and infinite
Schreier graphs of the Basilica group $B$ acting on the binary tree
in a self-similar fashion. The Basilica group is an
example of a group generated by a finite automaton (see Fig. 2).
It was introduced by R.~Grigorchuk and A.~\.{Z}uk in
\cite{grizuk}, where they show that it does not belong to the
closure of the set of groups of subexponential growth under the
operations of group extension and direct limit. L.~Bartholdi and
B.~Vir\' ag further showed it to be amenable, making Basilica the
first example of an amenable but not subexponentially amenable
group \cite{amenability}.
 This group has also been described by V.~Nekrashevych as the iterated monodromy group of the complex polynomial
 $z^2-1$ (see \cite{nekrashevyc}), and there exists therefore a natural way to associate to it a compact limit space homeomorphic to
 the well-known Basilica fractal (see Fig. 1). Finite Schreier graphs $\{\Gamma_n\}_{n\geq 1}$ form an approximating sequence of this limit space.\\
\indent The aim of this work is to classify (up to isomorphism of unrooted and unlabeled graphs) all limits $\{\Gamma_{\xi}\}_{\xi\in
\partial T}$ of the sequence $\{\Gamma_n\}_{n\geq 1}$ of finite Schreier graphs of the Basilica group. In other words, the question that we answer here is: given an infinite binary sequence $\xi$, describe the infinite Schreier graph $\Gamma_{\xi}$ of the action of $B$ on the orbit of $\xi$. The limit graph $\Gamma_{\xi}$ is shown to have one, two or four ends, the case of one end being generic with respect to the uniform measure on $\partial T$ (this latter fact can also be deduced from general arguments, see \cite{BDDN}). The main result of the paper is an explicit classification of infinite Schreier graphs of the Basilica group in terms of the boundary point $\xi$ (Theorems \ref{totaltheorem}, \ref{4isomorfismo}, \ref{xixibarra}, \ref{2isomorfismo} and \ref{ThmBijection}.) There exist one isomorphism class of $4$-ended graphs, all belonging to the same orbit; uncountably many isomorphism classes of $2$-ended graphs, each consisting of two orbits; and uncountably many isomorphism classes of $1$-ended graphs, almost all of which contain uncountably many orbits. This latter aspect is particularly interesting, as it is not the case in other examples where we were able to perform a similar analysis. For example in the case of the Hanoi towers group $H^{(3)}$, whose finite Schreier graphs form an approximating sequence for the Sierpi\'nski gasket, there are uncountably many isomorphism classes of graphs with one end, but each of them contains at most $6$ different orbits. We intend to pursue this investigation in a future work.\\
\indent It follows (see Section \ref{SubsectionRandomLimit} below) that in the case of the Basilica group there exist uncountably many isomorphism classes of (unlabeled) limit graphs, each of measure $0$. In other words, the random weak limit of the sequence of finite Schreier graphs is a continuous measure with uncountable support. A.~Vershik recently raised the question about existence of continuous ergodic probability measures on the lattice $L(G)$ of subgroups of a given group $G$, invariant under the action of $G$ on $L(G)$ by conjugation \cite{Ver}. For many self-similar groups including the Basilica (more precisely, for all weakly branched groups), the pull-back of the random weak limit of labeled finite Schreier graphs to $L(G)$ gives such a measure, concentrated on stabilizers of points in the boundary of the tree.\\
\indent It would be interesting to perform a similar analysis systematically on sequences of finite Schreier graphs associated with self-similar groups, in particular with iterated monodromy groups of quadratic polynomials, whose infinite Schreier graphs are closely related to the Julia set of the polynomial (see the recent work \cite{BarDud} for a detailed description of this relation.) In particular, the inflation recursive process of constructing finite Schreier graphs introduced in \cite{nekrashevyc} (and used in \cite{bondarenko} to study growth of infinite Schreier graphs) might also give a possible general approach to classification of infinite Schreier graphs of contracting groups of automorphisms of rooted trees, in the spirit of the present work. For example, though there are examples of groups (the first Grigorchuk group of intermediate growth, the adding machine...) with only one or two isomorphism classes of infinite (unlabeled) Schreier graphs, it seems that those are exceptional cases and that in general the support of the random weak limit of $\Gamma_n$'s is uncountable. We plan to address this question in a future work.\\
\indent Sequences of Schreier graphs of self-similar groups have been mainly studied in the literature from the viewpoint of spectral computations (see e.g. \cite{hecketype},\cite{GrigSunik06},\cite{rogers}...) However, in contrast with other well-known examples of self-similar groups, the spectral measure of the Laplace operator on infinite Schreier graphs of the Basilica group is not known. Explicit description of these graphs that we obtain in this paper can be useful in this aspect.\\

\vspace{0.5cm}

\begin{center}
\includegraphics[angle=90, width=0.4\textwidth]{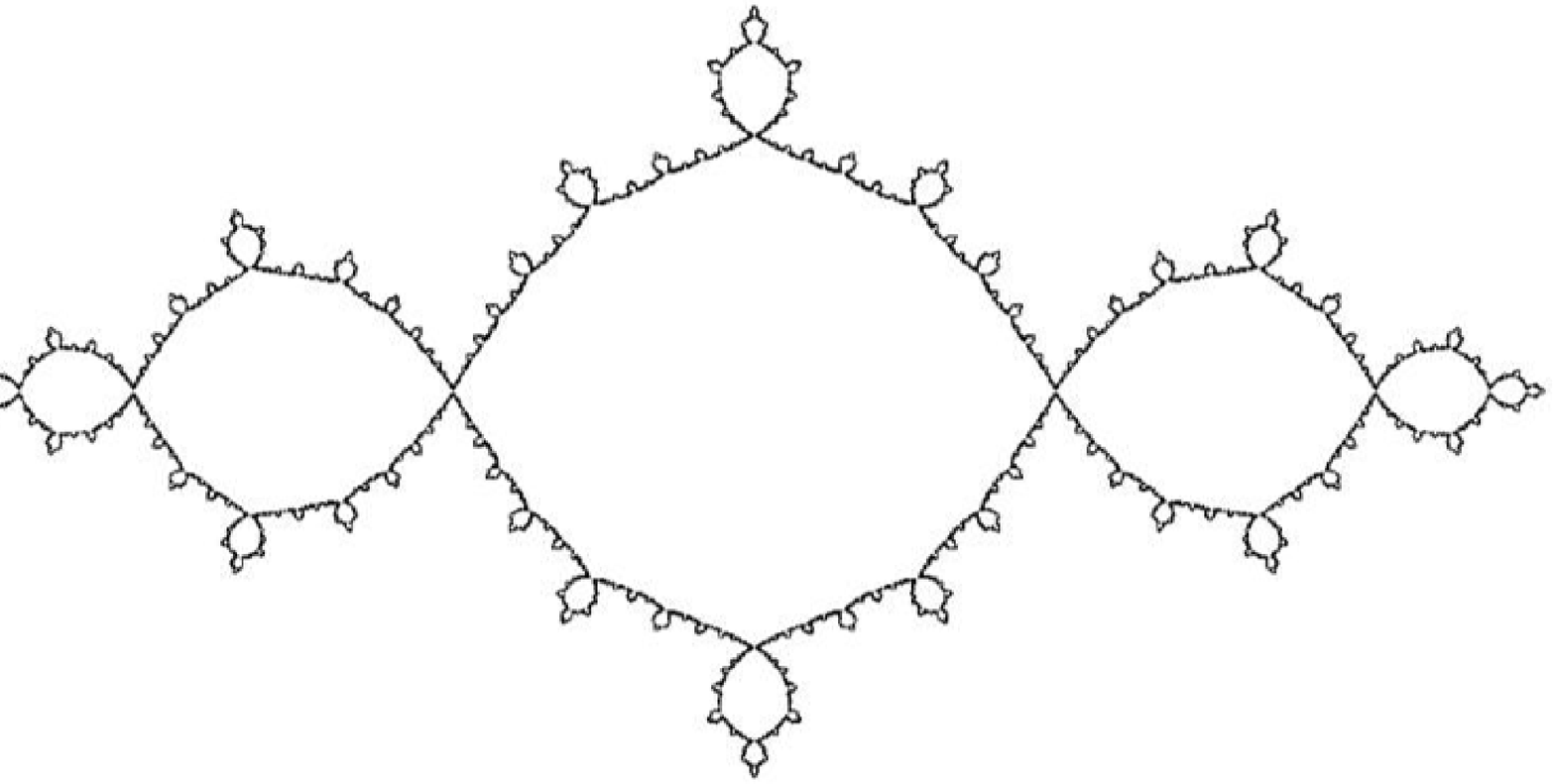}
\end{center}
\begin{center}
\textbf{Fig. 1.} the Julia set $J(z^2-1)$.
\end{center}

Our initial motivation comes from  the work \cite{MatNag09} that aims  at constructing new examples of asymptotic behaviour for the Abelian Sandpile Model (ASM). The infinite Schreier graphs of the Basilica group provide uncountable families of examples, on one hand, of $2$-ended graphs non quasi-isometric to $\mathbb{Z}$, on which the ASM is not critical and, on the other hand, of $1$-ended graphs of quadratic growth, on which the ASM is critical.\\

\section{Preliminaries}

\subsection{Groups acting on rooted trees}\label{groupsactingontrees}

\indent Let $T$ be the regular rooted tree of degree $q$, i.e., the rooted
tree in which each vertex has $q$ children. Given a finite alphabet $X = \{0,1, \ldots, q-1\}$ of $q$
elements, let us denote by $X^n$ the set of words of length $n$ in
the alphabet $X$ and put $X^{\ast}=\bigcup_{n\geq 0}X^n$, where
the set $X^0$ consists of the empty word. Moreover, we denote by
$X^{\omega}$ the set of infinite words in $X$. In this way, each
vertex of the $n$-th level of the tree can be regarded as an element of $X^n$ and the set
$X^{\omega}$ can be identified with the set $\partial T$ of
infinite geodesic rays starting at the root of $T$. The set
$X^{\omega}$ can be equipped with the direct product topology. The
basis of open sets is the collection of all cylindrical sets $\{\textrm{$wX^\omega;$ $w\in X^{\ast}$}\}$.
The space $X^{\omega}$ is totally disconnected and homeomorphic to
the Cantor set. The cylindrical sets generate a $\sigma$-algebra of Borel subsets of the space $X^{\omega}$. We shall denote by $\lambda$ the uniform measure on $X^\omega$.\\
\indent We denote by $Aut(T)$ the group of all automorphisms of $T$, i.e.,
the group of all bijections of the set of vertices of $T$ preserving the incidence relation. Clearly, the root and hence the levels of the tree are preserved by any automorphism of $T$. A group $G\leq Aut(T)$ is said to be \textbf{spherically transitive} if it acts transitively on each level of the tree.\\
\indent The \textbf{stabilizer} of a vertex $x\in T$ is the
subgroup of $G$ defined as
$Stab_G(x)=\{g\in G : g(x)=x\}$; the \textbf{stabilizer of the} $n${\bf -th level} $L_n$ of the tree is
$Stab_G(L_n)=\bigcap_{x\in L_n}Stab_G(x)$; finally,
the \textbf{stabilizer of a boundary point} $\xi\in X^{\omega}$ is $Stab_G(\xi)=\{g\in G : g(\xi)=\xi\}$. The following properties hold.
\begin{itemize}
\item For all $n \geq 1$, $Stab_G(L_n)$ is a normal subgroup of $G$ of
finite index.
\item The subgroups $Stab_G(x)$, for $x \in L_n$, are
all of index $q^n$. Moreover, if the action of $G$ on $T$ is spherically
transitive, they are all conjugate.
\item $\bigcap_{\xi \in \partial T}Stab_G(\xi)$ is trivial.
\item Denote by $\xi_n$ the prefix of $\xi$ of length $n$. Then
$Stab_G(\xi) = \bigcap_{n\in \mathbb{N}} Stab_G(\xi_n)$.
\item $Stab_G(\xi)$ has infinite index in $G$ and $Stab_G(\xi)/(Stab_G(\xi)\cap Stab_G(L_n)) \cong Stab_G(\xi_n)/Stab_G(L_n)$.
\end{itemize}

If $g\in Aut(T)$ and $v\in X^*$, define $g|_v\in Aut(T)$, called  \textbf{the restriction of the action of} $g$ \textbf{to the subtree rooted at} $v$, by
$g(vw)=g(v)g|_v(w)$ for all $v,w\in X^*$. Every subtree of $T$ rooted at a vertex is isomorphic to $T$. Therefore,
every automorphism $g\in Aut(T)$ induces a permutation of the vertices of the first level of the tree and $q$ restrictions, $g|_0,...,g|_{q-1},$ to the subtrees rooted at the vertices of the first level. It can be written as $g=\tau_g(g|_0,\ldots,g|_{q-1})$,
where $\tau_g\in S_q$ describes the action of $g$ on $L_1$.
In fact, $Aut(T)$ is isomorphic to the wreath product $S_q\wr Aut(T)$ where $S_q$ denotes the symmetric group on $q$ letters, and thus $Aut(T)\cong \wr_{i=1}^{\infty} S_q$.

\begin{defi}{\rm \cite{nekrashevyc}}
A group $G$ acting by automorphisms on a $q$-regular rooted tree $T$ is
\textbf{self-similar} if $g|_v\in G$, $\forall v\in X^*, \forall g\in G$.
\end{defi}

A self-similar group $G$ can be embedded into the wreath
product $S_q\wr G$. Consequently, an automorphism
$g\in G$ can be represented as
$g=\tau_g(g|_0,\ldots,g|_{q-1})$,
where $\tau_g\in S_q$ describes the action of $g$ on $L_1$, and
$g|_i\in G$ is the restriction of the action of $g$ on the subtree
$T_i$ rooted at the $i$-th vertex of the first level. So, if $x\in X$ and $w$ is a finite word in $X$, we have
$g(xw)=\tau_g(x)g|_x(w)$.

\begin{defi}\emph{\cite{nekrashevyc}}\label{selfreplicating} A self-similar group
$G$ is \textbf{self-replicating} (or fractal) if it acts transitively on the first level of the tree and, for all $x\in X$,
the map $g\mapsto g|_x$ from $Stab_G(x)$ to $G$ is surjective.
\end{defi}

\subsection{Schreier graphs} \label{subsecSchreierGraphs}

\indent Consider a finitely generated group $G$ with a set $S$ of generators such that $id \not \in S$ and $S = S^{-1}$, and suppose that $G$ acts on a set $M$. Then, one can consider a graph $\Gamma(G,S,M)$ with the set of vertices $M$, and two vertices $m,m'$ joined by an edge (labeled by $s$) if there exists $s\in S$ such that $s(m)=m'$. Clearly, if the action of $G$ on $M$ is transitive, then $\Gamma(G,S,M)$ is the \textbf{Schreier graph} $\Gamma(G,S,Stab_G(m))$ of the group $G$ with respect to the subgroup $Stab_G(m)$ for some (any) $m\in M$. If the action of $G$ on $M$ is not transitive, and $m\in M$, then we denote by $\Gamma(G,S,m)$ the Schreier graph of the action on the $G$-orbit of $m$, and we call such a graph \textbf{an orbital
Schreier graph}.\\
\indent Suppose now that $G$ acts spherically transitively on a rooted tree $T$. Then, the \textbf{$n$-th Schreier graph} of $G$ is by definition $\Gamma(G,S,X^n)=\Gamma(G,S,P_n)$ where $P_n$ denotes the subgroup stabilizing some word $w\in X^n$. For each $n\geq 1$, let $\pi_{n+1}:\Gamma(G,S,X^{n+1})\longrightarrow \Gamma(G,S,X^n)$
be the map defined on the vertex set of $\Gamma(G,S,X^{n+1})$ by $\pi_{n+1}(x_1\ldots x_nx_{n+1}) = x_1\ldots x_n$. Since $P_{n+1}\leq P_n$, $\pi_{n+1}$ induces a surjective morphism between $\Gamma(G,S,X^{n+1})$ and
$\Gamma(G,S,X^n)$. This morphism is a graph covering of degree $q$.\\
\indent We also consider the action of $G$ on $\partial T\equiv X^\omega$ and the orbital Schreier graph $\Gamma(G,S,G\cdot\xi)=\Gamma(G,S,P_\xi)$ where $P_\xi$ denotes the stabilizer of $\xi$ for the action of $G$ on $X^\omega$. Recall that, given a ray $\xi$, we denote by $\xi_n$ the prefix of $\xi$ of length $n$, and that $P_{\xi} = \cap_n P_n$. It follows that the infinite Schreier graph $\Gamma_{\xi}=\Gamma(G,S,P_{\xi})$ can be approximated (as a rooted graph) by finite Schreier graphs $\Gamma_n=\Gamma(G,S,P_n)$, as $n\rightarrow\infty$, in the compact space of rooted graphs of uniformly bounded degree endowed with \textbf{pointed Gromov-Hausdorff convergence} (\cite{Grom}, Chapter 3), provided, for example, by the following metric:
 given two rooted graphs $(\Gamma_1,v_1)$ and $(\Gamma_2,v_2)$,
$$Dist((\Gamma_1,v_1),(\Gamma_2,v_2)):=\inf\left\{\frac{1}{r+1};\textrm{$B_{\Gamma_1}(v_1,r)$
is isomorphic to $B_{\Gamma_2}(v_2,r)$}\right\}$$
where $B_{\Gamma}(v,r)$ is the ball of radius $r$ in $\Gamma$ centered in $v$.\\

\subsection{Self-similar groups and automata}

\indent An \textbf{automaton} is a quadruple $\mathcal{A} =
(\mathcal{S},X,\mu,\nu)$, where $\mathcal{S}$ is the set of states; $X$ is an alphabet;
$\mu: \mathcal{S}\times X \rightarrow \mathcal{S}$ is the transition map;
and $\nu: \mathcal{S}\times X \rightarrow X$ is the output map.
The automaton $\mathcal{A}$ is \textbf{finite} if $\mathcal{S}$ is
finite and it is \textbf{invertible} if, for all $s\in
\mathcal{S}$, the transformation $\nu(s, \cdot):X\rightarrow
X$ is a permutation of $X$. An automaton $\mathcal{A}$ can be
represented by its \textbf{Moore diagram}. This is a directed
labeled graph whose vertices are identified with the states of
$\mathcal{A}$. For every state $s\in \mathcal{S} $ and every
letter $x\in X$, the diagram has an arrow from $s$ to $\mu(s,x)$
labeled by $x|\nu(s,x)$. A natural action on the words over
$X$ is induced, so that the maps $\mu$ and $\nu$ can be
extended to $\mathcal{S}\times X^{\ast}$:
$$\mu(s,xw) = \mu(\mu(s,x),w),$$
\begin{eqnarray}\label{actionextended}
\nu(s,xw) = \nu(s,x)\nu(\mu(s,x),w),
\end{eqnarray}
where we set $\mu(s,\emptyset) = s$ and $\nu(s,\emptyset) =
\emptyset$, for all $s\in \mathcal{S}, x\in X$ and $w\in
X^{\ast}$. Moreover, (\ref{actionextended}) defines uniquely a map
$\nu: \mathcal{S} \times X^{\omega}\rightarrow X^{\omega}$.\\
\indent If we fix an initial state $s$ in an automaton $\mathcal{A}$, then
the transformation $\nu(s,\cdot)$ on the set $X^{\ast}\cup
X^{\omega}$ is defined by (\ref{actionextended}); it is denoted by
$\mathcal{A}_s$. The image of a word $x_1x_2\ldots$ under
$\mathcal{A}_s$ can be easily found using the Moore diagram (see, for instance, Fig. 2 below).
Consider the directed path starting at the state $s$
with consecutive labels $x_1|y_1$, $x_2|y_2,...$; the image
of the word $x_1x_2\ldots$ under the transformation
$\mathcal{A}_s$ is then $y_1y_2\ldots$. More generally,
given an invertible automaton $\mathcal{A} =
(\mathcal{S},X,\mu,\nu)$, one can consider the group generated
by the transformations $\mathcal{A}_s$, for $s\in \mathcal{S}$;
this group is called the \textbf{automaton group} generated by
$\mathcal{A}$ and is denoted by $G(\mathcal{A})$.\\
\indent A basic theorem \cite{nekrashevyc} states that the action of a group $G$ on $X^{\ast}\cup
X^{\omega}$ is self-similar if and only if $G$ is generated by an invertible
automaton.\\
\indent Let $\mathcal{A}$ be a finite automaton with the set of states $\mathcal{S}$ and alphabet $X$ and let us denote $\alpha(k,s)$,
for $k\in \mathbb{N}$ and $s\in \mathcal{S}$, the number of words $w\in X^k$ such that $s|_w \neq id$.
Sidki suggested to call $\mathcal{A}$ \textbf{bounded}, if the sequence $\alpha(k,s)$ is bounded as a function of $k$ for each state $s\in \mathcal{S}$.
He showed in \cite{sidki} that a finite invertible automaton is bounded if and only if any two non-trivial cycles in the Moore diagram of the automaton are disjoint and not connected by a directed path.\\
\indent It can be shown moreover, that any group generated by a bounded automaton is \textbf{contracting} \cite{postcritically}, which means, for a self-similar group,
 the existence of a finite set $\mathcal{N}\subset G$ such that for every $g\in G$ there exists $k\in \mathbb{N}$ such that $g|_v\in \mathcal{N}$, for all words $v$ of length greater or equal to $k$.

\subsection{The Basilica group}\label{beginningBasilica}

\indent The Basilica group $B$ was introduced by R. Grigorchuk and A. \.{Z}uk \cite{grizuk} as the group generated by the following three-state automaton.


\begin{center}
\begin{picture}(300,145)
\thicklines \setvertexdiam{20} \setprofcurve{15}\setloopdiam{20}
\letstate A=(75,140)
\letstate C=(225,90)
\letstate B=(75,40)

\drawstate(A){$a$} \drawstate(B){$b$} \drawstate(C){$id$}

\drawcurvededge(A,C){$1|1$}
\setprofcurve{-20}\drawcurvededge[r](B,C){$1|0$}
\drawloop[r](C){$0|0,1|1$} \drawcurvededge[r](A,B){$0|0$}
\drawcurvededge[r](B,A){$0|1$} \put(30,0){\textbf{Fig. 2.} The
automaton generating the Basilica group.}
\end{picture}
\end{center}

\noindent It can be read from the automaton that the Basilica group is an automorphism group of the rooted binary tree generated by two automorphisms $a$ and $b$ of the following self-similar structure:
$$
a=e(b,id) \ \ \ \ \ \ \ b=\varepsilon (a,id),
$$
where $id$ denotes the trivial automorphism of the tree, while $e$ and $\varepsilon$ are respectively the identity and the nontrivial permutation
in $S_2$. One can verify directly that the stabilizer of the first level is the subgroup
$Stab_B(L_1)= <a, a^b,b^2>,$
with $a^b = b^{-1}ab = e(id,b^a)$ and $b^2 = e(a,a)$. This implies in particular, that $B$ is self-replicating. Moreover, since the action of $B$ on the first level of the tree is transitive, its action is also spherically transitive.
It follows from the results cited in the previous subsection that the Basilica group is a contracting self-similar group generated by a
bounded automaton.\\
\indent We end this section by an observation about the action of the Basilica group on the boundary of the binary tree that we will need in Section 4 below.

\begin{defi}
Two right-infinite sequences $x_1x_2\ldots$, $y_1y_2\ldots \in X^{\omega}$ are called \textbf{cofinal} (denoted $x_1x_2\ldots\sim y_1y_2\ldots$) if they differ only in finitely many letters. cofinality is an equivalence relation. The respective equivalence classes are called the cofinality classes and they are denoted by $Cof(\cdot)$.
\end{defi}

\begin{prop}\label{cofinality}
The union of the cofinality classes $Cof(0^{\omega})\cup Cof((01)^{\omega})\cup Cof((10)^{\omega})$ constitutes one orbit of the action of the Basilica group on $X^{\omega}$. Any other cofinality class is exactly one orbit.
\end{prop}

\begin{proof}
One directly verifies that $a(0^{\omega}) = (01)^{\omega}$, $a^{-1}(0^{\omega}) = 010^{\omega}$, $b(0^{\omega}) =
(10)^{\omega}$ and $b^{-1}(0^{\omega}) = 10^{\omega}$. Moreover, it follows from the definition of the generators $a$ and $b$, that $0^\omega$, $(01)^\omega$ and $(10)^\omega$ are the only infinite words which are mapped onto infinite words not cofinal to them by some of the generators ($a,b$ for $0^\omega$, $a^{-1}$ for $(01)^\omega$ and $b^{-1}$ for $(10)^\omega$). Thus, the orbit $B\cdot0^\omega$ is contained in the union $Cof(0^{\omega})\cup Cof((01)^{\omega})\cup Cof((10)^{\omega})$. On the other hand, we show that any word $\eta\in Cof(0^{\omega})\cup Cof((01)^{\omega})\cup Cof((10)^{\omega})$ belongs to the orbit of $0^{\omega}$ by providing an automorphism in $B$ mapping $0^{\omega}$ to $\eta$. We only discuss the case of words of the type $\eta=w0^{\omega}$ (the other cases are analogous). Suppose that $|w|=k$. By transitivity, there exists $g \in B$ such that $g(0^k)=w$; set $g':=g|_{0^k}$. Since $B$ is self-replicating, there exists $h\in Stab_{B}(w)$ such that $h|_w=(g')^{-1}$. This gives $hg(0^{\omega})=w0^{\omega}$.\\
\indent  Since $0^\omega$, $(01)^\omega$ and $(10)^\omega$ are the only infinite words which are mapped onto infinite words not cofinal to them by some of the generators, the orbit of any $\xi\in X^\omega\backslash (Cof(0^\omega)\cup Cof((01)^\omega)\cup Cof((10)^\omega))$ is contained in one cofinality class. The other inclusion is proven by the same argument as used to show that any word $\eta\in Cof(0^{\omega})\cup Cof((01)^{\omega})\cup Cof((10)^{\omega})$ belongs to the orbit of $0^{\omega}$.
\end{proof}

\section{Schreier graphs of the Basilica group}\label{sectionSchreier}

\subsection{The structure of finite Schreier graphs}\label{finiteSch}

\indent For each $n\geq 1$, let us denote by $\Gamma_n\equiv
\Gamma(B,\{a,b\},\{0,1\}^n)$ the $n$-th Schreier graph of the action of
the Basilica group. Recall that the edges of $\Gamma_n$ are
labeled by the generators $a$,$b$ of the group $B$ and that its
vertices are encoded by words of length $n$ in the alphabet $\{0,1\}$.
We begin this subsection by providing some convenient substitutional rules which
allow to construct $\Gamma_n$'s recursively.

\begin{prop}\label{rulesbas}
The Schreier graph $\Gamma_{n+1}$ is obtained from $\Gamma_n$ by
applying to all subgraphs of $\Gamma_n$ given by single edges the
following substitutional rules \textbf{SR}:


\begin{center}
\begin{picture}(400,125)

\put(112,117){SR1}\put(192,117){SR2}\put(292,117){SR3}

\letvertex A=(120,100)\letvertex B=(100,20)\letvertex C=(140,20)

\letvertex D=(180,100)\letvertex E=(220,100)\letvertex F=(180,20)\letvertex G=(220,20)
\letvertex H=(280,100)\letvertex I=(320,100)\letvertex L=(260,10)\letvertex M=(300,20)\letvertex N=(340,10)
\put(117,60){$\Downarrow$}\put(197,60){$\Downarrow$}\put(297,60){$\Downarrow$}

\put(117,92){$1w$}\put(97,11){$11w$}\put(137,11){$01w$}

\put(177,92){$u$}\put(217,92){$v$}\put(177,11){$0u$}\put(217,11){$0v$}

\put(277,92){$0u$}\put(317,92){$0v$}
\put(257,1){$00u$}\put(296,10){$10v$}\put(337,1){$00v$}


\drawvertex(A){$\bullet$}\drawvertex(B){$\bullet$}
\drawvertex(C){$\bullet$}\drawvertex(D){$\bullet$}
\drawvertex(E){$\bullet$}\drawvertex(F){$\bullet$}
\drawvertex(G){$\bullet$}\drawvertex(H){$\bullet$}
\drawvertex(I){$\bullet$}\drawvertex(L){$\bullet$}
\drawvertex(M){$\bullet$}\drawvertex(N){$\bullet$}

\drawundirectedloop(A){$a$}\drawundirectedloop[l](B){$a$}
\drawundirectedcurvededge(B,C){$b$}\drawundirectedcurvededge(C,B){$b$}

\drawundirectededge(D,E){$b$} \drawundirectededge(F,G){$a$}
\drawundirectededge(H,I){$a$}
\drawundirectedcurvededge(L,M){$b$}\drawundirectedloop(M){$a$}
\drawundirectedcurvededge(M,N){$b$}
\end{picture}
\end{center}
with
\begin{center}
\begin{picture}(200,40)
\letvertex A=(70,25)\letvertex B=(130,25)

\put(67,16){$0$}\put(127,16){$1$}\put(20,21){$\Gamma_1 =$}

\drawvertex(A){$\bullet$}\drawvertex(B){$\bullet$}

\drawundirectedloop[l](A){$a$}\drawundirectedloop[r](B){$a$}
\drawundirectedcurvededge(A,B){$b$}\drawundirectedcurvededge(B,A){$b$}
\end{picture}
\end{center}

\end{prop}

\begin{proof}
 By definition of $a$,
we can distinguish two types of $a$-edges: a first type joining
vertices $0w$ and $0b(w)$, and a second type which are loops at
vertices $1w$, with $w\in \{0,1\}^n$. Similarly, we distinguish
the following types of $b$-edges:
\begin{itemize}
\item $b(00w) = 10b(w)$ (note that $w\neq b(w)$);
\item $b(01w) = 11w$;
\item $b(10w) = 00w$;
\item $b(11w) = 01w$.
\end{itemize}
All $b$-edges can be partitioned in pairs.
Pairs of $b$-edges of second and fourth type connecting vertices $11w$ and $01w$ in
$\Gamma_{n+1}$ arise from $a$-loops at the vertex $1w$ in
$\Gamma_n$; moreover, there must be an $a$-loop based at $11w$ by
definition of $a$, hence SR1.\\
\indent The $b$-edges of first and third type are paired in chains of length $2$. Such a chain connecting $00w$ to $10b(w)$ to $00b(w)$ in
$\Gamma_{n+1}$ arises from the $a$-edge joining vertices $0w$ and
$0b(w)$; moreover, there must be an $a$-loop based at $10b(w)$, hence SR3.\\
\indent Finally, $a$-edges between vertices $0u$ and $0v$ in
$\Gamma_{n+1}$ are
in bijection with $b$-edges joining $u$ and $v$ in $\Gamma_n$, hence SR2.\\
\end{proof}

\indent Below follow the pictures of the Schreier graphs $\Gamma_n$ for some first values of $n\geq 1$.

\input{Figures/Examples}

\indent Given $k\in\mathbb{N}$, recall that a
graph $G$ is $k$-connected if for every proper subset
$Y\subset V(G)$ with $|Y|<k$, $G\setminus Y$ is connected. A
connected graph $G$ is \textbf{separable} if it can be
disconnected by removing only one vertex. Such a vertex is called
a \textbf{cut} vertex. The biggest $2$-connected components of a
separable graph are called \textbf{blocks}.\\
\indent The following is an easy consequence of the substitutional rules (Proposition \ref{rulesbas}):

\begin{prop} \label{BasicProp}
For every $n\geq 1$, $\Gamma_n$ is a $4$-regular separable graph whose blocks are cycles. Every vertex without a loop in $\Gamma_n$ is a cut vertex. Moreover, removing any cut vertex disconnects $\Gamma_n$ into exactly two components. Finally, the maximal length of a cycle in $\Gamma_n$ is $2^{\lceil\frac{n}{2}\rceil}$.
\end{prop}

For the remainder of the paper, it will be convenient to consider Schreier graphs $\Gamma_n$, $n\geq 1$, embedded in the plane in such a way that each cycle is a regular
polygon and the graph $\Gamma_n$ has two symmetry axes, a horizontal one and a vertical one.
The center of the central cycle of $\Gamma_n$ coincides with the origin of the plane,
which is the intersection of the axes of symmetry. By convention, the positive rotation by an angle $\alpha$ around the origin is performed in the counterclockwise direction.\\
\indent From now on we will forget about the labels on the edges of the Schreier graphs and only consider unlabeled graphs. On one hand, our aim is to classify limits of Schreier graphs up to isomorphism of unlabeled graphs. On the other hand, in the case of the Basilica group the labeling is uniquely determined by the graph, so we are not losing any information by forgetting the labels.\\
\indent Recall from Subsection \ref{subsecSchreierGraphs} that the finite Schreier graphs $\{\Gamma_n\}_{n\geq 1}$ form a sequence of graph coverings and that we denote by $\pi_{n+1}:\Gamma_{n+1}\rightarrow\Gamma_n$ the covering projection given by $\pi_{n+1}(x_1\ldots x_nx_{n+1}) = x_1\ldots x_n$.
In the remaining part of this subsection we shall describe in detail the structure of the graphs $\Gamma_n$ and of the projections $\pi_n$.

\begin{defi}
Let $v\in V(\Gamma_n)\backslash\{0^n\}$ be a cut vertex. Removing $v$ splits $\Gamma_n$ into two connected components, $U_1$ and $U_2$, one of them (say $U_1$)
containing the vertex $0^n$. We call the \textbf{decoration} $\mathcal{D}(v)$ of $v$ the subgraph of $\Gamma_n$ induced by the vertex set $V(U_2)\cup \{v\}$.
If $v\in V(\Gamma_n)$ has a loop, then $\mathcal{D}(v)$ is the subgraph induced by $v$. Finally, if $v=0^n$, then $\mathcal{D}(0^n)$ is the
subgraph induced by $V(U_i)\cup \{0^n\}$, where $0^{n-1}1\notin U_i$.\\
\indent For any cut vertex $v$ of $\Gamma_n$, we also consider the subgraph $\mathcal{D}(v)^c$ induced by $V(U_1)\cup \{v\}$.
If $v$ has a loop attached to it, then $\mathcal{D}(v)^c$ is just $\Gamma_n$ with the corresponding loop erased.\\
\indent The decoration of a given vertex $v\in V(\Gamma_n)$ is
called a \textbf{k-decoration} (or a \textbf{decoration of height k}) if it is isomorphic to the
decoration of the vertex $0^k$ in the Schreier graph $\Gamma_k$ for some $1\leq k\leq n$.
\end{defi}

\begin{prop}\label{PropMultiple}
\begin{enumerate}
\item Every decoration in $\Gamma_n$ is a $k$-decoration for some $1\leq
k\leq n$.
\item Let $v\in V(\Gamma_n)\backslash\{0^n\}$. Then, the decoration
$\mathcal{D}(v)$ of $v$ is a $k$-decoration if and only if, while reading $v$ from the left, we first encounter $1$ in the $k$-th position.
\item  Let $v\in
V(\Gamma_{n+1})\backslash\{0^{n+1},0^n1\}$ and let
$\mathcal{D}(v)$ be its decoration. Then $\mathcal{D}(v)$ is
mapped under $\pi_{n+1}$ bijectively to the decoration
$\mathcal{D}(\pi_{n+1}(v))\in\Gamma_n$ of $\pi_{n+1}(v)$.
\end{enumerate}
\end{prop}

\begin{proof}
\emph{Part 1.} Observe that, given $v\in V(\Gamma_n)$ and its
decoration $\mathcal{D}(v)\subset \Gamma_n$, applying SR to
$\mathcal{D}(v)$ yields a graph isomorphic to
$\mathcal{D}(0v)\subset \Gamma_{n+1}$. Indeed, if $v\neq 0^n$,
then the subgraph of $\Gamma_{n+1}$ resulting from applying the SR
to $\mathcal{D}(v)$ does not contain $0^{n+1}$. Since the vertex
$0v$ is a cut vertex, the statement follows by definition of a
decoration. On the other hand, if $v=0^n$, we can repeat the
previous argument replacing $0^n$ by $0^{n-1}1$. We
now prove the assertion by induction on $n$. The
decorations of the vertices of $\Gamma_1$ are both 1-decorations.
Let $v\in V(\Gamma_{n+1})$. If $v=1u$ for some $u$ of length $n$,
then it follows from the SR that $v$ has a loop, hence the
decoration of $v$ is a $1$-decoration. If $v=0u$, then consider
vertex $u$ in $\Gamma_n$ together with its decoration
$\mathcal{D}(u)$. By induction hypothesis, $\mathcal{D}(u)$ is a
$k$-decoration for some $1\leq k\leq n$ that is, $\mathcal{D}(u)$
is isomorphic to $\mathcal{D}(0^k)$. Hence, applying the SR to
$\mathcal{D}(u)$ yields a subgraph of $\Gamma_{n+1}$ isomorphic to
the subgraph of $\Gamma_{k+1}$ obtained by applying the SR to
$\mathcal{D}(0^k)$. On the other hand,
applying the SR to $\mathcal{D}(u)$ yields
$\mathcal{D}(0u)=\mathcal{D}(v)$. Hence,
$\mathcal{D}(v)$ is isomorphic to $\mathcal{D}(0^{k+1})$.\\
\indent \emph{Part 2.} Let us write $v=0^l1u$ for some $l\geq 0$.
Consider the decoration $\mathcal{D}(1u)\subset \Gamma_{n-l}$ of
$1u$. By the SR, $\mathcal{D}(1u)$ is isomorphic to
$\mathcal{D}(0)\subset \Gamma_1$. Starting from $\mathcal{D}(0)$
and applying the SR $k-1$ times yields $\mathcal{D}(0^k)$.
Similarly, starting from $\mathcal{D}(1u)$ and applying the SR $l$
times yields $\mathcal{D}(v)$. Suppose now that $\mathcal{D}(v)$
is a $k$-decoration. This means that $\mathcal{D}(v)$ is
isomorphic to $\mathcal{D}(0^k)$, thus by the observations we just
made, $l$ must be equal to $k-1$. Conversely, if $l=k-1$, then
necessarily $\mathcal{D}(v)$ is isomorphic to
$\mathcal{D}(0^k)$.\\ \indent \emph{Part 3.} Given $v\in
V(\Gamma_{n+1})\backslash\{0^{n+1},0^n1\}$, $\pi_{n+1}(v)\in
V(\Gamma_n)$ is obtained by erasing the last letter of $v$. On one
hand, if $w\in V(\Gamma_{n+1})$ is a cut vertex, then
$\pi_{n+1}(w)\in V(\Gamma_n)$ is a cut vertex too. On the other
hand, for all $w\in\mathcal{D}(v)$, $\pi_{n+1}(w)\neq 0^n$. Thus,
$\mathcal{D}(v)$ must be mapped (surjectively) on
$\mathcal{D}(\pi_{n+1}(v))$. But we have proven in {\it Part 2} that
$\mathcal{D}(v)$ is a $k$-decoration if and only if
$\mathcal{D}(\pi_{n+1}(v))$ is a $k$-decoration. The statement
follows.
\end{proof}

\begin{defi}
For every $n\geq 1$, we call \textbf{central cycle} of $\Gamma_n$ the
unique cycle containing both vertices $0^n$ and $0^{n-1}1$. The
decoration $\mathcal{D}(0^n)$ is the \textbf{left part} of
$\Gamma_n$, the decoration $\mathcal{D}(0^{n-1}1)$ is the
\textbf{right part} of $\Gamma_n$, the subgraph $\Gamma_n\setminus
\{\mathcal{D}(0^n)\cup \mathcal{D}(0^{n-1}1)\}$ is the \textbf{central part} of $\Gamma_n$.
\end{defi}

\indent It is convenient to encode the graph $\Gamma_n$ by a diagram, denoted $\overline{D}_n$, constructed as follows: consider a path containing $2^{\lceil\frac{n}{2}\rceil}$ edges. Its vertices are identified with vertices of the central cycle of $\Gamma_n$ so that
the left-half of $\overline{D}_n$ encodes the upper-half of the central cycle of $\Gamma_n$, whereas the right-half of $\overline{D}_n$ encodes the lower-half of the central cycle; and the vertices situated at the extremities of the path both encode the vertex $0^n$. Denote by $D_n$ the diagram without these two \emph{boundary} vertices. Label every vertex by the height of the decoration attached to it. Here are some examples:


\begin{center}
\begin{picture}(360,25)
\letstate A=(0,10) \letstate B=(22,10)\letstate C=(44,10) \letstate D=(154,10)
\letstate E=(176,10) \letstate F=(198,10)

\drawundirectededge(A,C){}\drawundirectededge(D,F){}

\drawvertex(A){$\circ$}\drawvertex(B){$\bullet$}\drawvertex(C){$\circ$}
\drawvertex(D){$\circ$}\drawvertex(E){$\bullet$}\drawvertex(F){$\circ$}

\put(-2,1){1}\put(20,1){1}\put(42,1){1}\put(152,1){2}\put(174,1){2}\put(196,1){2}
\put(-2,17){$\overline{D}_1$}\put(152,17){$\overline{D}_2$}
\end{picture}
\end{center}

\begin{center}
\begin{picture}(360,25)
\letstate A=(0,10) \letstate B=(22,10)\letstate C=(44,10) \letstate D=(66,10)
\letstate E=(88,10) \letstate F=(154,10)\letstate G=(176,10) \letstate H=(198,10)\letstate I=(220,10) \letstate L=(242,10)

\drawundirectededge(A,E){}\drawundirectededge(F,L){}

\drawvertex(A){$\circ$}\drawvertex(B){$\bullet$}\drawvertex(C){$\bullet$}
\drawvertex(D){$\bullet$}\drawvertex(E){$\circ$}\drawvertex(F){$\circ$}
\drawvertex(G){$\bullet$}\drawvertex(H){$\bullet$}\drawvertex(I){$\bullet$}\drawvertex(L){$\circ$}

\put(-2,1){3}\put(20,1){1}\put(42,1){3}\put(64,1){1}\put(86,1){3}\put(152,1){4}
\put(174,1){2}\put(196,1){4}\put(218,1){2}\put(240,1){4}

\put(-2,17){$\overline{D}_3$}\put(152,17){$\overline{D}_4$}
\end{picture}
\end{center}

\begin{center}
\begin{picture}(360,25)
\letstate A=(0,10) \letstate B=(22,10)\letstate C=(44,10) \letstate D=(66,10)
\letstate E=(88,10) \letstate F=(110,10)\letstate G=(132,10) \letstate H=(154,10)\letstate I=(176,10)

\drawundirectededge(A,I){}

\drawvertex(A){$\circ$}\drawvertex(B){$\bullet$}\drawvertex(C){$\bullet$}
\drawvertex(D){$\bullet$}\drawvertex(E){$\bullet$}\drawvertex(F){$\bullet$}
\drawvertex(G){$\bullet$}\drawvertex(H){$\bullet$}\drawvertex(I){$\circ$}

\put(-2,1){5}\put(20,1){1}\put(42,1){3}\put(64,1){1}\put(86,1){5}\put(108,1){1}
\put(130,1){3}\put(152,1){1}\put(174,1){5}

\put(-2,17){$\overline{D}_5$}
\end{picture}
\end{center}

\begin{center}
\begin{picture}(360,25)
\letstate A=(0,10) \letstate B=(22,10)\letstate C=(44,10) \letstate D=(66,10)
\letstate E=(88,10) \letstate F=(110,10)\letstate G=(132,10) \letstate H=(154,10)\letstate I=(176,10)

\drawundirectededge(A,I){}

\drawvertex(A){$\circ$}\drawvertex(B){$\bullet$}\drawvertex(C){$\bullet$}
\drawvertex(D){$\bullet$}\drawvertex(E){$\bullet$}\drawvertex(F){$\bullet$}
\drawvertex(G){$\bullet$}\drawvertex(H){$\bullet$}\drawvertex(I){$\circ$}

\put(-2,1){6}\put(20,1){2}\put(42,1){4}\put(64,1){2}\put(86,1){6}\put(108,1){2}
\put(130,1){4}\put(152,1){2}\put(174,1){6}

\put(-2,17){$\overline{D}_6$}
\end{picture}
\end{center}

\begin{center}
\begin{picture}(360,25)
\letstate A=(0,10) \letstate B=(22,10)\letstate C=(44,10) \letstate D=(66,10)
\letstate E=(88,10) \letstate F=(110,10)\letstate G=(132,10) \letstate H=(154,10)\letstate I=(176,10)
\letstate L=(198,10) \letstate M=(220,10)\letstate N=(242,10) \letstate O=(264,10)
\letstate P=(286,10) \letstate Q=(308,10)\letstate R=(330,10) \letstate S=(352,10)

\drawundirectededge(A,S){}

\drawvertex(A){$\circ$}\drawvertex(B){$\bullet$}\drawvertex(C){$\bullet$}
\drawvertex(D){$\bullet$}\drawvertex(E){$\bullet$}\drawvertex(F){$\bullet$}
\drawvertex(G){$\bullet$}\drawvertex(H){$\bullet$}\drawvertex(I){$\bullet$}
\drawvertex(L){$\bullet$}\drawvertex(M){$\bullet$}\drawvertex(N){$\bullet$}
\drawvertex(O){$\bullet$}\drawvertex(P){$\bullet$}\drawvertex(Q){$\bullet$}
\drawvertex(R){$\bullet$}\drawvertex(S){$\circ$}

\put(-2,1){7}\put(20,1){1}\put(42,1){3}\put(64,1){1}\put(86,1){5}\put(108,1){1}
\put(130,1){3}\put(152,1){1}\put(174,1){7}\put(196,1){1}\put(218,1){3}\put(240,1){1}\put(262,1){5}\put(284,1){1}\put(306,1){3}
\put(328,1){1}\put(350,1){7}

\put(-2,17){$\overline{D}_7$}
\end{picture}
\end{center}


\begin{prop} \label{PropRCR}
The diagram $\overline{D}_{n+1}$ encoding $\Gamma_{n+1}$ is obtained from $D_{n-1}$ by the following recursive rule (denoted by \textbf{RCR}):

\begin{center}
\begin{picture}(400,25)
\letstate A=(110,10) \letstate B=(200,10)\letstate C=(290,10)

\drawundirectededge(A,C){}

\drawvertex(A){$\circ$}\drawvertex(B){$\bullet$}\drawvertex(C){$\circ$}

\put(100,1){$n+1$}\put(280,1){$n+1$}\put(190,1){$n+1$}\put(148,15){$D_{n-1}$}\put(238,15){$D_{n-1}$}

\put(50,8){$\overline{D}_{n+1}:$}
\end{picture}
\end{center}
\end{prop}

\begin{proof}
We know by Proposition \ref{PropMultiple} that when we
project $\Gamma_{n+1}$ on $\Gamma_n$ under $\pi_{n+1}$, the
decoration of a vertex $v\in
V(\Gamma_{n+1})\backslash\{0^{n+1},0^{n}1\}$ is mapped onto a decoration
of the same height in $\Gamma_n$. There are only two
$(n+1)$-decorations in $\Gamma_{n+1}$, namely those of the vertices
$0^{n+1}$ and $0^n1$, and they are both mapped under $\pi_{n+1}$ to
$\mathcal{D}(0^n)^c$ in such a way that $\mathcal{D}(0^{n+1})$
superposes onto $\mathcal{D}(0^n1)$ after a rotation of
$-180^{\circ}$ around the origin and then a translation
(see Fig. 3).

\begin{center}
\psfrag{w0}{$w0$}\psfrag{u0}{$u0$}\psfrag{0n+1}{$0^{n+1}$}
\psfrag{D(0n+1)}{$\mathcal{D}(0^{n+1})$}\psfrag{pin+1}{$\pi_{n+1}$}\psfrag{pi}{$-180^{\circ}$}

\psfrag{D(0n)c}{$\mathcal{D}(0^n)^c$}\psfrag{0n}{$0^n$}\psfrag{w}{$w$}
\psfrag{u}{$u$}\psfrag{D(0n1)}{$\mathcal{D}(0^n1)$}\psfrag{0n1}{$0^n1$}

\psfrag{0n}{$0^n$}\psfrag{w1}{$w1$}\psfrag{u1}{$u1$}\psfrag{Fig.3}{\textbf{Fig. 3.}}
\includegraphics[width=0.6\textwidth]{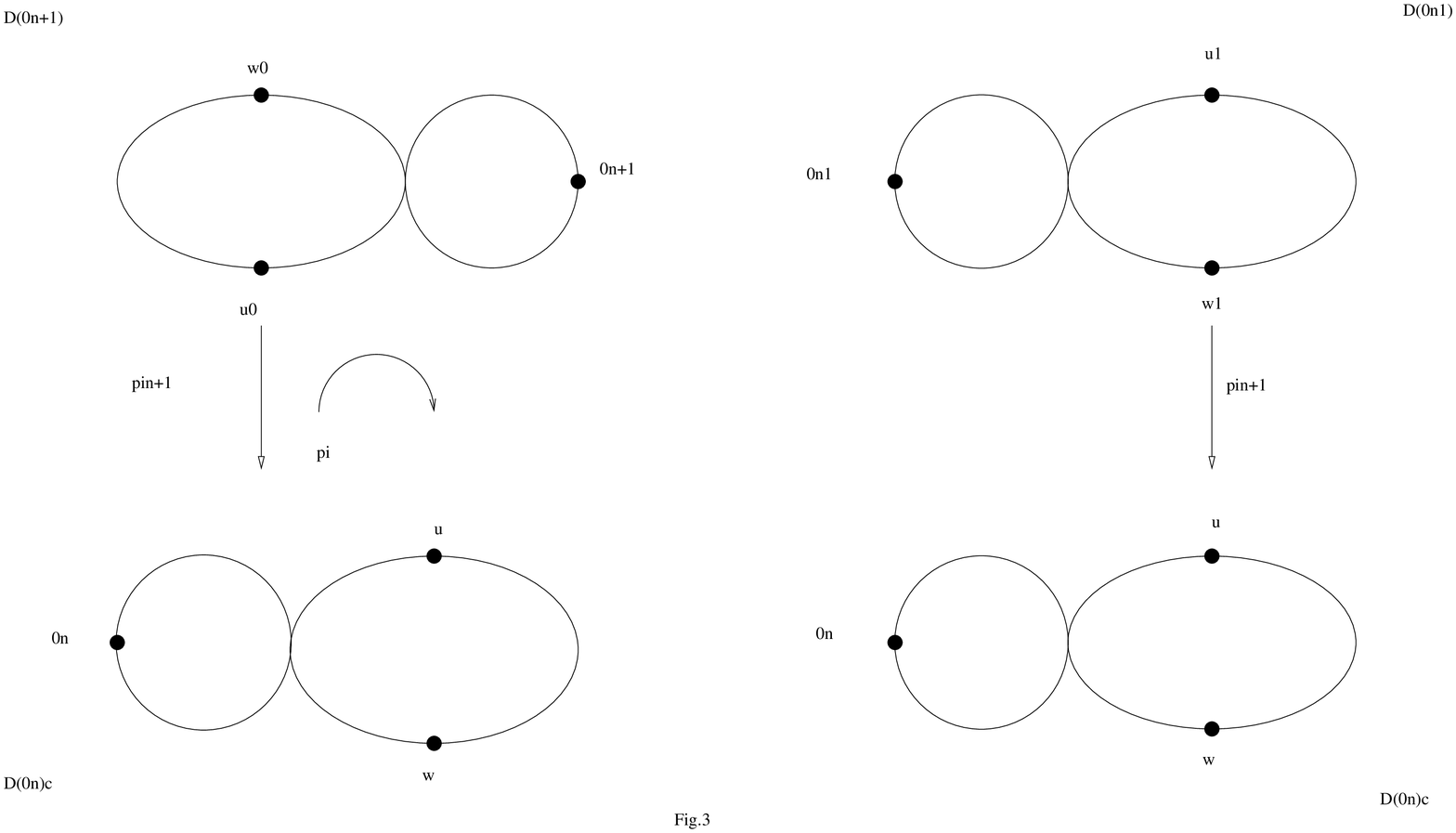}
\end{center}

Thus, the cycle in each of these two $(n+1)$-decorations which
contains respectively $0^{n+1}$ and $0^n1$ is mapped bijectively
to the central cycle of $\Gamma_n$. Moreover, pairs of opposite vertices of the
central cycle of $\Gamma_{n+1}$ are
identified under the projection $\pi_{n+1}$. Hence, the central
cycle of $\Gamma_{n+1}$ is mapped to the twice shorter cycle in
$\Gamma_n$ containing the vertex $0^n$ but not the vertex
$0^{n-1}1$ in such a way that $0^{n+1}$ is identified with $0^n1$
and the two halves of the central cycle are superposed (see Fig. 4).

\begin{center}
\psfrag{pin+1}{$\pi_{n+1}$}
\psfrag{w1}{$w1$}\psfrag{0n+1}{$0^{n+1}$}\psfrag{0n1}{$0^n1$}
\psfrag{w}{$w$}\psfrag{w0}{$w0$}\psfrag{0n}{$0^n$}\psfrag{Fig.4}{\textbf{Fig. 4.}}
\includegraphics[width=0.4\textwidth]{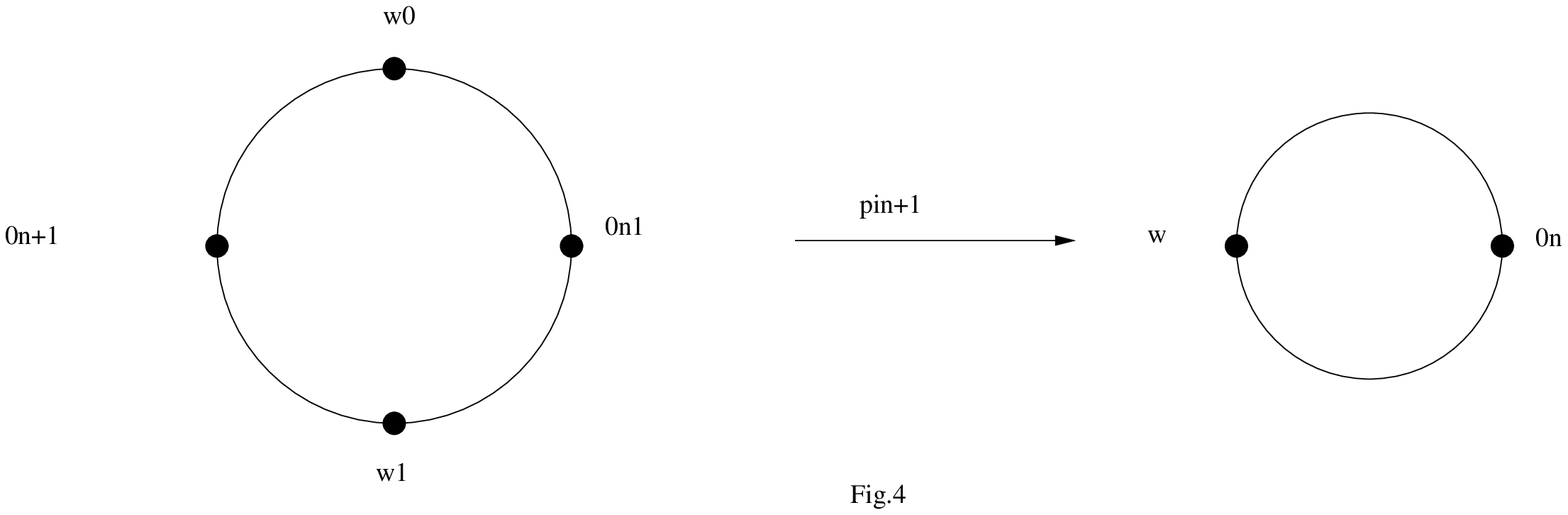}
\end{center}

\noindent Since the decoration $\mathcal{D}(0^n)$ in $\Gamma_n$ is bijectively mapped by $\pi_n$ to
$\mathcal{D}(0^{n-1})^c$ in $\Gamma_{n-1}$, the statement follows.
\end{proof}

\subsection{Converging sequences of Schreier graphs} \label{SubsecCovering}

Recall that for any $n\geq 1$, $\Gamma_n$ is a separable graph whose blocks are cycles.

\begin{defi}
A \textbf{cycle-path} of length $k$, $k\in\mathbb N\cup\{\infty\}$, in a graph $\Gamma$ is a sequence of $k$ distinct cycles in $\Gamma$ such that two consecutive cycles intersect at exactly one point.
\end{defi}

Given $w,w'\in V(\Gamma_n)$, there is a unique cycle-path $P_1,\dots,P_r$ in $\Gamma_n$ of minimal length, such that $w\in P_1$ and $w'\in P_r$. We will say that $P_1,\dots,P_r$ \emph{joins $w$ to $w'$}. Similarly, given $w\in V(\Gamma_n)$ and $P$ a cycle of $\Gamma_n$, there is a unique cycle-path of minimal length $P_1,\dots,P_r=P$ such that $w\in P_1$. We will say then that $P_1,\dots,P_r$  \emph{joins $w$ to $P$}.

\begin{nota} \label{notationCPCL}\rm
Given an infinite word $\xi$ and its prefix $\xi_n$ of length $n$, write $\mathcal{CP}_{\xi_n}=\{P_i^n\}_{i=1}^t$ for the unique cycle-path joining $\xi_n$ to the central cycle of $\Gamma_n$. Then we have $t=|\mathcal{CP}_{\xi_n}|$, the length of the cycle-path, and we denote by $\mathcal{LP}_{\xi_n}=\{k_i^n\}_{i=1}^t$ the sequence of lengths of the cycles in the cycle-path.
\end{nota}

In order to identify the limit $\lim_{n\to\infty}(\Gamma_n,\xi_n)$ in the space of rooted graphs, it is important to keep track of the behaviour of the cycle-path $\mathcal{CP}_{\xi_i}$ as $i$ changes from $n$ to $n+1$. As we will see, there are three possibilities for the length of $\mathcal{CP}_{\xi_i}$: it can either increase/decrease by one, or stay stable. More precisely, we have

\begin{prop} \label{PropInverseCovering}
Let $\xi\in\{0,1\}^\omega$. For any $n\geq 1$, consider the cycle-path $\mathcal{CP}_{\xi_n}$ (respectively $\mathcal{CP}_{\xi_{n+1}}$) in $\Gamma_n$ (respectively in $\Gamma_{n+1}$). Then, the following situations may occur:

\begin{enumerate}
\item  "\textbf{E}xpansion", that is, $|\mathcal{CP}_{\xi_{n+1}}|=|\mathcal{CP}_{\xi_n}|+1$; this occurs if and only if $\xi_n\in\mathcal{D}(0^n)^c$.
\item  "\textbf{C}ontraction", that is, $|\mathcal{CP}_{\xi_{n+1}}|=|\mathcal{CP}_{\xi_n}|-1$; this occurs if and only if $\xi_n\in\mathcal{D}(0^n)$.
\item  "\textbf{S}tability", that is, $|\mathcal{CP}_{\xi_{n+1}}|=|\mathcal{CP}_{\xi_n}|$; this occurs if and only if $\xi_{n+1}$ is either $0^{n+1}$ or
$0^{n}1$.
\end{enumerate}
\end{prop}

Proposition \ref{PropInverseCovering} motivates the following definition:

\begin{defi} \label{Defindexofstab}
For every word $\xi\in \{0,1\}^{\omega}$, we define its index of stability as

\begin{equation}\label{Eqindexofstab}
z_{\xi} = \sup\{n\geq 1 \ : \ |\mathcal{CP}_{\xi_n}|=1\}.
\end{equation}
Moreover, define the \textbf{SEC-sequence} of $\xi$ to be the infinite word in $\{S,E,C\}$ beginning with $S^{z_{\xi}}$ and such that the $(n+1)$-st letter (for $n \geq z_{\xi}$) of this word is either $E$ or $C$ depending on whether in the passage from $\Gamma_n$ to $\Gamma_{n+1}$ we observe an expansion or a contraction between the cycle-paths $\mathcal{CP}_{\xi_n}$ and $\mathcal{CP}_{\xi_{n+1}}$.
\end{defi}

\begin{rem}\label{RemonSEC}\rm
We can consider the SEC-sequence of any finite prefix $\xi_n$ of $\xi$ by taking the restriction up to the $n$-th term of the SEC-sequence of $\xi$. The occurrence of a block of $t$ expansions beginning at position $n$ in the SEC-sequence of a word $\xi$ encodes the fact that $|\mathcal{CP}_{\xi_{n+t-1}}|-|\mathcal{CP}_{\xi_{n-1}}|=t$ whereas the occurrence of a block of $m/2$ expansions-contractions $EC$ beginning at position $n$ means that
$|\mathcal{CP}_{\xi_{n+m-1}}|-|\mathcal{CP}_{\xi_{n-1}}|=0$. Moreover, observe (see the proof of Proposition \ref{PropInverseCovering} below) that the SEC-sequence associated with an infinite word $\xi$ cannot contain two consecutive C's. 
\end{rem}

We will also need to control the lengths of individual cycles in
the cycle-path $\mathcal{CP}_{\xi_i}$ as $i$ changes from $n$ to
$n+1$. Namely, we have

\begin{lem} \label{lemLP}
Let $\xi\in\{0,1\}^\omega$. For any $n\geq 1$, consider the
cycle-path $\mathcal{CP}_{\xi_n}$ (respectively
$\mathcal{CP}_{\xi_{n+1}}$) in $\Gamma_n$ (respectively in
$\Gamma_{n+1}$). Then, the sequence $\mathcal{LP}_{\xi_{n+1}}$ is
obtained from $\mathcal{LP}_{\xi_{n}}$ as:

\begin{enumerate}
  \item if $\xi_n\in\mathcal{D}(0^n)^c$, then $k_i^{n+1}=k_i^{n}$ for every $i=1,\ldots, t$ and
\begin{eqnarray*}
  \begin{cases}
    k_{t+1}^{n+1}=k_{t}^n & \text{if $n$ is odd}, \\
    k_{t+1}^{n+1}=2k_{t}^n & \text{if $n$ is even};
  \end{cases}
\end{eqnarray*}
   \item if $\xi_n\in\mathcal{D}(0^n)$, then $k_i^{n+1}=k_i^n$ for every $i=1,\ldots, t-2$ and
\begin{eqnarray*}
  \begin{cases}
    k_{t-1}^{n+1}=k_{t}^{n} & \text{if $n$ is odd}, \\
    k_{t-1}^{n+1}=2k_{t}^{n} & \text{if $n$ is even}.
  \end{cases}
\end{eqnarray*}
\end{enumerate}
\end{lem}

\begin{proof}[Proof of Proposition \ref{PropInverseCovering} and Lemma \ref{lemLP}]
Let us look at how the graph $\Gamma_{n+1}$ is obtained from
$\Gamma_n$ under $\pi_{n+1}^{-1}$. Recall that all $\Gamma_n$'s are embedded in the plane.
Consider the decoration
$\mathcal{D}(0^n)$ in $\Gamma_n$. Take two copies of
$\mathcal{D}(0^n)$: the first one is rotated by $-90^{\circ}$
around the origin, whereas the second one is rotated by
$+90^{\circ}$ around the origin. Now split the only
vertex of degree $2$ (which is $0^n$) into two vertices $v_1,v_1'$
for the first copy and $v_2,v_2'$ for the second copy. Then glue
$v_1$ with $v_2$ into a vertex $v$ and $v_1'$ with $v_2'$ into a
vertex $v'$ (see Fig. 5).

\begin{center}
\psfrag{w0}{$w0$}\psfrag{u0}{$u0$}\psfrag{w1}{$w1$}\psfrag{u1}{$u1$}\psfrag{w}{$w$}
\psfrag{u}{$u$}

\psfrag{v=0n+1}{$v=0^{n+1}$}\psfrag{0n}{$0^n$}\psfrag{v'=0n1}{$v'=0^n1$}

\psfrag{v1}{$v_1$}\psfrag{v2}{$v_2$}\psfrag{v1'}{$v_1'$}\psfrag{v2'}{$v_2'$}

\psfrag{D(0n)}{$\mathcal{D}(0^n)$}\psfrag{pin+1-1}{$\pi_{n+1}^{-1}$}\psfrag{Fig.5}{\textbf{Fig. 5.}}

\psfrag{-pi/2}{$-90^{\circ}$}\psfrag{pi/2}{$+90^{\circ}$}
\includegraphics[width=0.8\textwidth]{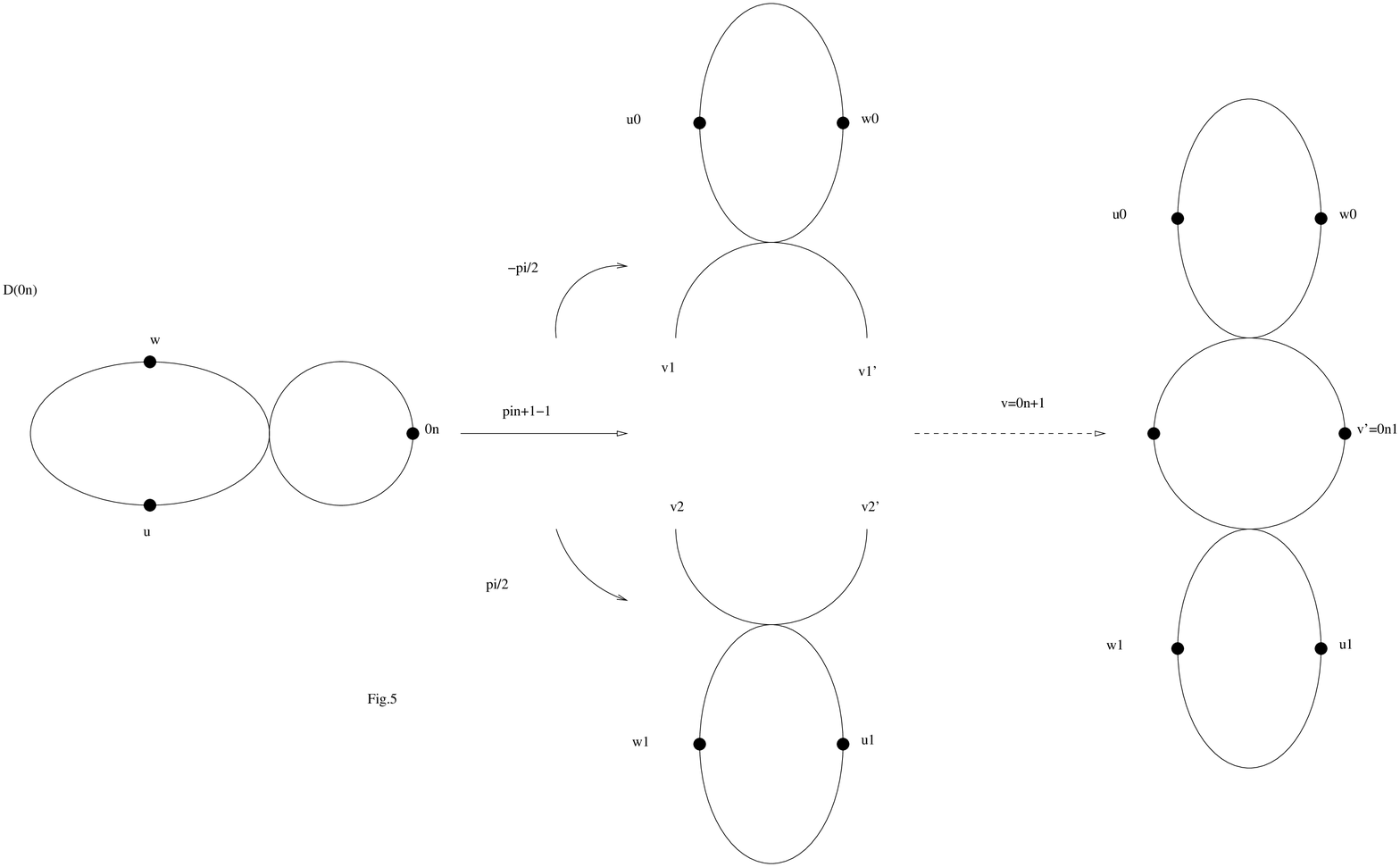}
\end{center}

Finally, after a rotation by $+180^{\circ}$ around the
origin, attach to $v$ a copy of $\mathcal{D}(0^n)^c$ by gluing $v$
with the only vertex of $\mathcal{D}(0^n)^c$ of degree 2 and
attach to $v'$ a copy of $\mathcal{D}(0^n)^c$ (see Fig. 6).

\begin{center}
\psfrag{w0}{$w0$}\psfrag{u0}{$u0$}\psfrag{w1}{$w1$}\psfrag{u1}{$u1$}\psfrag{w}{$w$}
\psfrag{u}{$u$}

\psfrag{0n-11}{$0^{n-1}1$}\psfrag{0n}{$0^n$}\psfrag{0n+1}{$0^{n+1}$}\psfrag{0n1}{$0^n1$}

\psfrag{D(0n)c}{$\mathcal{D}(0^n)^c$}\psfrag{pin+1-1}{$\pi_{n+1}^{-1}$}

\psfrag{pi}{$+180^{\circ}$}\psfrag{Fig.6}{\textbf{Fig. 6.}}

\psfrag{D(0n+1)}{$\mathcal{D}(0^{n+1})$}\psfrag{D(0n1)}{$\mathcal{D}(0^n1)$}

\includegraphics[width=1.0\textwidth]{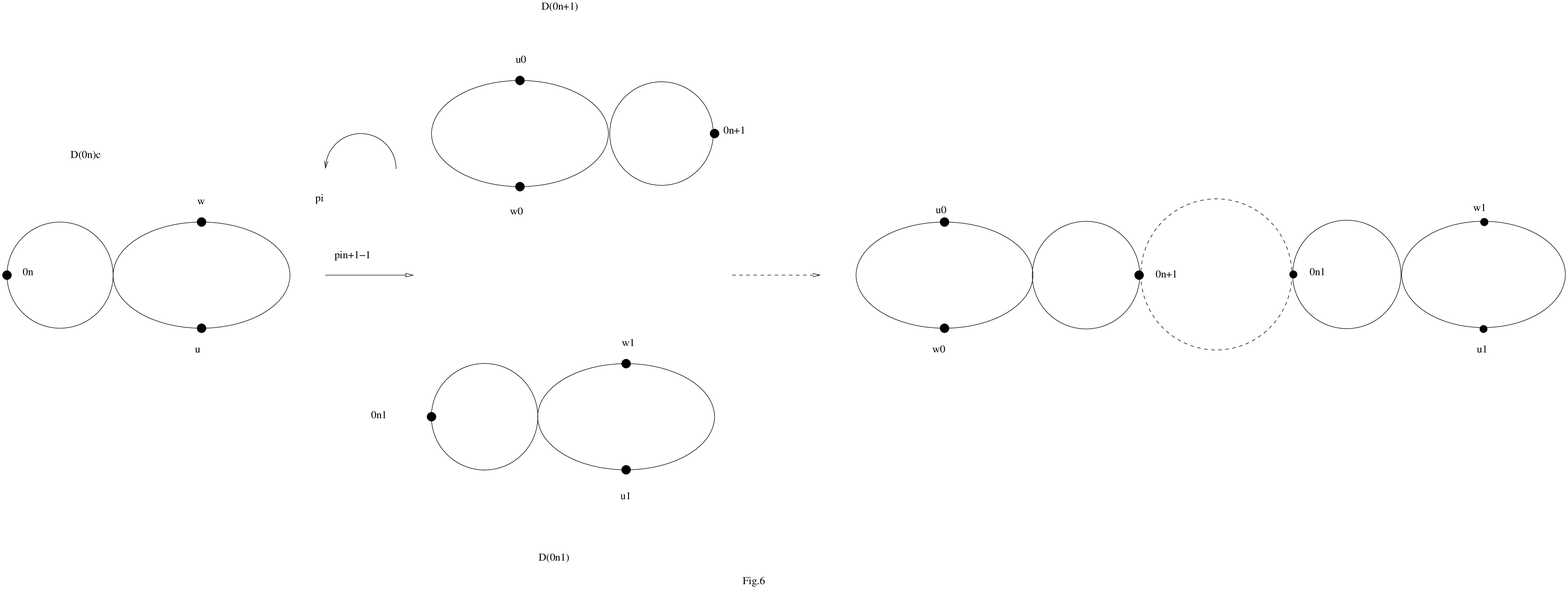}
\end{center}

\indent Accordingly, vertices of $\Gamma_{n+1}$ are labeled as follows:

\begin{itemize}
\item add a $0$ at the right of each vertex of the first copy of $\mathcal{D}(0^n)$ except for vertices $v_1$ and $v_1'$,
\item add a $1$ at the right of each vertex of the second copy of $\mathcal{D}(0^n)$ except for vertices $v_2$ and $v_2'$,
\item put $v_1\equiv v_2\equiv v=0^{n+1}$ and $v_1'\equiv v_2'\equiv
v'=0^n1$,
\item add a $0$ at the right of each vertex of the left copy of $\mathcal{D}(0^n)^c$,
\item add a $1$ at the right of each vertex of the right copy of $\mathcal{D}(0^n)^c$.
\end{itemize}

Thus, if $\xi_n\in\mathcal{D}(0^n)^c$, then
$|\mathcal{CP}_{\xi_{n+1}}|=|\mathcal{CP}_{\xi_n}|+1$ (see Fig.
6). On the other hand, if
$\xi_n\in\mathcal{D}(0^n)\backslash\{0^n\}$, then
$\xi_{n+1}\in\mathcal{D}(0^{n+1})^c$ and
$|\mathcal{CP}_{\xi_{n+1}}|=|\mathcal{CP}_{\xi_n}|-1$ (see Fig.
5). Finally, if $\xi_{n+1}=0^{n+1}$ or $\xi_{n+1}=0^{n}1$, then
obviously, $|\mathcal{CP}_{\xi_{n+1}}|=|\mathcal{CP}_{\xi_n}|=1$.
Proposition \ref{PropInverseCovering} follows.\\
\indent Lemma \ref{lemLP} follows from the above description of $\pi_{n+1}^{-1}$ (see Fig. 5 and Fig. 6).
\end{proof}

\section{Infinite Schreier graphs}

This section and the next one contain the main results of the paper.
Our aim is to give a classification of the infinite Schreier graphs $(\Gamma_{\xi}, \xi)$, $\xi\in \partial T$, of the Basilica group acting on the boundary of the binary tree. \\
\indent In this section, we show that a graph $\Gamma_\xi$ has one, two or four ends (Theorem \ref{totaltheorem}), and classify the graphs with four and two ends
up to isomorphism (Theorems \ref{4isomorfismo}, \ref{xixibarra}, \ref{2isomorfismo}). It happens however that almost every infinite Schreier graph is one-ended (see Subsection \ref{SubsectionRandomLimit} for the discussion), and the one-ended Schreier graphs are classified up to isomorphism in Subsection \ref{section1end} (Theorem \ref{ThmBijection}).\\
\indent Given an infinite graph $\Gamma=(V,E)$, a \textbf{ray} is an infinite sequence of distinct vertices of $\Gamma$ such that any two consecutive vertices of this sequence are adjacent in $\Gamma$. Consider an equivalence relation on the set of rays in $\Gamma$: two rays $\mathcal{R}$ and $\mathcal{R}'$ are equivalent if for any finite set $S\subset V$ both $\mathcal{R}$ and $\mathcal{R}'$ have a tail in the same component of $\Gamma\backslash S$. If two rays are equivalent, and only then, they can be linked by infinitely many disjoint paths. An \textbf{end} is an equivalence class of rays. Note that every infinite, locally finite graph must have at least one end.

\begin{teo}\label{totaltheorem}
Set $E_i = \{\xi \in \{0,1\}^{\omega} \ | \ \textrm{the infinite Schreier graph}\ \Gamma_{\xi} \
\mbox{has }i \mbox{ ends}\}$. Then
\begin{enumerate}
\item  $E_4 = \{w0^{\omega}, w(01)^{\omega}\ |\ \ w\in
X^{\ast}\}$;
\item $E_1 = \{\alpha_1\beta_1\alpha_2\beta_2\ldots, \
\alpha_i,\beta_j\in \{0,1\}\ | \  \textrm{$\{\alpha_i\}_{i\geq 1}$
and $\{\beta_j\}_{j\geq 1}$ both contain infinitely many }1's\}$;
\item $E_2 = \{0,1\}^{\omega} \setminus \left\{E_1\sqcup
E_4\right\}$.
\end{enumerate}
\end{teo}

\indent This theorem will be proven in the next subsection. \\
\indent Observe that, by Proposition \ref{cofinality},  $E_4$ consists of exactly one orbit,
whereas $E_1$ and $E_2$ consist of infinitely many orbits. More precisely, our classification results proven below (Theorems \ref{totaltheorem}, \ref{4isomorfismo}, \ref{xixibarra}, \ref{ThmBijection}) imply the following:

\begin{cor} \label{ThmClassif}

\begin{enumerate}
\item There exists only one class of isomorphism of 4-ended (unrooted) infinite Schreier graphs. It contains a single orbit.
\item There exist uncountably many classes of isomorphism of 2-ended (unrooted) infinite Schreier graphs. Each of these classes contains exactly two orbits.
\item There exist uncountably many classes of isomorphism of 1-ended (unrooted) infinite Schreier graphs. The isomorphism class of $\Gamma_{1^\omega}$ is a single orbit, and every other class contains uncountably many orbits.
\end{enumerate}
\end{cor}

\indent See also Propositions \ref{unifmeasure} and \ref{classemisuranulla} for a measurable classification of infinite Schreier graphs $\{\Gamma_\xi\ |\ \xi\in\partial T\}$.\\
\indent We end this subsection with an infinite analogue of Propositions \ref{BasicProp} and \ref{PropMultiple} about the structure of the Basilica Schreier graphs:

\begin{prop} \label{PropSeparability}
\begin{enumerate}
\item For any $\xi\in\{0,1\}^\omega$, the infinite orbital Schreier graph $\Gamma_{\xi}$ is separable and its blocks are either cycles or single edges.
\item For any $\eta\in V(\Gamma_{\xi})$ ($\eta\neq 0^\omega$), removing $\eta$ splits $\Gamma_\xi$ into several components among which one, denoted by $U$, is finite. Then, the subgraph induced by $V(U)\cup\{\eta\}$ is isomorphic to a $k$-decoration for some $k\geq 1$.
\item The decoration of $\eta$ is a $k$-decoration if and only if while reading $\eta$ from the left, we first encounter $1$ in the $k$-th position.
\end{enumerate}
\end{prop}

\begin{proof}
\emph{Part 1.} Since the sequence of finite rooted graphs $(\Gamma_n,\xi_n)$ converges to $(\Gamma_{\xi},\xi)$, we have that for all $r\geq 1$, there exists $N\geq 1$ such that for all $n\geq N$, $B_{\Gamma_n}(\xi_n,r)\simeq B_{\Gamma_{\xi}}(\xi,r)$. But the subgraph $B_{\Gamma_n}(\xi_n,r)$ of $\Gamma_n$ is separable and its blocks are either cycles or single edges. Since $\lim_{r\to\infty}B_{\Gamma_{\xi}}(\xi,r)=(\Gamma_{\xi},\xi)$, the statement follows.\\
\indent \emph{Parts 2. and 3.} Let $\eta\in V(\Gamma_{\xi})$ ($\eta\neq 0^\omega$) and let $\eta_k$ be the shortest prefix of $\eta$ ending by $1$. By  Proposition \ref{PropMultiple}, the decoration of $\eta_n$ in $\Gamma_n$ is a $k$-decoration for every $n\geq k+1$. If $d\equiv d_{\Gamma_\xi}(\eta,\xi)$ denotes the usual graph distance in $\Gamma_\xi$ between $\xi$ and $\eta$, then there exists a radius $r_{(d,k)}$ and a $N(r_{(d,k)})\geq 1$ such that for all $n\geq N(r_{(d,k)})$, the ball $B_{\Gamma_n}(\xi_n,r_{(d,k)})$ contains $\eta_n$ together with its decoration and $B_{\Gamma_n}(\xi_n,r_{(d,k)})\simeq B_{\Gamma_{\xi}}(\xi,r_{(d,k)})$.
\end{proof}

\subsection{Limit graph with four ends and Proof of Theorem \ref{totaltheorem}}\label{section4ends}

\indent In this subsection, we determine for which infinite words $\xi\in\{0,1\}^\omega$ the corresponding orbital infinite Schreier
graph $\Gamma_{\xi}$ has four ends, and we study the shape of these graphs.\\
\indent Let $\xi\in\{0,1\}^{\omega}$. By Proposition \ref{BasicProp}, for any $n>1$, $\xi_n$ belongs to exactly two cycles of $\Gamma_n$ denoted by $C(\xi_n)$ and $C'(\xi_n)$; $C(\xi_n)$ is contained in $\mathcal{D}(\xi_n)$ whereas $C'(\xi_n)$ is contained in $\mathcal{D}(\xi_n)^c$. Finally, write $l_n(\xi):=|C(\xi_n)|$ and $l_n'(\xi):=|C'(\xi_n)|$.
We have the following:

\begin{lem}\label{lemmazeroomega}
Given $\xi\in\{0,1\}^{\omega}$, the sequence $\{l_n(\xi)\}_{n\geq 1}$ diverges as $n\to \infty$ if and only if $\xi=0^\omega$.
\end{lem}

\begin{proof}
It follows from SR that $l_n(0^\omega)=2^{\lceil\frac{n-1}{2}\rceil}$ and $l_n'(0^\omega)=2^{\lceil\frac{n}{2}\rceil}$ for every $n\geq 1$. Conversely, consider $\eta\in\{0,1\}^{\omega}$ such that $\eta\neq 0^\omega$. Then $\eta=0^k1\eta'$ for some $k\geq 0$ and it follows from Proposition \ref{PropMultiple}, Part \emph{2.}, that a $(k+1)$-decoration is
attached to $\eta_n$ for every $n\geq k+1$. Hence, $l_n(\eta)$ is constant for each $n\geq k+1$.
\end{proof}

\indent The next lemma describes the decorations attached to $C(0^n)$ and $C'(0^n)$.

\begin{lem} \label{LemDistDec}
For $n\geq 2$, consider the graph $\Gamma_n$. Then, for $0\leq k\leq \lceil\frac{n-2}{2}\rceil$, to every vertex situated on $C'(0^n)$ at distance $2^k+l\cdot 2^{k+1}$ for some $l\geq 0$ from $0^n$ is attached a $(2k+1)$-decoration if $n$ is odd and a $(2k+2)$-decoration if $n$ is even. On the other hand, for $0\leq k\leq \lfloor\frac{n-2}{2}\rfloor$, to every vertex situated on $C(0^n)$ at distance $2^k+l\cdot 2^{k+1}$ for some $l\geq 0$ from $0^n$ is attached a $(2k+2)$-decoration if $n$ is odd and a $(2k+1)$-decoration if $n$ is even.
\end{lem}

\begin{proof}
Observe that the cycle $C'(0^n)$ (which is the central cycle of $\Gamma_n$) and the decorations attached to the vertices of $C'(0^n)$ are encoded by the diagram $D_n$ (see Subsection \ref{finiteSch}). Similarly, the cycle $C(0^n)$ and the decorations attached to the vertices of $C(0^n)$ are encoded by the diagram $D_{n-1}$. Induction on $n$ and the RCR yield the statement.
\end{proof}

\begin{proof}[Proof of Theorem \ref{totaltheorem}, Part 1]
Consider the vertex $0^n$ in $\Gamma_n$ together with its neighbours. Consider the following four disjoint paths: each of them starts at a different neighbour of $0^n$. The two first are included in $C(0^n)$ and they are of length $l_n(0^\omega)/2-1$ whereas the two others are included in $C'(0^n)$ and they are of length $l'_n(0^\omega)/2-1$.
Since $l_n(0^\omega)=2^{\lceil\frac{n-1}{2}\rceil}$ and $l_n'(0^\omega)=2^{\lceil\frac{n}{2}\rceil}$, their respective lengths diverge as $n\to\infty$. Since $\lim_{n\to\infty}(\Gamma_n,0^n)=(\Gamma_{0^\omega},0^\omega)$, the corresponding four rays are disjoint in $(\Gamma_{0^\omega},0^\omega)$ so that they are not equivalent. Hence, the graph
$(\Gamma_{0^{\omega}}, 0^{\omega})$ has at least four ends.\\
\indent Consider now $\xi\in V(\Gamma_{0^\omega})$, $\xi\neq 0^\omega$. Then $\xi=0^k1\xi'$ for some $k\geq 0$; by Proposition \ref{PropMultiple}, a $(k+1)$-decoration is attached to $\xi_n$ for every $n\geq k+1$. Since every such
decoration is finite, any ray must have infinitely many vertices in common with one of the four previous rays. We thus
conclude that the graph $(\Gamma_{0^{\omega}}, 0^{\omega})$ has exactly four ends.\\
\indent Finally, Lemma \ref{lemmazeroomega} implies that the set of infinite words $\xi$ such that the orbital Schreier graph $\Gamma_\xi$ has four ends is exactly constituted by the orbit of $0^{\omega}$. Then the result follows from Proposition \ref{cofinality}.
\end{proof}

\indent We describe now a 4-regular infinite separable graph $\Gamma_{(4)}$ whose blocks are either cycles or single
edges. The number of ends of $\Gamma_{(4)}$ is 4. It will turn out that $\Gamma_{(4)}$ is isomorphic (as an unrooted graph) to $\Gamma_{\xi}$ for each $\xi \in E_4$. $\Gamma_{(4)}$ is constructed as follows: consider two copies, $\mathcal{R}_1$ and $\mathcal{R}_2$, of the double ray whose vertices are naturally identified with the integers. Let these two double rays intersect at vertex $0$. For every $k\geq 0$, we define the subset of
$\mathbb{Z}$
$$
A_k:=\left\{n\in\mathbb{Z}|n\equiv 2^k\mod 2^{k+1}\right\}.
$$
\noindent For every $k\geq 0$, attach to each vertex of $A_k$ in $\mathcal{R}_1$ (respectively in $\mathcal{R}_2$) a $(2k+1)$-decoration (respectively a $(2k+2)$-decoration) by its unique vertex of degree $2$.

\unitlength=0,3mm
\begin{center}
\begin{picture}(400,370)
\setloopdiam{4}
\letvertex A=(200,170)\letvertex B=(235,170)\letvertex C=(270,170)\letvertex D=(305,170)\letvertex E=(340,170)\letvertex F=(375,170)\letvertex G=(400,170)
\letvertex H=(200,135)\letvertex I=(200,100)\letvertex L=(200,65)\letvertex M=(200,40)\letvertex N=(165,170)\letvertex O=(130,170)
\letvertex P=(95,170)\letvertex Q=(70,170)\letvertex R=(200,205)\letvertex S=(200,240)\letvertex T=(200,275)\letvertex U=(200,310)\letvertex V=(200,345)
\letvertex Z=(200,370)\letvertex J=(145,185)\letvertex K=(130,200)\letvertex X=(130,220)\letvertex W=(95,190)\letvertex Y=(130,240)
\letvertex a=(220,135)\letvertex b=(215,115)\letvertex c=(215,85)\letvertex d=(230,100)\letvertex e=(250,100)\letvertex f=(270,100)\letvertex g=(220,65)
\letvertex h=(270,150)\letvertex i=(270,130)\letvertex l=(305,155)\letvertex m=(325,155)\letvertex n=(355,155)\letvertex o=(375,155)\letvertex p=(340,140)
\letvertex q=(325,125)\letvertex r=(355,125)\letvertex s=(340,110)\letvertex t=(340,90)\letvertex u=(340,70)\letvertex v=(165,190)
\letvertex z=(115,185)\letvertex j=(180,240)\letvertex k=(160,240)\letvertex x=(185,345)\letvertex w=(185,295)\letvertex y=(185,325)
\letvertex aa=(185,275)\letvertex bb=(170,310)\letvertex cc=(155,325)\letvertex dd=(155,295)\letvertex ee=(140,310)
\letvertex ff=(120,310)\letvertex gg=(100,310)
\put(155,20){\textbf{Fig. 7.} A finite part of $\Gamma_{(4)}$.}

\drawvertex(A){$\bullet$}\drawvertex(B){$\bullet$}\drawvertex(C){$\bullet$}\drawvertex(D){$\bullet$}\drawvertex(E){$\bullet$}\drawvertex(F){$\bullet$}
\drawvertex(H){$\bullet$}\drawvertex(I){$\bullet$}\drawvertex(L){$\bullet$}\drawvertex(N){$\bullet$}
\drawvertex(O){$\bullet$}\drawvertex(P){$\bullet$}\drawvertex(R){$\bullet$}\drawvertex(S){$\bullet$}\drawvertex(T){$\bullet$}
\drawvertex(U){$\bullet$}\drawvertex(V){$\bullet$}\drawvertex(X){$\bullet$}\drawvertex(Y){$\bullet$}\drawvertex(W){$\bullet$}\drawvertex(J){$\bullet$}
\drawvertex(K){$\bullet$}\drawvertex(a){$\bullet$}\drawvertex(b){$\bullet$}\drawvertex(c){$\bullet$}\drawvertex(d){$\bullet$}\drawvertex(e){$\bullet$}
\drawvertex(f){$\bullet$}\drawvertex(g){$\bullet$}\drawvertex(h){$\bullet$}\drawvertex(i){$\bullet$}\drawvertex(l){$\bullet$}\drawvertex(m){$\bullet$}
\drawvertex(n){$\bullet$}\drawvertex(o){$\bullet$}\drawvertex(p){$\bullet$}\drawvertex(q){$\bullet$}\drawvertex(r){$\bullet$}\drawvertex(s){$\bullet$}
\drawvertex(t){$\bullet$}\drawvertex(u){$\bullet$}\drawvertex(v){$\bullet$}\drawvertex(z){$\bullet$}\drawvertex(x){$\bullet$}\drawvertex(y){$\bullet$}
\drawvertex(w){$\bullet$}\drawvertex(j){$\bullet$}\drawvertex(k){$\bullet$}\drawvertex(aa){$\bullet$}\drawvertex(bb){$\bullet$}\drawvertex(cc){$\bullet$}
\drawvertex(dd){$\bullet$}\drawvertex(ee){$\bullet$}\drawvertex(ff){$\bullet$}\drawvertex(gg){$\bullet$}

\drawundirectededge(A,B){}\drawundirectededge(B,C){}\drawundirectededge(C,D){}\drawundirectededge(D,E){}\drawundirectededge(E,F){}
\drawundirectededge(F,G){}\drawundirectededge(A,H){}\drawundirectededge(H,I){}\drawundirectededge(I,L){}\drawundirectededge(L,M){}
\drawundirectededge(A,N){}\drawundirectededge(N,O){}\drawundirectededge(O,P){}\drawundirectededge(P,Q){}\drawundirectededge(A,R){}
\drawundirectededge(R,S){}\drawundirectededge(S,T){}\drawundirectededge(T,U){}\drawundirectededge(U,V){}\drawundirectededge(V,Z){}
\drawundirectededge(I,b){}\drawundirectededge(b,d){}\drawundirectededge(d,c){}\drawundirectededge(c,I){}\drawundirectededge(E,n){}
\drawundirectededge(n,p){}\drawundirectededge(p,m){}\drawundirectededge(m,E){}\drawundirectededge(p,r){}\drawundirectededge(r,s){}
\drawundirectededge(s,q){}\drawundirectededge(q,p){}\drawundirectededge(O,z){}\drawundirectededge(z,K){}\drawundirectededge(K,J){}
\drawundirectededge(J,O){}\drawundirectededge(U,w){}\drawundirectededge(w,bb){}\drawundirectededge(bb,y){}\drawundirectededge(y,U){}
\drawundirectededge(bb,dd){}\drawundirectededge(dd,ee){}\drawundirectededge(ee,cc){}\drawundirectededge(cc,bb){}

\drawundirectedcurvededge(a,H){}\drawundirectedcurvededge(H,a){}\drawundirectedcurvededge(d,e){}\drawundirectedcurvededge(e,d){}
\drawundirectedcurvededge(e,f){}\drawundirectedcurvededge(f,e){}\drawundirectedcurvededge(L,g){}\drawundirectedcurvededge(g,L){}
\drawundirectedcurvededge(C,h){}\drawundirectedcurvededge(h,C){}\drawundirectedcurvededge(h,i){}\drawundirectedcurvededge(i,h){}
\drawundirectedcurvededge(l,m){}\drawundirectedcurvededge(m,l){}\drawundirectedcurvededge(n,o){}\drawundirectedcurvededge(o,n){}
\drawundirectedcurvededge(s,t){}\drawundirectedcurvededge(t,s){}\drawundirectedcurvededge(t,u){}\drawundirectedcurvededge(u,t){}
\drawundirectedcurvededge(N,v){}\drawundirectedcurvededge(v,N){}\drawundirectedcurvededge(K,X){}\drawundirectedcurvededge(X,K){}
\drawundirectedcurvededge(X,Y){}\drawundirectedcurvededge(Y,X){}\drawundirectedcurvededge(P,W){}\drawundirectedcurvededge(W,P){}
\drawundirectedcurvededge(S,j){}\drawundirectedcurvededge(j,S){}\drawundirectedcurvededge(j,k){}\drawundirectedcurvededge(k,j){}
\drawundirectedcurvededge(gg,ff){}\drawundirectedcurvededge(ff,gg){}\drawundirectedcurvededge(ff,ee){}\drawundirectedcurvededge(ee,ff){}
\drawundirectedcurvededge(x,y){}\drawundirectedcurvededge(y,x){}\drawundirectedcurvededge(w,aa){}\drawundirectedcurvededge(aa,w){}

\drawundirectedloop[r](a){}\drawundirectedloop(b){}\drawundirectedloop[b](c){}\drawundirectedloop[r](f){}\drawundirectedloop[r](g){}\drawundirectedloop[b](B){}\drawundirectedloop[b](i){}
\drawundirectedloop[b](D){}\drawundirectedloop[l](l){}\drawundirectedloop[r](o){}\drawundirectedloop[l](q){}\drawundirectedloop[r](r){}\drawundirectedloop[b](u){}\drawundirectedloop[b](F){}
\drawundirectedloop[l](V){}\drawundirectedloop[l](gg){}\drawundirectedloop(cc){}\drawundirectedloop[b](dd){}\drawundirectedloop(x){}\drawundirectedloop[b](aa){}\drawundirectedloop[l](T){}
\drawundirectedloop[l](k){}\drawundirectedloop[l](R){}\drawundirectedloop(W){}\drawundirectedloop(Y){}\drawundirectedloop[l](V){}
\drawundirectedloop[r](J){}\drawundirectedloop(v){}\drawundirectedloop[l](z){}
\end{picture}
\end{center}
\unitlength=0,4mm

\begin{teo}\label{4isomorfismo}
If $\xi\in E_4$, then $\Gamma_{\xi}$ and $\Gamma_{(4)}$ are isomorphic as unrooted graphs.
\end{teo}

\begin{proof}
We can suppose without loss of generality that $\xi=0^\omega$ since we know that all words in $E_4$ belong to the same orbit. Recall that the length of $C(0^n)$ in $\Gamma_n$ is $2^{\lceil\frac{n-1}{2}\rceil}$. Consider the ball
$B_{\Gamma_n}(0^n,2^{\lceil\frac{n-1}{2}\rceil-1}-1)$ centered on $0^n$ and of radius $2^{\lceil\frac{n-1}{2}\rceil-1}-1$ as well as the ball $B_{\Gamma_{\xi}}(0^\omega,2^{\lceil\frac{n-1}{2}\rceil-1}-1)$
in $\Gamma_{\xi}$. By Lemma \ref{LemDistDec}, these balls are isomorphic for every $n\geq 2$ and their radii tend to infinity as $n\rightarrow\infty$.
\end{proof}

\begin{proof}[Proof of Theorem \ref{totaltheorem}, Parts 2 and 3]
Let $\xi \in E_2$. We can write $\xi=\alpha_1\beta_1\alpha_2\beta_2\dots$ where $\alpha_i,\beta_i \in \{0,1\}$ and either $\{\alpha_i\}_{i\geq 1}$ or $\{\beta_i\}_{i\geq 1}$ (but not both) has only finitely many $1$'s. Suppose without loss of generality that $\{\alpha_i\}_{i\geq 1}$ contains finitely many $1$'s (the other case can be treated in a similar way). Let $l\geq 0$ be such that the prefix $\xi_{2l-1}$ ends by the last non-zero value of the sequence $\{\alpha_i\}$, $\xi_{-1}$ denoting by convention the empty prefix. Let $k = l +\inf\{i\geq 0 \ : \ \beta_{i+l}=1\}$ so
that $\beta_k$ is the first $1$ appearing in $\xi$ after the prefix $\xi_{2l-1}$. The infinite word $\xi$ can be rewritten as
$$
\xi=\xi_{2k}0\beta_{k+1}0\beta_{k+2}\dots.
$$\\
\indent The next observation follows from the proof of Proposition \ref{PropInverseCovering}: since $\xi_{2k}$ ends by a $1$, $\xi_{2k}\in\mathcal{D}(0^{2k})^c$ and $\xi_{2k}0\in\mathcal{D}(0^{2k+1})$. For each $r\geq 1$, on one hand, $\xi_{2k}0\beta_{k+1}\dots0\beta_{k+r}0$ belongs to the left part of $\Gamma_{2k+2r+1}$ and there is an expansion in the passage from $\xi_{2k+2r}$ to $\xi_{2k+2r+1}$, and on the other hand, $\xi_{2k}0\beta_{k+1}\dots0\beta_{k+r}$ belongs to the central part of $\Gamma_{2k+2r}$ and a contraction occurs in the passage from $\xi_{2k+2r-1}$ to $\xi_{2k+2r}$.\\
\indent Let $\mathcal{CP}_{\xi_{2k}}$ be the cycle-path joining $\xi_{2k}$ to the central cycle of $\Gamma_{2k}$ and let $\mathcal{LP}_{\xi_{2k}}=\{k_i^{2k}\}_{i=1}^t$ be as defined in Notation \ref{notationCPCL}. By the previous observation, $|\mathcal{CP}_{\xi_{2k+2r}}|=|\mathcal{CP}_{\xi_{2k}}|=t$ for each $r\geq 1$. Moreover, it follows from Lemma \ref{lemLP}, that $k_i^{2k+2r}=k_i^{2k}$ for all $1\leq i\leq t-1$ and each $r\geq 1$. Thus, we conclude that for $N$ sufficiently large, the distance separating the vertex $\xi_{2n}$ from the central cycle in $\Gamma_{2n}$ is constant for every $n>N$. On the other hand, the length $k_t^{2k+2r}$ of the central cycle of $\Gamma_{2k+2r}$ tends to infinity as $r\to \infty$, so that, in the limit, the central cycle splits into two disjoint rays $\mathcal{R}$ and $\mathcal{R}'$. Thus, $\Gamma_{\xi}$ has at least two ends.\\
\indent By Proposition \ref{PropSeparability}, given any $\eta\in V(\Gamma_\xi)$, the height of the decoration $\mathcal{D}(\eta)$ is finite. Thus, any ray in $\Gamma_\xi$ must be equivalent either to $\mathcal{R}$ or to $\mathcal{R}'$. Thus, $\Gamma_{\xi}$ has exactly two ends.\\
\indent We show now that if $\xi\in E_1$, then $\Gamma_\xi$ has one end. On one hand, the SEC-sequence associated with $\xi$ cannot contain two consecutive C's (see Remark \ref{RemonSEC}). On the other hand, if $\xi_n$ belongs to the right part of $\Gamma_n$ and
$n'$ is the first index greater than $n$ such that $\xi_{n'}$ belongs again to the right part of $\Gamma_{n'}$, then
$|\mathcal{CP}_{\xi_{n'}}|=|\mathcal{CP}_{\xi_{n}}|+1$.\\

\begin{lem} \label{lemonCP}
For any $\xi\in\{0,1\}^\omega$, consider the subgraph of $(\Gamma_\xi,\xi)$, $(\mathcal{CP}_\xi,\xi):=\lim_{n\to\infty}(\mathcal{CP}_{\xi_n},\xi_n)$. Then $\mathcal{CP}_\xi$ is a cycle-path of infinite length, $\mathcal{CP}_\xi=\{P_i\}_{i\geq 1}$, if and only if $\xi\in E_1$. In this case, for any $n\geq 1$ such that $\xi_n$ ends with a $1$, the subgraph $\{P^n_i\}_{i=1}^{t-1}$ of $\mathcal{CP}_{\xi_n}=\{P^n_i\}_{i=1}^t$ is isomorphic to  $\{P_i\}_{i=1}^{t-1}\subset\mathcal{CP}_\xi$.
\end{lem}

\begin{proof}[Proof of the lemma]
Observe that, $\mathcal{CP}_{\xi}$ is a cycle-path of infinite length if and only if there exist infinitely many indices $n$ such that $\xi_n$ is in the right part of $\Gamma_n$; otherwise, we would observe from some point in the SEC-sequence associated with $\xi$ an alternating sequence of expansions E and contractions C, so that $|\mathcal{CP}_{\xi_{n}}|$ would be bounded as $n\to\infty$.\\
\indent It is easy to check that $E_1$ coincides with the set of words $\xi \in \{0,1\}^{\omega}$ in which the number of subwords of type $10^k1$, with $k\geq 0$ even, is infinite. Hence, if $\xi\in E_1$, there are infinitely many indices $n$ such that $\xi_n$ belongs to the right part of the graph $\Gamma_n$. Conversely, suppose that $\xi\notin E_1$, so that the number of subwords in $\xi$ of type $10^k1$, with $k\geq 0$ even, must be finite. Thus, the SEC-sequence associated with $\xi$ contains from some point an alternating sequence of expansions E and contractions C. It
follows that $\mathcal{CP}_\xi$ is not a cycle-path of infinite length (it consists of a finite cycle-path joining $\xi$ to a double-ray).\\
\indent The second part of the lemma follows from the proof of Proposition \ref{PropInverseCovering}.
\end{proof}

\emph{End of Proof of Theorem 4.1.} Let $\mathcal{R}$ be a ray contained in $\mathcal{CP}_{\xi}$. Each cycle $P_i\in \mathcal{CP}_{\xi}$ is finite and, by Proposition \ref{PropSeparability}, there is a finite decoration attached to each of its vertices. Hence, any other ray $\mathcal{R}'$ in $\Gamma_\xi$ must have infinitely many vertices in common with $\mathcal{R}$. We thus conclude that $\Gamma_{\xi}$ has one end.
\end{proof}

\subsection{Limit graphs with two ends}

Given $\xi\in E_2$, let $w$ be the longest prefix of $\xi$ ending by a subword of type $10^k1$, with $k$ even (see \emph{Proof of Theorem}
\ref{totaltheorem}, \emph{Parts 2 and 3}). If $\{\alpha_i\}\equiv 0$, then we set $w$ to be the empty word; if
$\{\beta_i\}\equiv 0$, then $w$ is the prefix of $\xi$ of length 1. Thus, we can write $\xi = w0x_10x_2\ldots$, $x_i\in\{0,1\}$. Introduce $\overline{\xi}=w0\overline{x}_10\overline{x}_2\ldots$, where
$\overline{x}_i=1-x_i$ for every $i\geq 1$. Note that $\xi\in E_2$ if and only if $\bar{\xi}\in E_2$.

\begin{teo}\label{xixibarra}
Let $\xi,\eta\in E_2$. Then $\Gamma_{\xi}$ and $\Gamma_{\eta}$ are isomorphic as unrooted graphs if and
only if either $\xi\sim\eta$ or $\overline{\xi}\sim\eta$. In particular, each isomorphism class of $2$-ended graphs consists of exactly two orbits.
\end{teo}

\begin{proof}
\indent $\Leftarrow$: If $\xi\sim\eta$, then $\Gamma_{\xi} = \Gamma_{\eta}$ by Proposition \ref{cofinality}. Suppose now that $\eta\sim\overline{\xi}$. This implies that there exist prefixes $v,\widehat{v},\widetilde{v}$ with $|v|=|\widehat{v}|=|\widetilde{v}|=n_0$ and where $\widetilde{v}$ ends by a $1$, such that
$$
\eta = v0z_10z_2\ldots \qquad \overline{\xi}=\hat{v}0z_10z_2\ldots
\qquad \xi = \widetilde{v}0\overline{z}_10\overline{z}_2\ldots,
$$
where $\overline{z}_i = 1-z_i$ for each $i\geq 1$. It follows from
the proof of Proposition \ref{PropInverseCovering} that
$\xi_{n_0+2}$ belongs to $\mathcal{D}(0^{n_0+2})^c\subset
\Gamma_{n_0+2}$. If we encode $\Gamma_{n_0+2}$ by the diagram
$\overline{D}_{n_0+2}$ (see Subsection \ref{finiteSch}), let
$u_{n_0+2}$ be the vertex of $\overline{D}_{n_0+2}$ to which is
attached the decoration containing $\xi_{n_0+2}$. Let
$u'_{n_0+2}\in \overline{D}_{n_0+2}$ be the symmetric vertex with
respect to the middle vertex of $\overline{D}_{n_0+2}$ so that the
decorations attached to $u_{n_0+2}$ and $u'_{n_0+2}$ have the same
height (observe that the vertices of $\overline{D}_n$ are labeled symmetrically with respect to the middle vertex labeled by $n$.) Finally, let
$\widetilde{\xi}_{n_0+2}\in \mathcal{D}(u'_{n_0+2})$ be
symmetric to $\xi_{n_0+2}$ as shown in Fig. 8.

\hspace{2cm}\begin{picture}(0,0)%
\includegraphics{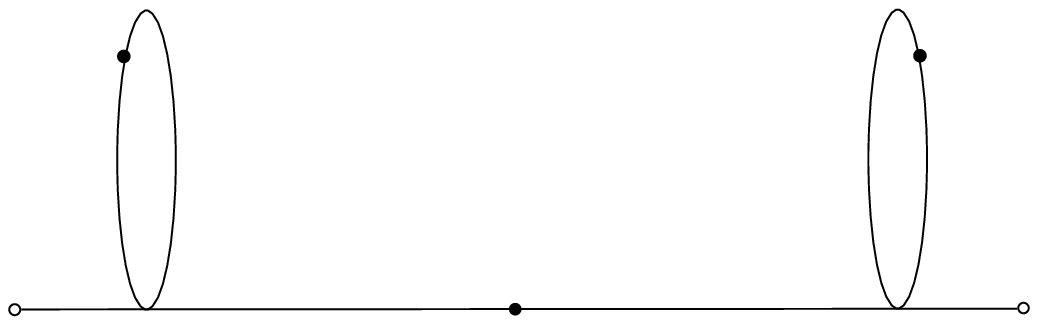}%
\end{picture}%
%
%
\setlength{\unitlength}{4342sp}%
\begingroup\makeatletter\ifx\SetFigFont\undefined%
\gdef\SetFigFont#1#2#3#4#5{%
  \reset@font\fontsize{#1}{#2pt}%
  \fontfamily{#3}\fontseries{#4}\fontshape{#5}%
  \selectfont}%
\fi\endgroup%
\begin{picture}(4741,1928)(1735,-1307)
\put(2060,362){\makebox(0,0)[lb]{\smash{{\SetFigFont{10}{12.0}{\rmdefault}{\mddefault}{\updefault}{\color[rgb]{0,0,0}$\xi_{n_0+2}$}%
}}}}
\put(6050,343){\makebox(0,0)[lb]{\smash{{\SetFigFont{10}{12.0}{\rmdefault}{\mddefault}{\updefault}{\color[rgb]{0,0,0}$\widetilde{\xi}_{n_0+2}$}%
}}}}
\put(3952,-908){\makebox(0,0)[lb]{\smash{{\SetFigFont{10}{12.0}{\rmdefault}{\mddefault}{\updefault}{\color[rgb]{0,0,0}$n_0+2$}%
}}}}
\put(1750,-891){\makebox(0,0)[lb]{\smash{{\SetFigFont{10}{12.0}{\rmdefault}{\mddefault}{\updefault}{\color[rgb]{0,0,0}$n_0+2$}%
}}}}
\put(6264,-891){\makebox(0,0)[lb]{\smash{{\SetFigFont{10}{12.0}{\rmdefault}{\mddefault}{\updefault}{\color[rgb]{0,0,0}$n_0+2$}%
}}}}
\put(2485,-880){\makebox(0,0)[lb]{\smash{{\SetFigFont{10}{12.0}{\rmdefault}{\mddefault}{\updefault}{\color[rgb]{0,0,0}$u_{n_0+2}$}%
}}}}
\put(5640,-880){\makebox(0,0)[lb]{\smash{{\SetFigFont{10}{12.0}{\rmdefault}{\mddefault}{\updefault}{\color[rgb]{0,0,0}$u'_{n_0+2}$}%
}}}}
\put(3809,-1247){\makebox(0,0)[lb]{\smash{{\SetFigFont{10}{12.0}{\rmdefault}{\mddefault}{\updefault}{\color[rgb]{0,0,0}\textbf{Fig. 8.} $\bar{D}_{n_0+2}$}%
}}}}
\end{picture}%

\vspace{0.5cm}

Consider  $\varphi(\xi)= \varphi (\xi_{n_0+2}0\overline{z}_20\overline{z}_3\ldots) := \widetilde{\xi}_{n_0+2}0z_20z_3\ldots.$ Since $\varphi(\xi)\sim\eta$, then by Proposition \ref{cofinality}, $\varphi(\xi)\in V(\Gamma_\eta)$.\\
\indent For each $i\geq 1$, let $R_i = d(\xi_{i+2},0^{i+2})$ in
$\Gamma_{i+2}$. Then, since $\{(\Gamma_n,\xi_n)\}_{n\geq 1}$
(respectively $\{(\Gamma_n,\eta_n)\}_{n\geq 1}$) converges to
$(\Gamma_\xi,\xi)$ (respectively $(\Gamma_\eta,\eta)$), there
exists $N(R_i)>1$ such that for all $n>N(R_i)$, we have
$$
B_{\Gamma_{\xi}}(\xi,R_i)\simeq B_{\Gamma_{n}}(\xi_{n},R_i)\simeq
B_{\Gamma_{n}}(\varphi(\xi)_{n},R_i)\simeq B_{\Gamma_{\eta}}(\varphi(\xi),R_i),
$$
where the second isomorphism follows from the previous observations and from the fact that if $\xi_n\in \mathcal{D}(0^n)^c$ and $\xi_{n+2}\in \mathcal{D}(0^{n+2})^c\setminus \mathcal{D}(0^{n+1}1)$, then $\xi_{n+2}=\xi_n0x$ where $x\in\{0,1\}$, so that $x$ determines whether $\xi_{n+2}$ belongs to a decoration situated in the left-half or in the right-half of $D_{n+2}$.
 Since $\xi\in E_2$, $R_i$ tends to infinity as $i\to\infty$ by Proposition \ref{cofinality}. We conclude that $\Gamma_\xi$ and $\Gamma_\eta$ are isomorphic.\\
\indent $\Rightarrow$: Let $\eta,\xi \in E_2$ and suppose that $\xi\nsim\eta$ and $\overline{\xi}\nsim\eta$. For the sake of contradiction, suppose that there exists an isomorphism $\varphi:\Gamma_{\xi}\rightarrow \Gamma_{\eta}$ of unrooted graphs. We will show that there exists a radius $R$ such that the balls $B_{\Gamma_{\xi}}(\xi,R)$
and $B_{\Gamma_{\eta}}(\varphi(\xi),R)$ are not isomorphic which will yield a contradiction.\\
\indent By Proposition \ref{cofinality}, $\varphi(\xi)\sim\eta$. There exist finite words $u,\hat{u}$  such that
$$
\xi = u0x_10x_2\ldots \qquad \varphi(\xi)=\hat{u}0y_10y_2\ldots.
$$
and $|u|=|\hat{u}|$. If not, then  $\xi$ and
$\varphi(\xi)$  belong to infinite graphs whose decorations
have different parity, which is a contradiction.
Indeed, depending on whether it is the sequence $\{\alpha_i\}$ or $\{\beta_j\}$ that takes only finitely many $1$'s, the corresponding infinite Schreier graph has decorations of only even or only odd height.\\
\indent Since
$\eta\nsim\xi$, $\eta\nsim\bar{\xi}$ but $\eta\sim\varphi(\xi)$,
there exist infinitely many indices $i$ such that $x_i\neq y_i$
and $x_{i+1}=y_{i+1}$. We can thus suppose without loss of
generality that $u$ and $\hat{u}$ both end with a $1$. Moreover,
arguments as in Proof of Proposition \ref{PropInverseCovering} imply that there exists an infinite
subsequence $\{n_k\}_{k\geq 0}\subset\mathbb{N}$ such that, for
any $k\geq 0$, $\xi_{n_k}$ and $\varphi(\xi)_{n_k}$ belong to
decorations attached to vertices in different halves of
$D_{n_k}$, while $\xi_{n_k+2}$ and $\varphi(\xi)_{n_k+2}$ belong
to decorations attached to distinct vertices $v_{n_k+2}$ and
$v_{n_k+2}'$ in the same half of $D_{n_k+2}$. Thus
$d(v_{n_k+2},w_{n_k+2}) \neq d(v'_{n_k+2},w_{n_k+2})$ where
$w_{n_k+2}$ denotes the vertex of $\overline{D}_{n_k+2}$ encoding
$0^{n_k+1}1$ (see Fig. 9).

\hspace{2cm}\begin{picture}(0,0)%
\includegraphics{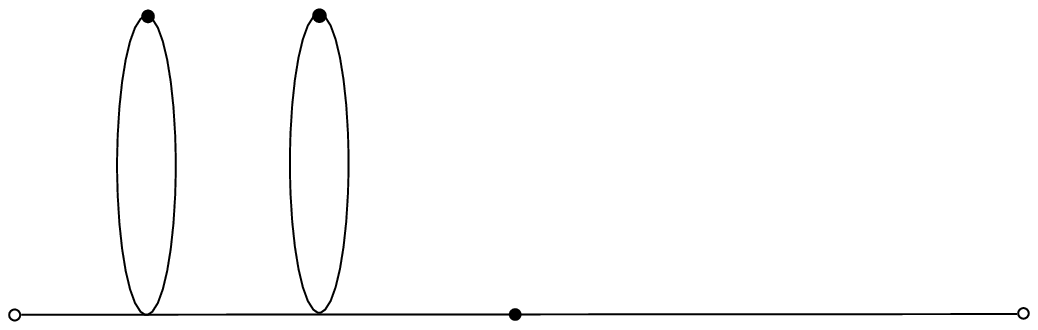}%
\end{picture}%
%
%
\setlength{\unitlength}{4342sp}%
\begingroup\makeatletter\ifx\SetFigFont\undefined%
\gdef\SetFigFont#1#2#3#4#5{%
  \reset@font\fontsize{#1}{#2pt}%
  \fontfamily{#3}\fontseries{#4}\fontshape{#5}%
  \selectfont}%
\fi\endgroup%
\begin{picture}(4698,2084)(1778,-1729)
\put(6271,-1371){\makebox(0,0)[lb]{\smash{{\SetFigFont{10}{12.0}{\rmdefault}{\mddefault}{\updefault}{\color[rgb]{0,0,0}$n_k+2$}%
}}}}
\put(2228,195){\makebox(0,0)[lb]{\smash{{\SetFigFont{10}{12.0}{\rmdefault}{\mddefault}{\updefault}{\color[rgb]{0,0,0}$\varphi(\xi)_{n_k+2}$}%
}}}}
\put(1793,-1364){\makebox(0,0)[lb]{\smash{{\SetFigFont{10}{12.0}{\rmdefault}{\mddefault}{\updefault}{\color[rgb]{0,0,0}$n_k+2$}%
}}}}
\put(2450,-1369){\makebox(0,0)[lb]{\smash{{\SetFigFont{10}{12.0}{\rmdefault}{\mddefault}{\updefault}{\color[rgb]{0,0,0}$v'_{n_k+2}$}%
}}}}
\put(3208,-1367){\makebox(0,0)[lb]{\smash{{\SetFigFont{10}{12.0}{\rmdefault}{\mddefault}{\updefault}{\color[rgb]{0,0,0}$v_{n_k+2}$}%
}}}}
\put(3993,-1367){\makebox(0,0)[lb]{\smash{{\SetFigFont{10}{12.0}{\rmdefault}{\mddefault}{\updefault}{\color[rgb]{0,0,0}$w_{n_k+2}$}%
}}}}
\put(3166,218){\makebox(0,0)[lb]{\smash{{\SetFigFont{10}{12.0}{\rmdefault}{\mddefault}{\updefault}{\color[rgb]{0,0,0}$\xi_{n_k+2}$}%
}}}}
\put(3792,-1669){\makebox(0,0)[lb]{\smash{{\SetFigFont{10}{12.0}{\rmdefault}{\mddefault}{\updefault}{\color[rgb]{0,0,0}\textbf{Fig. 9.} $\bar{D}_{n_k+2}$}%
}}}}
\end{picture}%

\vspace{0.5cm}

Suppose without loss of generality, that $d_{\Gamma_{n_k+2}}(v_{n_k+2},w_{n_k+2})<d_{\Gamma_{n_k+2}}(v'_{n_k+2},w_{n_k+2})$. By RCR, for each $k\geq 0$, the minimal distance between $v_{n_k+2}$ (respectively $v'_{n_k+2}$) and a vertex in $\overline{D}_{n_k+2}$ labeled by $n_0+2$ is constant. On the other hand, since $\varphi(\xi)\in E_2$, the distance between $v_{n_k+2}$ (respectively $v'_{n_k+2}$) and the boundary vertices of $\overline{D}_{n_k+2}$ tends to infinity as $k\to\infty$.\\
\indent Thus, for any sufficiently large $k$, there exists a radius $R(n_0)$ defined as the minimum over all radii $r$ such that the ball $B_{\Gamma_{n_k+2}}(\xi_{n_k+2},r)$ contains a $(n_0+2)$-decoration attached to a vertex of $D_{n_k+2}$, and such that both $B_{\Gamma_{n_k+2}}(\xi_{n_k+2},r)$ and $B_{\Gamma_{n_k+2}}(\varphi(\xi)_{n_k+2},r)$ do not contain any boundary vertices of $\overline{D}_{n_k+2}$.\\
\indent Hence, there exists $N(R(n_0))>1$ such that for all $n_k\geq N(R(n_0))$, we have
$$
B_{\Gamma_{\xi}}(\xi,R(n_0))\simeq B_{\Gamma_{n_k+2}}(\xi_{n_k+2},R(n_0))\not\simeq
B_{\Gamma_{n_k+2}}(\varphi(\xi)_{n_k+2},R(n_0))\simeq B_{\Gamma_{\eta}}(\varphi(\xi),R(n_0)).
$$
\end{proof}

\indent For $\xi\in E_2$, we will now construct explicitly a $4$-regular, $2$-ended graph $\Gamma(\xi)$ that we will show to be isomorphic (as unrooted graph) to $\Gamma_\xi$ (see Theorem \ref{2isomorfismo} below).
Recall from Theorem \ref{totaltheorem} that $E_2 = \{\alpha_1\beta_1\alpha_2\beta_2...\ |\ \alpha_i,\beta_j\in\{0,1\} \textrm{ such that either } \{\alpha_i\} \textrm{ or } \{\beta_j\} \textrm{ has finitely many 1's}\}$.
Observe that in the case where the sequence $\{\alpha_i\}$ (respectively $\{\beta_j\}$) takes the value 1 a finite number of times, only decorations of even (respectively odd) height will appear in the limit graph. We can thus partition the set $E_2=E_{2,even}\sqcup E_{2,odd}$. We will only consider the case of $\xi\in E_{2,even}$, (the other case can be treated similarly.)
Suppose therefore that $\xi = \alpha_1\beta_1\alpha_2\beta_2\alpha_3\beta_3\ldots \in E_{2,even}$, so that $\{\alpha_i\}_{i\geq 1}$ contains finitely many $1$'s
and define the graph $\Gamma(\xi)$ associated with $\xi$.

\begin{prop} \label{PropDef2ends}
Consider the following subsets of $\mathbb{Z}$:
\begin{displaymath}
A'_0:= 2\mathbb{Z} \ \ \mbox{ and } \ \ \ A'_k:=\{n\in \mathbb{Z} \
: \ n\equiv 2^k-1-\sum_{i=1}^k2^i\beta_{i+1} \mod 2^{k+1}\} \
\mbox{ for each }k\geq 1.
\end{displaymath}
Construct $\Gamma(\xi)$ as a straight line with integer vertices with, for each $k\geq 0$, a $(2k+2)$-decoration attached by its unique vertex of degree 2 to every vertex corresponding to an integer in $A'_k$. The constructed graph $\Gamma(\xi)$ is well defined.
\end{prop}

\indent The graph $\Gamma(\xi)$ for $\xi\in E_{2,odd}$ is defined similarly, replacing $\beta$ by $\alpha$ in the definition of $A_k'$, and by attaching $(2k+1)$-decorations instead of $(2k+2)$-decorations. Proposition \ref{PropDef2ends} follows from the following lemma.

\begin{lem}
Let $\{x_i\}_{i\geq 1}$ be a sequence with values in $\{0,1\}$ supposing that it takes both values infinitely many times. For every $k\geq 1$, define
$$
L_k = \{n \in \mathbb{Z} \ : \ n \equiv 2^k-1 -\sum_{i=1}^k2^ix_i
\mod 2^{k+1}\}.
$$
Then $\coprod_{k\geq 1}L_k$ is a partition of $2\mathbb{Z}+1$.
\end{lem}

\begin{proof}
We first prove that $L_k\cap L_h = \emptyset$ if $k\neq h$. Suppose that $h>k$ and let $s,t\in \mathbb{Z}$ be such that
$$
2^k-1-\sum_{i=1}^k2^ix_i + 2^{k+1}t = 2^h-1-\sum_{i=1}^h2^ix_i +
2^{h+1}s.
$$
Hence, $2^k(1+2t) = 2^h(1+2s) -\sum_{i=k+1}^h2^ix_i = 2^{k+1}m$ where $m\in \mathbb{Z}$ which is impossible.\\
\indent We prove now that $2\mathbb{Z}+1\subset \coprod_{k\geq 1}L_k$ by showing that for any $n\in\mathbb{Z}$ odd, there exist $k\geq 1$ and $t\in \mathbb{Z}$ such that
\begin{eqnarray}\label{congruenzona}
n = 2^k-1-\sum_{i=1}^k2^ix_i + t\cdot 2^{k+1}.
\end{eqnarray}
Let $p_1:= \min\{i\geq 1 : x_i =1\}$. If $n+1=0$, then a solution of (\ref{congruenzona}) is $k=p_1$ and $t=0$. Suppose that $n+1\neq 0$ and write $n+1 = 2^{u_1}m_1$, with $u_1\geq 1$ and $m_1$ odd. If $p_1\neq u_1$, a solution of (\ref{congruenzona}) is $k = \min\{p_1,u_1\}$ and $t = \frac{m_1-1}{2}$ if $k=u_1$ or $t = 2^{u_1-p_1-1}m_1$ if $k=p_1$. If $u_1=p_1$, then necessarily, $k>p_1$ and (\ref{congruenzona}) becomes
\begin{equation} \label{congruenzonabis}
2^{p_1}\left(m_1+1+\sum_{i=p_1+1}^k2^{i-p_1}x_i\right)=2^k(1+2t).
\end{equation}
Let $p_2:=\min\{i \geq 1 : x_{p_1+i} = 1\}$. If $m_1+1 = 0$, then a solution of (\ref{congruenzonabis}) is
$k=p_1+p_2$ and $t=0$. Suppose that $m_1+1\neq 0$ and write $m_1+1=2^{u_2}m_2$, with $u_2\geq 1$ and $m_2$ odd. If $p_2\neq u_2$, then setting $k=p_1 + \min\{p_2,u_2\}$, one can check that (\ref{congruenzonabis}) has a solution. If
$p_2=u_2$, then necessarily $k>p_1+p_2$ and (\ref{congruenzonabis}) becomes
\begin{equation}\label{congruenzonater}
2^{p_1+p_2}\left(m_2+1+\sum_{i=p_1+p_2+1}^k2^{i-p_1-p_2}x_i\right)=2^k(1+2t).
\end{equation}
We iterate the above argument: at the $s+1$ step, let $p_s:=\min\{i \geq 1 : x_{\sum_{r=1}^{s-1}p_r+i} = 1\}$;  if $m_s+1=0$, set $k=\sum_{r=1}^sp_r$ and $t=0$; if $m_s+1\neq0$, write $m_s+1=2^{u_{s+1}}m_{s+1}$, with $u_{s+1}\geq 1$ and $m_{s+1}$ odd. If $p_s\neq u_s$, one can check as above that (\ref{congruenzona}) has a solution. Clearly, the sequence $\{|m_s+1|\}_{s\geq 1}$ is strictly decreasing. If $n<0$, then there exists $s_0$ such that $m_{s_0}=-1$ and (\ref{congruenzona}) has a solution. If $n>0$, then there exists $s_0$ such that $m_s=1$ for all $s\geq s_0$. In the latter case, $u_s=1$ for all $s>s_0$. On the other hand, there are only finitely many indices $s$ such that $p_s=u_s=1$. Otherwise, there would exist $i_0$ such that $x_i=1$ for every $i>i_0$ which would contradict the hypothesis on the sequence $\{x_i\}_{\geq 1}$.
\end{proof}

\begin{center}
\begin{picture}(350,155)
\letvertex A=(5,10)\letvertex B=(35,10)\letvertex C=(75,10)\letvertex D=(75,40)
\letvertex E=(75,70)\letvertex F=(115,10)\letvertex G=(175,10)\letvertex H=(155,30)
\letvertex I=(130,30)\letvertex L=(175,50)\letvertex M=(195,30)\letvertex N=(220,30)
\letvertex O=(155,70)\letvertex R=(175,120)\letvertex S=(175,150)\letvertex T=(235,10)
\letvertex U=(275,10)\letvertex V=(275,40)\letvertex P=(195,70)\letvertex Q=(175,90)
\letvertex X=(315,10)\letvertex W=(275,70)\letvertex Y=(345,10)

\put(128,37){$\xi$}\put(232,-3){$\widetilde{\xi}$}

\put(110,-10){\textbf{Fig. 10.} A finite part of $\Gamma(\xi)$.}

\drawvertex(B){$\bullet$}\drawvertex(C){$\bullet$}
\drawvertex(D){$\bullet$}
\drawvertex(E){$\bullet$}\drawvertex(F){$\bullet$}
\drawvertex(G){$\bullet$}\drawvertex(H){$\bullet$}
\drawvertex(I){$\bullet$}\drawvertex(L){$\bullet$}
\drawvertex(M){$\bullet$}\drawvertex(N){$\bullet$}
\drawvertex(O){$\bullet$}\drawvertex(P){$\bullet$}
\drawvertex(Q){$\bullet$}\drawvertex(R){$\bullet$}
\drawvertex(S){$\bullet$}\drawvertex(T){$\bullet$}
\drawvertex(U){$\bullet$}\drawvertex(V){$\bullet$}
\drawvertex(W){$\bullet$}\drawvertex(X){$\bullet$}

\drawundirectededge(A,B){}\drawundirectededge(B,C){}
\drawundirectedcurvededge(C,D){}\drawundirectedcurvededge(D,C){}
\drawundirectedcurvededge(E,D){}\drawundirectedcurvededge(D,E){}

\drawundirectededge(C,F){} \drawundirectededge(F,G){}
\drawundirectededge(G,H){} \drawundirectededge(H,L){}
\drawundirectededge(L,M){} \drawundirectededge(M,G){}
\drawundirectededge(L,O){} \drawundirectededge(O,Q){}
\drawundirectededge(Q,P){} \drawundirectededge(P,L){}
\drawundirectededge(G,T){} \drawundirectededge(T,U){}
\drawundirectededge(U,X){} \drawundirectededge(X,Y){}

\drawundirectedcurvededge(I,H){}\drawundirectedcurvededge(H,I){}
\drawundirectedcurvededge(M,N){}\drawundirectedcurvededge(N,M){}
\drawundirectedcurvededge(Q,R){}\drawundirectedcurvededge(R,Q){}
\drawundirectedcurvededge(R,S){}\drawundirectedcurvededge(S,R){}
\drawundirectedcurvededge(U,V){}\drawundirectedcurvededge(V,U){}
\drawundirectedcurvededge(V,W){}\drawundirectedcurvededge(W,V){}

\drawundirectedloop(B){}\drawundirectedloop(E){}\drawundirectedloop(S){}
\drawundirectedloop[l](O){}\drawundirectedloop[r](P){}\drawundirectedloop[l](I){}\drawundirectedloop[r](N){}
\drawundirectedloop(T){}\drawundirectedloop(W){}\drawundirectedloop(X){}\drawundirectedloop(F){}
\end{picture}
\end{center}
\vspace{0,5cm}

\begin{teo}\label{2isomorfismo}
Let $\xi\in E_2$. Then $\Gamma_{\xi}$ is isomorphic to $\Gamma(\xi)$.
\end{teo}

\begin{proof}
We again only consider the case where $\xi\in E_{2,even}$ (the case $\xi\in E_{2,odd}$ can be treated analogously). For any $t\geq 1$, consider the diagram $\overline{D}_{2t}$ encoding $\Gamma_{2t}$. Consider also $\hat{\xi}:=010\beta_20\beta_3\dots$. Note that $\hat\xi$ and $\xi$ are
cofinal, $\hat\xi\sim\xi$, so that, by Proposition \ref{cofinality}, $\Gamma_{\xi}=\Gamma_{\hat{\xi}}$.
Arguments as in Proof of Proposition \ref{PropInverseCovering} imply that $\hat{\xi}_{2t}\in V(\Gamma_{2t})$ is situated on the central cycle of this graph, and therefore can be identified with a vertex of $\overline{D}_{2t}$ (also denoted by $\hat{\xi}_{2t}$). Let $v_{2t}$ be the middle vertex in $\overline{D}_{2t}$; it is labeled by $2t$ (see the figures below).\\
\indent It is convenient to embed the diagram $\overline{D}_{2t}$ in $\mathbb{Z}$, so that every vertex of $\overline{D}_{2t}$ has a \emph{coordinate} (in addition to its label). Suppose that we embed $\overline{D}_{2t}$ in $\mathbb{Z}$ so that the coordinate of $\hat{\xi}_{2t}$ is $0$.\\
\indent \textit{Claim}: The coordinate of $v_{2t}$ is $2^{t-1}-1-\sum_{i=1}^{t-1}2^i\beta_{i+1}$.\\
We prove the claim by induction on $t$. For $t=2$, $v_4$ has coordinate $1$ if $\beta_2=0$, while it has coordinate $-1$ if $\beta_2=1$, and the claim holds.\\
\indent Suppose that the assertion is true for $t-1$ and suppose that $\beta_{t-1}=0$ (the case $\beta_{t-1}=1$ is treated similarly). The diagram $\overline{D}_{2t-2}$ is

\vspace{1cm}

\hspace{2.5cm}\begin{picture}(0,0)%
\includegraphics{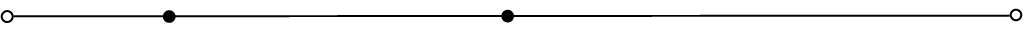}%
\end{picture}%
%
%
\setlength{\unitlength}{4342sp}%
\begingroup\makeatletter\ifx\SetFigFont\undefined%
\gdef\SetFigFont#1#2#3#4#5{%
  \reset@font\fontsize{#1}{#2pt}%
  \fontfamily{#3}\fontseries{#4}\fontshape{#5}%
  \selectfont}%
\fi\endgroup%
\begin{picture}(4465,297)(2011,-1447)
\put(2634,-1387){\makebox(0,0)[lb]{\smash{{\SetFigFont{10}{12.0}{\rmdefault}{\mddefault}{\updefault}{\color[rgb]{0,0,0}$\hat{\xi}_{2t-2}$}%
}}}}
\put(4083,-1379){\makebox(0,0)[lb]{\smash{{\SetFigFont{10}{12.0}{\rmdefault}{\mddefault}{\updefault}{\color[rgb]{0,0,0}$v_{2t-2}$}%
}}}}
\end{picture}%

\vspace{0.5cm}

\noindent If $\beta_t=0$, then by RCR, the diagram $\overline{D}_{2t}$ is

\vspace{1cm}

\hspace{2.5cm}\begin{picture}(0,0)%
\includegraphics{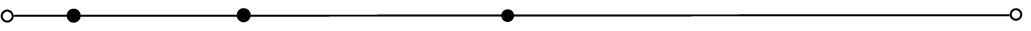}%
\end{picture}%
%
%
\setlength{\unitlength}{4342sp}%
\begingroup\makeatletter\ifx\SetFigFont\undefined%
\gdef\SetFigFont#1#2#3#4#5{%
  \reset@font\fontsize{#1}{#2pt}%
  \fontfamily{#3}\fontseries{#4}\fontshape{#5}%
  \selectfont}%
\fi\endgroup%
\begin{picture}(4465,317)(2011,-1463)
\put(4083,-1360){\makebox(0,0)[lb]{\smash{{\SetFigFont{10}{12.0}{\rmdefault}{\mddefault}{\updefault}{\color[rgb]{0,0,0}$v_{2t}$}%
}}}}
\put(2231,-1403){\makebox(0,0)[lb]{\smash{{\SetFigFont{10}{12.0}{\rmdefault}{\mddefault}{\updefault}{\color[rgb]{0,0,0}$\hat{\xi}_{2t}$}%
}}}}
\end{picture}%

\vspace{0.5cm}

\noindent Hence, recalling that the length of $\overline{D}_{2t}$ is $2^t$, the coordinate of $v_{2t}$ is
$$
2^{t-2}+2^{t-2}-1-\sum_{i=1}^{t-2}2^i\beta_{i+1}=2^{t-1}-1-\sum_{i=1}^{t-1}2^i\beta_{i+1}.
$$
On the other hand, if $\beta_t=1$, the diagram $\overline{D}_{2t}$ is

\vspace{1cm}

\hspace{2.5cm}\begin{picture}(0,0)%
\includegraphics{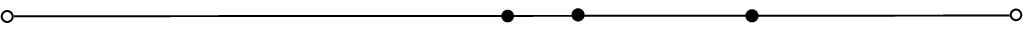}%
\end{picture}%
%
%
\setlength{\unitlength}{4342sp}%
\begingroup\makeatletter\ifx\SetFigFont\undefined%
\gdef\SetFigFont#1#2#3#4#5{%
  \reset@font\fontsize{#1}{#2pt}%
  \fontfamily{#3}\fontseries{#4}\fontshape{#5}%
  \selectfont}%
\fi\endgroup%
\begin{picture}(4465,312)(2011,-1457)
\put(4071,-1377){\makebox(0,0)[lb]{\smash{{\SetFigFont{10}{12.0}{\rmdefault}{\mddefault}{\updefault}{\color[rgb]{0,0,0}$v_{2t}$}%
}}}}
\put(4408,-1397){\makebox(0,0)[lb]{\smash{{\SetFigFont{10}{12.0}{\rmdefault}{\mddefault}{\updefault}{\color[rgb]{0,0,0}$\hat{\xi}_{2t}$}%
}}}}
\end{picture}%

\vspace{0.5cm}

Thus, in this case, the coordinate of $v_{2t}$ is
$$
2^{t-2}-1-\sum_{i=1}^{t-2}2^i\beta_{i+1}-2^{t-2}=2^{t-1}-1-\sum_{i=1}^{t-1}2^i\beta_{i+1},
$$
and the claim is proven.\\
\indent Let $R_t=d_{\Gamma_{2t}}(\hat{\xi}_{2t},0^{2t})$; we prove by induction on $t$, that for any $t\geq 2$, $B_{\Gamma(\xi)}(0,R_t)\simeq B_{\Gamma_{2t}}(\hat{\xi}_{2t},R_t)$. If $t=2$, then $R_2=1$ and, clearly, $B_{\Gamma(\xi)}(0,1)\simeq B_{\Gamma_{4}}(\hat{\xi}_{4},1)$. Suppose that $B_{\Gamma(\xi)}(0,R_{t-1})\simeq
B_{\Gamma_{2t-2}}(\hat{\xi}_{2t-2},R_{t-1})$. Then, by the RCR and the Claim, the only vertex in $D_{2t}$ labeled by $2t$ is situated at distance $2^{t-1}-1-\sum_{i=1}^{t-1}2^i\beta_{i+1}$ from $\hat{\xi}_{2t}$. Moreover, any two consecutive vertices in $\overline{D}_{2t}$ labeled by $2k+2$ are situated at distance $2^{k+1}$ from each other. Thus, by definition of the graph $\Gamma(\xi)$, $B_{\Gamma(\xi)}(0,R_t)\simeq
B_{\Gamma_{2t}}(\hat{\xi}_{2t},R_t)$. Finally, since $\hat{\xi}\in E_2$, $R_t$ tends to infinity as $t\to\infty$ and since $\Gamma_{\hat{\xi}}=\Gamma_\xi$, the proof of the theorem is completed.
\end{proof}

\section{The typical case: limit graphs with one end} \label{chapter1end}

\subsection{Classification of infinite Schreier graphs with one end} \label{section1end}

It follows from Theorem \ref{totaltheorem} that the set of infinite words whose associated infinite orbital Schreier graph has one end is
$$
E_1 = \{\alpha_1\beta_1\alpha_2\beta_2\ldots, \
\alpha_i,\beta_j\in \{0,1\} | \  \textrm{$\{\alpha_i\}_{i\geq 1}$
and $\{\beta_j\}_{j\geq 1}$ both contain infinitely many $1$'s}\}.
$$
We prove in this subsection Part \emph{3.} of Corollary \ref{ThmClassif} by classifying all limit graphs $\Gamma_\xi$, $\xi\in E_1$, up to isomorphism (of unrooted graphs). We begin with the following lemma:

\begin{lem}\label{lemmaonlyoneend}
Let $\xi\in E_1$ and let $\{P_i\}_{i\geq 1},\{Q_j\}_{j\geq 1}$ be two infinite cycle-paths in $\Gamma_{\xi}$. Then, there exist $i_0, j_0\in \mathbb{N}$ such that $P_{i_0+k}= Q_{j_0+k}$ for every $k\geq 0$.
\end{lem}

\begin{proof}
Let $\mathcal{R}_1\subset\{P_i\}_{i\geq 1}$ and $\mathcal{R}_2\subset\{Q_j\}_{j\geq 1}$ be two rays. Since $\Gamma_{\xi}$ has one end, $\mathcal{R}_1$ and $\mathcal{R}_2$ are equivalent. Suppose that $\{P_i\}_{i\geq 1}$
and $\{Q_j\}_{j\geq 1}$ are disjoint. Since $\Gamma_{\xi}$ is separable with blocks either cycles or single edges, there exists a cut vertex which separates $\{P_i\}_{i\geq 1}$ from $\{Q_j\}_{j\geq 1}$; this contradicts the fact that
$\mathcal{R}_1$ and $\mathcal{R}_2$ are equivalent.\\
\indent If there is only a finite number of vertices belonging both to $\{P_i\}_{i\geq 1}$ and $\{Q_j\}_{j\geq 1}$, then there are integers $I,J$ such that $\{P_i\}_{i\geq I}$ and $\{Q_j\}_{j\geq J}$ are disjoint. Then, $V(\{P_i\}_{i=1}^{I-1}\cup \{Q_j\}_{j=1}^{J-1})$ is a finite subset of vertices which separates $\{P_i\}_{i\geq I}$ from $\{Q_j\}_{j\geq J}$ which contradicts the fact that $\mathcal{R}_1$ and $\mathcal{R}_2$ are equivalent.\\
\indent Suppose that $|\{P_i\}_{i\geq 1}\cap\{Q_j\}_{j\geq 1}|=\infty$ and consider the minimal $i_0,j_0$ such that $P_{i_0}$ and $Q_{j_0}$ have a vertex $v$ in common. Let $w\in P_{i_0+k}\cap Q_{j_0+k'}$ for some $k,k'\geq 0$ be at minimal distance from $v$. Then, $\{P_i\}_{i=i_0}^{i_0+k}$ and $\{Q_j\}_{j=j_0}^{j_0+k'}$ must coincide. Indeed, suppose that this is not the
case: if $k=k'=0$, then there are four disjoint paths joining $v$ to $w$ which contradicts Part \emph{1.} of Proposition \ref{PropSeparability}. If $k=0$ or $k'=0$, then there are three disjoint paths joining $v$ to $w$ which is again a contradiction. Finally, if $k,k'\neq 0$, then there are two disjoint paths joining $v$ to $w$
which again contradicts Proposition \ref{PropSeparability}.
\end{proof}

\begin{prop}  \label{PropAlg}
An element $\xi\in\{0,1\}^\omega$, $\xi$ not cofinal to $1^\omega$, belongs to $E_1$ if and only if there exists a unique triple $(k,\{m_l\}_{l\geq 0},\{t_l\}_{l\geq 0})$ where $k\geq 1$ and $m_0\geq 0$ are integers (and $m_0$ is even); $t_0=0$; and $\{m_l\}_{l\geq 1}$, $\{t_l\}_{l\geq 1}$ are sequences of strictly positive integers (and the $m_l$'s are even), such that $\xi$ can be written as
\begin{equation} \label{CanFormOfXi}
\xi=0^{k-1}1(0x^0_{1}0x^0_{2}\dots 0x^0_{\frac{m_0}{2}})1^{t_1}(0x^1_{1}0x^1_{2}\dots 0x^1_{\frac{m_1}{2}})1^{t_2}\dots.
\end{equation}
with $x_i^j\in\{0,1\}$ for all $i,j$.\\
\indent If $\xi=w1^\omega$ for some $w\in X^\ast$, then there exists a unique triple $(k,\{m_l\}_{l=0}^{l_0},\{t_l\}_{l=0}^{l_0})$ where $k\geq 1$ and $m_0\geq 0$ are integers (and $m_0$ is even);
$t_0=0$; and $\{m_l\}_{l=1}^{l_0}$, $\{t_l\}_{l=1}^{l_0}$ are finite sequences of strictly positive integers (and the $m_l$'s are even),  such that $\xi$ can be written as
$$\xi=0^{k-1}1(0x^0_{1}0x^0_{2}\dots 0x^0_{\frac{m_0}{2}})1^{t_1}(0x^1_{1}0x^1_{2}\dots 0x^1_{\frac{m_1}{2}})1^{t_2}\dots
.1^{t_{l_0}}(0x^{l_0}_{1}0x^{l_0}_{2}\dots 0x^{l_0}_{\frac{m_{l_0}}{2}})1^\omega.$$
\end{prop}

\begin{cor} \label{RemonSEC1end}
Observe that the index of stability of $\xi$ is equal to $k+m_0$ in notation of (\ref{CanFormOfXi}). Hence, the SEC-sequence corresponding to $\xi\in E_1$, $\xi\neq w1^\omega$, is
\begin{equation} \label{SEC}
S^{k+m_0}E^{t_1}(EC)^{\frac{m_1}{2}}E^{t_2}(EC)^{\frac{m_2}{2}}\ldots
\end{equation}
while the SEC-sequence corresponding to $\xi=w1^\omega$ is
$$
S^{k+m_0}E^{t_1}(EC)^{\frac{m_1}{2}}E^{t_2}\ldots E^{t_{l_0}}(EC)^{\frac{m_{l_0}}{2}}E^{\omega}.
$$
\end{cor}

\begin{proof} Corollary follows from the description of the inverse covering map $\pi_{n+1}^{-1}$, as in Proof of Proposition \ref{PropInverseCovering} and from Remark \ref{RemonSEC}.
\end{proof}

\begin{proof}[Proof of Proposition \ref{PropAlg}]
Suppose that $\xi\in\{0,1\}^\omega$ has the form (\ref{CanFormOfXi}). If we write $\xi=\alpha_1\beta_1\alpha_2\beta_2\dots$ where $\alpha_i,\beta_i\in\{0,1\}$, then, since $t_l\geq 1$ for all $l\geq 1$, it follows that both sequences $\{\alpha_i\}_{i\geq 1}$ and $\{\beta_i\}_{i\geq 1}$ take infinitely often the value $1$. This implies by definition that $\xi\in E_1$.\\
\indent The proof of the converse is constructive; before writing down the explicit algorithm, introduce the following notation: given two sequences of integers $\{m_l\}_{l\geq 0}$ and
$\{t_l\}_{l\geq 0}$, define for all $i,j\geq 0$, $M_i:=\sum_{l=0}^im_l$, $T_j:=\sum_{l=0}^jt_l$. Finally, set $M_{-1}:=0$. Let $\xi\in E_1$ and for all $i\geq 1$, denote by $X(i)\in\{0,1\}$ the $i$-th letter in $\xi$. The output of the following algorithm is a triple $(k,\{m_l\}_{l\geq 0},\{t_l\}_{l\geq 0})$ satisfying the assumptions of the proposition.

        \noindent \textbf{procedure}($\xi$)\\
        $k:=\min\{n\geq 1; X(n)=1\}$\\
        $M_{-1}:=0$\\
       \begin{tabbing}\textbf{if} \=$X(k+1)=0$ \textbf{then} \hspace{7cm} \textbf{if} \=$X(k+1)=1$ \textbf{then}\\
        \>\textbf{for} \=$j=0$ \textbf{to} $"\infty"$        \hspace{7.5cm} $m_0:=0$\\
        \>\>$m_j:=2$                                         \hspace{8.4cm} $M_0:=0$\\
        \>\>$n:=k+M_{j-1}+T_j+3$                           \hspace{6.1cm} \textbf{for} \=$j=1$ \textbf{to} $"\infty"$\\
        \>\>\textbf{while} \=X(n)=0                          \hspace{8.1cm} $t_j:=1$\\
            \>\>\>$m_j:=m_j+2$                               \hspace{7.3cm} $n:=k+M_{j-1}+T_{j-1}+2$\\
            \>\>\>$n:=n+2$                                   \hspace{7.7cm} \textbf{while} \=X(n)=1\\
        \>\>\textbf{end while}                               \hspace{9.7cm} $t_j:=t_j+1$\\
        \>\>$M_j:=M_{j-1}+m_j$                               \hspace{8.6cm} $n:=n+1$\\
        \>\>\textbf{output} $m_j$                            \hspace{8.7cm} \textbf{end while}\\
       \>\>$t_{j+1}:=0$                                      \hspace{9cm} $T_j:=T_{j-1}+t_j$   \\
        \>\>\textbf{while} $X(n)=1$                          \hspace{8cm} \textbf{output} $t_j$\\
            \>\>\>$t_{j+1}:=t_{j+1}+1$                       \hspace{7cm} $m_j:=0$\\
            \>\>\>$n:=n+1$                                   \hspace{7.7cm} \textbf{while} $X(n)=0$\\
        \>\>\textbf{end while}                               \hspace{9.8cm} $m_j:=m_j+2$\\
        \>\>$T_{j+1}:=T_j+t_{j+1}$                           \hspace{8.8cm} $n:=n+2$\\
        \>\>\textbf{output} $t_{j+1}$                        \hspace{8.5cm} \textbf{end while}\\
        \>\textbf{end for}                                   \hspace{9.6cm} $M_j:=M_{j-1}+m_j$\\
                                                             \hspace{11.4cm} \textbf{output} $m_j$\\
                                                             \hspace{10.8cm} \textbf{end for}\\
                                                             \hspace{10.2cm} \textbf{end procedure}
       \end{tabbing}
The second part of the proposition follows identically after we put $t_{l_0+1}=\infty$.
\end{proof}

\indent The main result of this subsection is the following characterization of words in $E_1$ with isomorphic unrooted Schreier graphs.

\begin{thm}\label{ThmBijection}
Let $\xi, \eta \in E_1$.

\begin{enumerate}
\item If $\xi=u1^{\omega}$ and $\eta=w1^{\omega}$ with $u,w\in\{0,1\}^\ast$, then $\Gamma_\xi=\Gamma_\eta$. If $\xi=u1^{\omega}$ but $\eta\neq w1^{\omega}$ for any $w\in\{0,1\}^\ast$, then $\Gamma_\xi\not\simeq\Gamma_\eta$.
\item Suppose that $\xi\neq u1^{\omega}$ and $\eta\neq w1^{\omega}$ for any $u,w\in\{0,1\}^\ast$; and consider the respective triples $(k,\{m_l\}_{l\geq 0},\{t_l\}_{l\geq 0})$ and $(k',\{m'_l\}_{l\geq 0},\{t'_l\}_{l\geq 0})$ that give the canonical representation of $\xi$ and $\eta$
$$
\xi = 0^{k-1}1(0x^0_1\ldots 0x^0_{m_0/2})1^{t_1}(0x^1_1\ldots
0x^1_{m_1/2})1^{t_2}\ldots,
$$
$$
\eta = 0^{k'-1}1(0y^0_1\ldots 0y^0_{m'_0/2})1^{t'_1}(0y^1_1\ldots
0y^1_{m'_1/2})1^{t'_2}\ldots
$$

where $x^j_i, y^j_i\in \{0,1\}$.  Then $\Gamma_\eta\simeq\Gamma_\xi$ if and only if there exist $r,s\geq 1$
such that, for each $n\geq 0$,
\begin{enumerate}
\item\label{nuova8} $m_{r+n}=m'_{s+n}$, $t_{r+n+1}=t'_{s+n+1}$;
\item $k+T_r+M_{r-1} = k'+T'_s+M'_{s-1}$, where $M_i=\sum_{l=0}^im_l$ and $T_j=\sum_{l=0}^jt_l$.
\item either $y^{s+n}_p=x_p^{r+n}$ or $y_p^{s+n}=1-x_p^{r+n}$, for $p=1,\ldots, m_{r+n}/2$.
\end{enumerate}
\end{enumerate}
\end{thm}

In order to prove Theorem \ref{ThmBijection}, we need to study in more detail the geometry of the limit graphs $\Gamma_\xi$. It turns out that it is determined by the infinite cycle-path $\mathcal{CP}_\xi$ in $\Gamma_\xi$. In Propositions \ref{PropA_i} and \ref{introductiondesangles} below we describe these infinite cycle-paths completely by listing the lengths of the cycles and finding their relative position, in terms of the infinite ray $\xi\in E_1$ represented in the canonical form (\ref{CanFormOfXi}).

\begin{defi} \label{definizioneveraai}
Given $\xi\in E_1$, we associate with it the sequence of integers $a_i=a_i^\xi$, $i\geq 1$, defined as follows:
\begin{itemize}
\item if $\xi=1^\omega$, then $a_i:=i$ for all $i\geq 1$;
\item if $\xi\neq 1^\omega$, then Proposition \ref{PropAlg} provides a triple $(k,\{m_l\},\{t_l\})$  associated with $\xi$. Define for all $j\geq 1$, $0\leq s< t_j$,
    \begin{equation}
    a_{T_{j-1}+s+1}:=k+M_{j-1}+T_{j-1}+s.
    \end{equation}
\end{itemize}
\end{defi}

\begin{prop} \label{PropA_i}
Given $\xi\in E_1$, let $\mathcal{CP}_\xi=\{P_i\}_{i\geq 1}$ be the infinite cycle-path associated with $\xi$. Then, for all $i\geq 1$, $|P_i|=2^{\lceil\frac{a_i}{2}\rceil}$.
\end{prop}

\begin{proof}
By induction on $i$: if $i=1$, then by definition, $a_1=k+m_0$, that is the index of stability of $\xi$ (see (\ref{Eqindexofstab}) and Corollary \ref{RemonSEC1end}). In other words, the prefix $\xi_{k+m_0}$ of $\xi$ is situated on the central cycle of $\Gamma_{k+m_0}$. Thus, $\mathcal{CP}_{\xi_{k+m_0}}$ is constituted of a single cycle of length $2^{\lceil\frac{k+m_0}{2}\rceil}=2^{\lceil\frac{a_1}{2}\rceil}$. Moreover, since $\xi_{a_1+1}=\xi_{a_1}1$, then this cycle is the first one of the cycle-path $\mathcal{CP}_\xi$ (see Remark \ref{RemonSEC}).\\
\indent Suppose that the statement is true for $i$; the cycle $P_i$ of $\mathcal{CP}_\xi$ is then isomorphic to the central cycle of $\Gamma_{a_i}$. Consider the prefix $\xi_{a_i}$ of $\xi$ as well as the cycle-path $\mathcal{CP}_{\xi_{a_i}}$ in $\Gamma_{a_i}$.
Let $j\geq 1$, $0\leq s<t_j$ be such that $T_{j-1}+s+1=i$ so that $a_i=a_{T_{j-1}+s+1}$. We distinguish two cases:

\begin{itemize}
\item if $s\leq t_j-2$, then $a_{i+1}=\ (a_{T_{j-1}+s+2}=)\ k+M_{j-1}+T_{j-1}+s+1=a_i+1$. On one hand, $\xi_{a_{i+1}}=\xi_{a_i}1$, so that the central cycle of $\Gamma_{a_i}$ is isomorphic to the penultimate cycle of the cycle-path $\mathcal{CP}_{\xi_{a_i+1}}$ in $\Gamma_{a_i+1}$. On the other hand, $\xi_{a_i+2}=\xi_{a_i+1}1$ which implies that the central cycle of $\Gamma_{a_i+1}$ is isomorphic to the penultimate cycle of the cycle-path $\mathcal{CP}_{\xi_{a_i+2}}$ (see Remark \ref{RemonSEC}).  Hence, by induction hypothesis, the penultimate cycle of $\mathcal{CP}_{\xi_{a_i+2}}$ is the $(i+1)$-st cycle of $\mathcal{CP}_{\xi_{a_i+2}}$ which is isomorphic (by Lemma \ref{lemonCP}) to the $(i+1)$-st cycle of $\mathcal{CP}_\xi$. Its length is
    $|P_{i+1}|=2^{\lceil\frac{a_i+1}{2}\rceil}=2^{\lceil\frac{a_{i+1}}{2}\rceil}$.
\item if $s=t_j-1$, then $a_{i+1}=\ (a_{T_{j-1}+s+2}=)\ k+M_{j-1}+T_j+m_j=a_{i}+1+m_j$. Moreover, $\xi_{a_{i+1}+1}=\xi_{a_{i+1}}1$. Thus, as in the previous case, the last cycle of $\mathcal{CP}_{\xi_{a_{i+1}}}$ is isomorphic to the $(i+1)$-st cycle of $\mathcal{CP}_\xi$. By induction hypothesis, its length is
\begin{displaymath}
    |P_{i+1}|=\left\{\begin{array}{cc}
    |P_i|\cdot2^{\frac{m_j}{2}+1}=2^{\frac{a_i}{2}+\frac{m_j}{2}+1}=2^\frac{a_{i+1}+1}{2}=2^{\lceil\frac{a_{i+1}}{2}\rceil} & \textrm{if $a_i$ is even,}\\
    |P_i|\cdot2^{\frac{m_j}{2}}=2^{\frac{a_i+1}{2}+\frac{m_j}{2}}=2^{\frac{a_{i+1}}{2}}=2^{\lceil\frac{a_{i+1}}{2}\rceil} & \textrm{if $a_i$ is odd.}
    \end{array}\right.
    \end{displaymath}
\end{itemize}
\end{proof}

Given $\xi\in E_1$, for any $n\geq 1$, consider the cycle path $\mathcal{CP}_{\xi_n}=\{P_i^n\}_{i=1}^t$ in $\Gamma_n$. Set $v_0^n:= \xi_n$ and $\{v_i^n\}:=P_i^n\cap P_{i+1}^n$ for every $i=1,\ldots, t-1$. 
By Lemma \ref{lemonCP}, the limit $v_i:=\lim_{n\to\infty}v^n_i$ exists for each $i\geq 1$; $\{v_i\}=P_i\cap P_{i+1}$ where $P_i$, $P_{i+1}\in\mathcal{CP}_\xi$. 
Since we consider the graph $\Gamma_n$ embedded in the plane in such a way that every  $P_i^n$ is a regular polygon, we may introduce the notion of an angle between two consecutive polygons, as follows.

\begin{defi} \label{defangles}
Let $\xi\in E_1$. For any $n\geq 1$, consider
$\mathcal{CP}_{\xi_n}=\{P_i^n\}_{i=1}^t$. For $1\leq i\leq t-1$,
$\alpha_i^n$ is the (counterclockwise) angle between vectors
$\overrightarrow{c_i^nv_{i-1}^n}$ and
$\overrightarrow{c^n_iv_i^n}$ (see Fig. 11). The limit
$\alpha_i:=\lim_{n\to\infty}\alpha_i^n$ is well defined for each
$i\geq 1$; we
call $\alpha_i$ the angle between cycles $P_{i-1}$ and $P_i$ (in
$\mathcal{CP}_\xi$).
\end{defi}

\begin{center}
\psfrag{Pni}{$P^n_i$}\psfrag{Pni-1}{$P^n_{i-1}$}\psfrag{Pni+1}{$P^n_{i+1}$}\psfrag{Fig.9}{\textbf{Fig. 11.}}
\psfrag{cni}{$c^n_i$}\psfrag{vni-1}{$v^n_{i-1}$}\psfrag{vni}{$v^n_i$}\psfrag{alphani}{$\alpha_i^n$}
\includegraphics[width=0.5\textwidth]{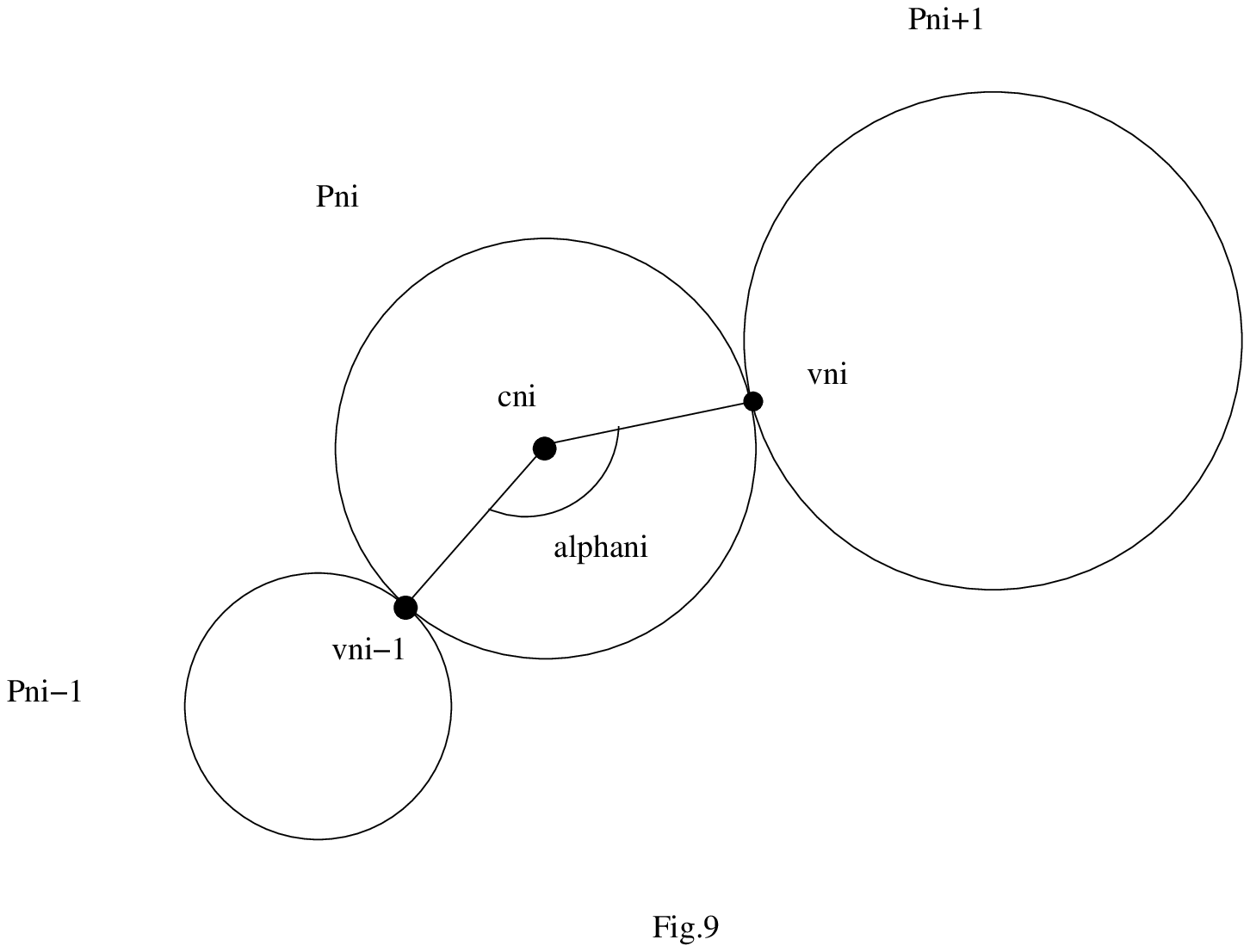}
\end{center}

The next proposition shows that, given $\xi\in E_1$, the angle $\alpha_i$ can be expressed, for each $i\geq 1$, in terms of the canonical representation (\ref{CanFormOfXi}) of $\xi$.

\begin{prop}\label{introductiondesangles}
Let $\xi\in E_1$ and let $(k,\{m_l\},\{t_l\})$  be the triple provided by Proposition \ref{PropAlg}. Let $\{a_i\}_{i\geq 1}$ be the sequence associated with $\xi$ as in Definition \ref{definizioneveraai}. Then,

\begin{itemize}
\item if $m_0=0$, then $\alpha_1=\pi$, whereas if $m_0>0$, then $\alpha_1=
\pi/2^{m_0/2}+\sum_{h=1}^{m_0/2}x_h\pi/2^{m_0/2-h}$;
\item if $i>1$ and $a_i-a_{i-1}=1$, then $\alpha_i= \pi$;
\item if $i>1$ and $a_i-a_{i-1}>1$, then there exists $d\geq 1$ such that $a_i-a_{i-1}=m_d+1$. In this case $\alpha_i =
\pi/2^{m_d/2}+\sum_{h=1}^{m_d/2}x_h^d\pi/2^{m_d/2-h}$.
\end{itemize}
\end{prop}

\vspace{0.5cm}
\unitlength=0,3mm

\begin{center}
\begin{picture}(400,210)
\letvertex I=(160,110)\letvertex N=(170,140)
\letvertex O=(200,150)\letvertex R=(230,140)\letvertex S=(240,110)\letvertex T=(230,80)
\letvertex U=(200,70)\letvertex V=(170,80)\letvertex P=(200,180)\letvertex Q=(200,210)
\letvertex Z=(200,40)\letvertex J=(200,10)\letvertex K=(260,130)\letvertex X=(280,110)
\letvertex W=(260,90)\letvertex g=(260,160)\letvertex h=(260,60)\letvertex c=(300,130)
\letvertex Y=(300,90)\letvertex d=(320,110)\letvertex e=(360,110)\letvertex f=(400,110)
\put(396,98){$1^{\omega}$}

\letvertex II=(60,160)\letvertex NN=(60,60)

\letvertex A=(120,220)\letvertex B=(120,200)
\letvertex C=(120,180)\letvertex D=(105,165)\letvertex E=(135,165)\letvertex F=(120,150)
\letvertex G=(90,140)\letvertex H=(77,153)\letvertex i=(80,110)\letvertex l=(90,80)
\letvertex m=(77,67)\letvertex n=(105,55)\letvertex o=(120,70)\letvertex p=(120,40)
\letvertex q=(120,20)\letvertex r=(120,0)\letvertex s=(135,55)\letvertex t=(150,80)
\letvertex u=(163,67)\letvertex v=(150,140)\letvertex z=(163,153)

\put(140,-15){\textbf{Fig. 12.} A finite part of $\Gamma_{1^{\omega}}$.}

\drawvertex(I){$\bullet$}\drawvertex(N){$\bullet$}\drawvertex(O){$\bullet$}\drawvertex(P){$\bullet$}
\drawvertex(J){$\bullet$}\drawvertex(K){$\bullet$}\drawvertex(Q){$\bullet$}\drawvertex(R){$\bullet$}
\drawvertex(S){$\bullet$}\drawvertex(T){$\bullet$}\drawvertex(U){$\bullet$}\drawvertex(V){$\bullet$}
\drawvertex(W){$\bullet$}\drawvertex(X){$\bullet$}\drawvertex(Y){$\bullet$}\drawvertex(Z){$\bullet$}
\drawvertex(g){$\bullet$}\drawvertex(h){$\bullet$}\drawvertex(c){$\bullet$}\drawvertex(f){$\bullet$}
\drawvertex(d){$\bullet$}\drawvertex(e){$\bullet$}\drawvertex(A){$\bullet$}\drawvertex(n){$\bullet$}
\drawvertex(B){$\bullet$}\drawvertex(o){$\bullet$}\drawvertex(C){$\bullet$}\drawvertex(p){$\bullet$}
\drawvertex(D){$\bullet$}\drawvertex(q){$\bullet$}\drawvertex(E){$\bullet$}\drawvertex(r){$\bullet$}
\drawvertex(F){$\bullet$}\drawvertex(s){$\bullet$}\drawvertex(G){$\bullet$}\drawvertex(t){$\bullet$}
\drawvertex(H){$\bullet$}\drawvertex(u){$\bullet$}\drawvertex(i){$\bullet$}\drawvertex(v){$\bullet$}
\drawvertex(l){$\bullet$}\drawvertex(z){$\bullet$}\drawvertex(m){$\bullet$}

\drawundirectededge(II,i){} \drawundirectededge(NN,i){}

\drawundirectededge(I,N){} \drawundirectededge(N,O){}
\drawundirectededge(O,R){} \drawundirectededge(R,S){}
\drawundirectededge(S,T){} \drawundirectededge(T,U){}
\drawundirectededge(U,V){} \drawundirectededge(V,I){}

\drawundirectedcurvededge(O,P){}\drawundirectedcurvededge(P,O){}
\drawundirectedcurvededge(Q,P){}\drawundirectedcurvededge(P,Q){}

\drawundirectedcurvededge(U,Z){}\drawundirectedcurvededge(Z,U){}
\drawundirectedcurvededge(Z,J){}\drawundirectedcurvededge(J,Z){}

\drawundirectededge(S,K){} \drawundirectededge(K,X){}
\drawundirectededge(X,W){} \drawundirectededge(W,S){}
\drawundirectededge(X,c){} \drawundirectededge(c,d){}
\drawundirectededge(d,Y){} \drawundirectededge(Y,X){}

\drawundirectedcurvededge(d,e){}\drawundirectedcurvededge(e,d){}
\drawundirectedcurvededge(e,f){}\drawundirectedcurvededge(f,e){}
\drawundirectedcurvededge(K,g){}\drawundirectedcurvededge(g,K){}
\drawundirectedcurvededge(W,h){}\drawundirectedcurvededge(h,W){}

\drawundirectedloop(Q){}\drawundirectedloop[b](J){}
\drawundirectedloop(g){}\drawundirectedloop[b](h){}\drawundirectedloop(c){}\drawundirectedloop[b](Y){}
\drawundirectedloop[r](f){}\drawundirectedloop[b](V){}\drawundirectedloop(N){}\drawundirectedloop[b](T){}
\drawundirectedloop(R){}

\drawundirectededge(C,D){}\drawundirectededge(D,F){}\drawundirectededge(F,E){}\drawundirectededge(E,C){}\drawundirectededge(F,G){}
\drawundirectededge(G,i){}\drawundirectededge(i,l){}\drawundirectededge(l,o){}\drawundirectededge(o,t){}
\drawundirectededge(t,I){}\drawundirectededge(I,v){}\drawundirectededge(v,F){}\drawundirectededge(o,n){}
\drawundirectededge(n,p){}\drawundirectededge(p,s){}\drawundirectededge(s,o){}

\drawundirectedcurvededge(A,B){}\drawundirectedcurvededge(B,A){}\drawundirectedcurvededge(B,C){}\drawundirectedcurvededge(C,B){}
\drawundirectedcurvededge(G,H){}\drawundirectedcurvededge(H,G){}\drawundirectedcurvededge(l,m){}\drawundirectedcurvededge(m,l){}
\drawundirectedcurvededge(t,u){}\drawundirectedcurvededge(u,t){}\drawundirectedcurvededge(v,z){}\drawundirectedcurvededge(z,v){}
\drawundirectedcurvededge(p,q){}\drawundirectedcurvededge(q,p){}\drawundirectedcurvededge(q,r){}\drawundirectedcurvededge(r,q){}

\drawundirectedloop(A){}\drawundirectedloop[l](D){}\drawundirectedloop(H){}
\drawundirectedloop[r](E){}\drawundirectedloop[l](n){}\drawundirectedloop[r](s){}\drawundirectedloop[b](r){}\drawundirectedloop[b](m){}\drawundirectedloop(z){}\drawundirectedloop[b](u){}
\end{picture}
\end{center}
\vspace{1cm}

\begin{proof}
Suppose that $\xi\in E_1$ is such that $\xi\neq w1^\omega$ for any $w\in\{0,1\}^\ast$ (the other case is treated similarly). By Corollary \ref{RemonSEC1end}, the SEC-sequence associated with $\xi$ is
$S^{k+m_0}E^{t_1}(EC)^{\frac{m_1}{2}}E^{t_2}(EC)^{\frac{m_2}{2}}\ldots E^{t_d}(EC)^{\frac{m_d}{2}}..$.
If $a_i-a_{i-1}=1$, then there exists $n\geq 0$ such that the $a_i$-th and $(a_i+1)$-st letter in (\ref{SEC}) belong to the block $E^{t_n}$ and $\xi_{a_i+1}=\xi_{a_i}1=\xi_{a_i-1}11$. Hence, $\alpha_i=\pi$.\\
\indent If $a_i-a_{i-1}>1$, then there exists $d\geq 1$
such that $a_i-a_{i-1}=m_d+1$. More precisely, we have $\xi_{a_i}=\xi_{a_{i-1}}10x^d_10x^d_2\ldots 0x^d_{\frac{m_d}{2}}$, where $x^d_j\in \{0,1\}$. Corollary \ref{RemonSEC1end} tells us that every letter $x_j^d$ corresponds to a contraction. The value of $x_j^d$ determines whether, after the contraction, $\xi_{a_{i-1}+1+2j}$ is in the upper or in the lower part of the diagram in Fig. 6.\\
\indent Fix $i\geq 2$. We prove by induction on $m_d$, that
$$
\alpha_i =
\frac{\pi}{2^{\frac{m_d}{2}}}+\sum_{h=1}^{\frac{m_d}{2}}\frac{\pi}{2^{\frac{m_d}{2}-h}}x_h^d.
$$
If $m_d=2$, then (\ref{CanFormOfXi}) becomes
\begin{displaymath}
\xi=0^{k-1}1(0x^0_1\ldots 0x^0_{\frac{m_0}{2}})1^{t_1}(0x^1_1\ldots 0x^1_{\frac{m_1}{2}})1^{t_2}(0x^2_1\ldots 0x^2_{\frac{m_2}{2}})\ldots 1^{t_d}(0x^d_1)1^{t_{d+1}}\ldots
\end{displaymath}
and (\ref{SEC}) becomes
\begin{displaymath}
S^{k+m_0}E^{t_1}(EC)^{\frac{m_1}{2}}E^{t_2}(EC)^{\frac{m_2}{2}}\ldots E^{t_d}(EC)E^{t_{d+1}}\ldots.
\end{displaymath}
We look at how $P_{i-1}$ is attached to $P_i$ in $\mathcal{CP}_\xi$. Consider the finite graph $\Gamma_{a_{i-1}+1}$; since
\begin{displaymath}
\xi_{a_{i-1}+1}=0^{k-1}1(0x^0_1\ldots 0x^0_{m_0/2})1^{t_1}(0x^1_1\ldots 0x^1_{m_1/2})1^{t_2}(0x^2_1\ldots 0x^2_{m_2/2})\ldots 1^{t_d}
\end{displaymath}
$\xi_{a_{i-1}+1}$ is situated in the right part of $\Gamma_{a_{i-1}+1}$:

\begin{center}
\psfrag{Pi}{$P_i$}\psfrag{Fig.10}{\textbf{Fig. 13.}}
\includegraphics[width=0.2\textwidth]{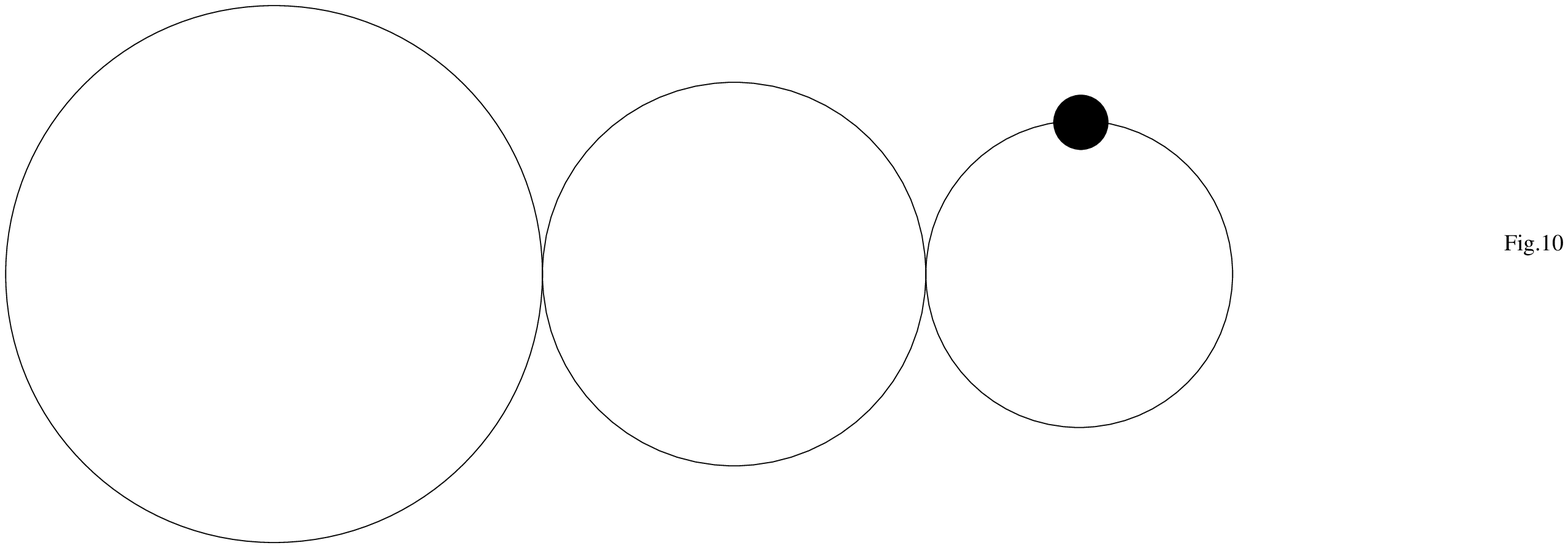}
\end{center}

\noindent Then, since $\xi_{a_{i-1}+2}=\xi_{a_{i-1}+1}0$, we get

\begin{center}
\psfrag{Fig.11}{\textbf{Fig. 14.}}
\includegraphics[width=0.3\textwidth]{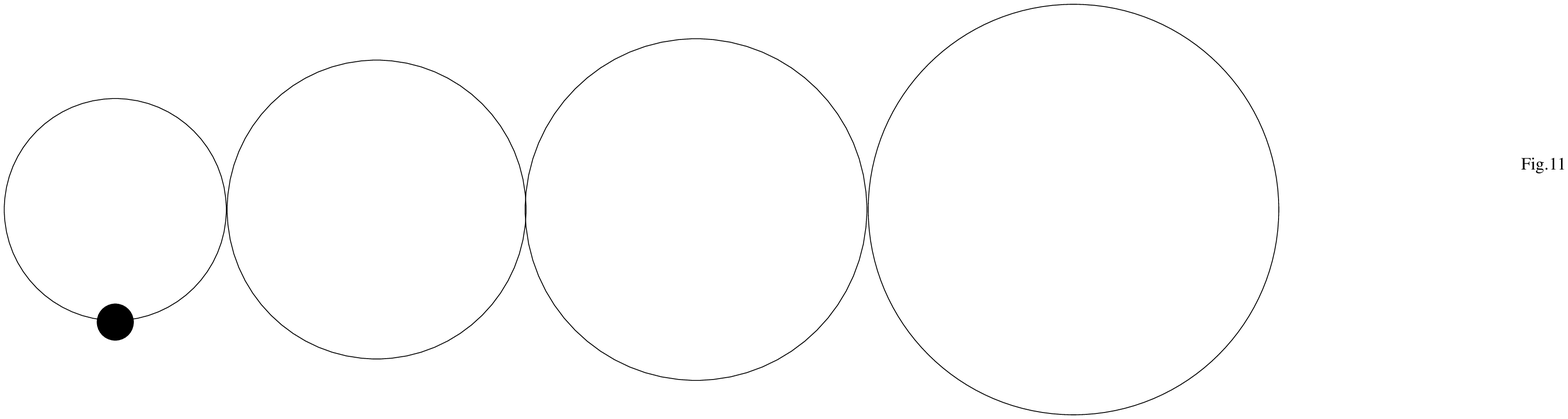}
\end{center}

\noindent Suppose now that $x^d_1=0$. (If $x_1^d=1$ we get a symmetric picture with the decoration growing downwards.) We have the following picture for $\mathcal CP_{\xi_{a_{i-1}+3}}\subset \Gamma_{a_{i-1}+3}$:

\begin{center}
\psfrag{Fig.12}{\textbf{Fig. 15.}}
\includegraphics[width=0.1\textwidth]{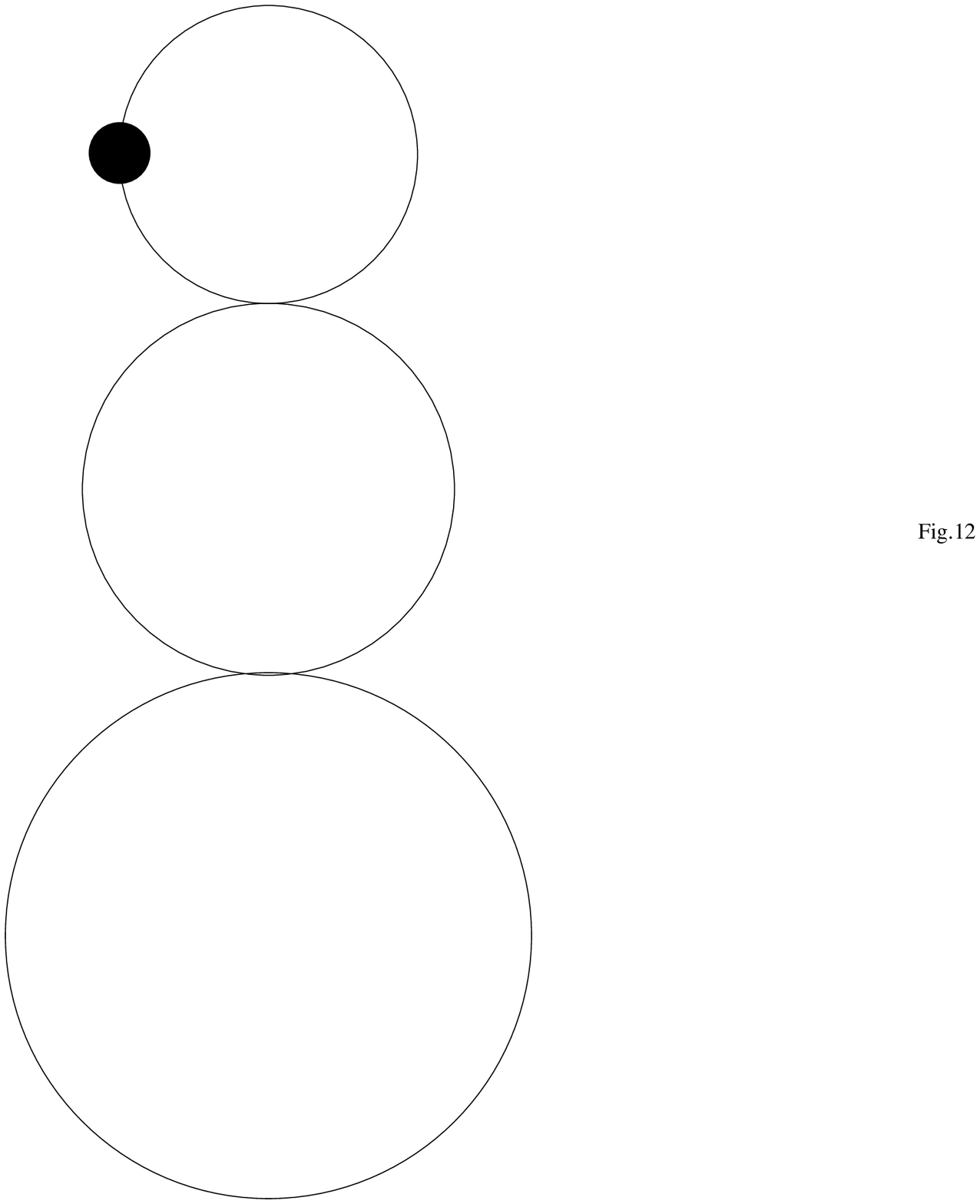}
\end{center}

\noindent Finally, adding $1$ (corresponding to the first letter of the block $E^{t_{d+1}}$) to the right of $\xi_{a_{i-1}+3}$, we get

\begin{center}
\psfrag{alphai}{$\alpha_i$}\psfrag{ci}{$c_i$}\psfrag{Pi}{$P_i$}\psfrag{vi-1}{$v_{i-1}$}\psfrag{Fig.13}{\textbf{Fig. 16.}}
\includegraphics[width=0.2\textwidth]{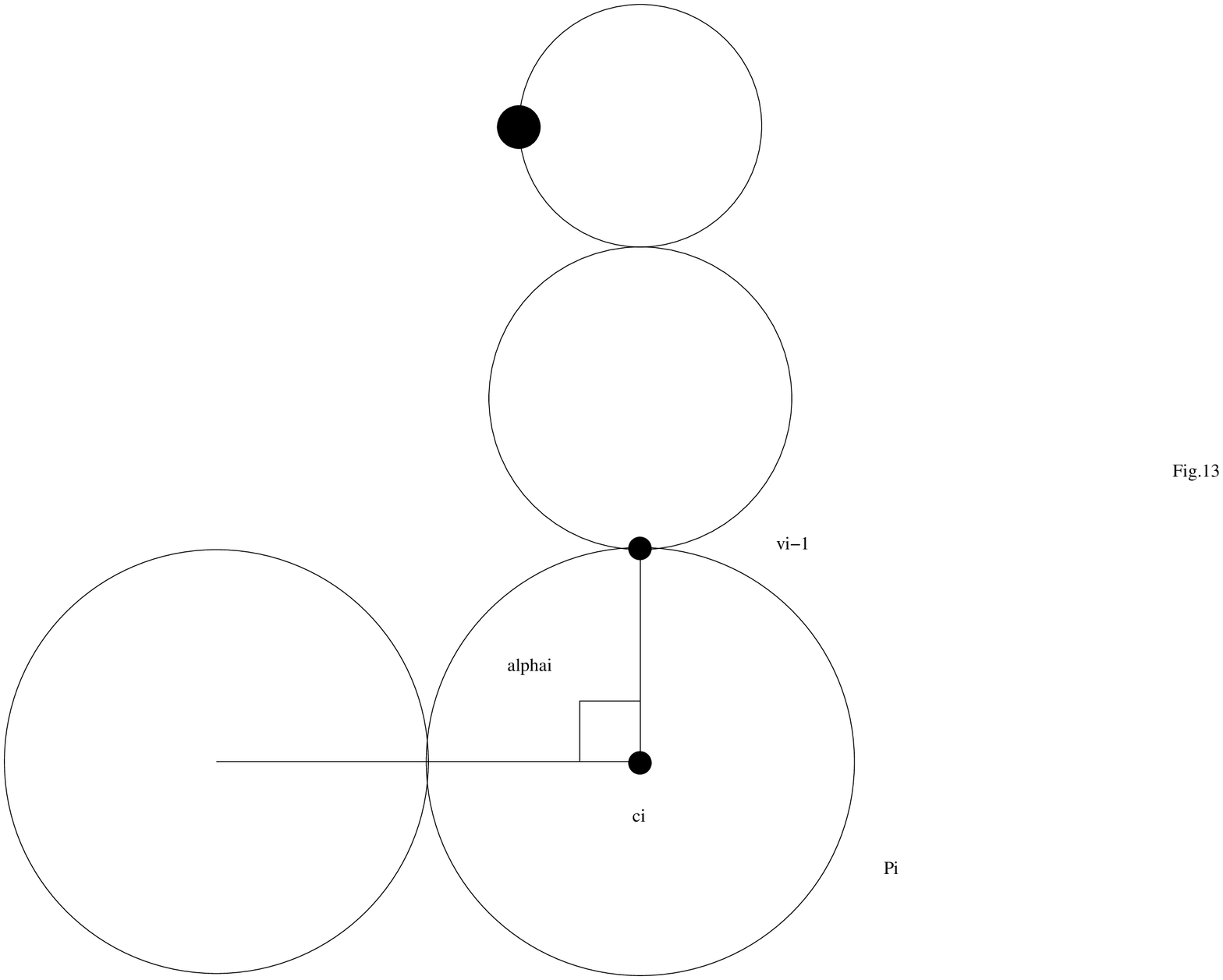}
\end{center}

\noindent It follows that $\alpha_i=\frac{\pi}{2}=\frac{\pi}{2^1}+\sum_{h=1}^1\frac{\pi}{2^{m_d-h}}\cdot 0$.\\
\indent Suppose now that the assertion is true for $m_d-2$; consider the graph $\Gamma_{a_i-2}$ and
\begin{displaymath}
\xi_{a_i}=0^{k-1}1(0x^0_1\ldots 0x^0_{m_0/2})1^{t_1}(0x^1_1\ldots 0x^1_{m_1/2})1^{t_2}(0x^2_1\ldots 0x^2_{m_d/2})\ldots 1^{t_d}(0x^d_1\ldots 0x^d_{(m_d-2)/2}0x^d_{m_2/2}).
\end{displaymath}
For $\mathcal CP_{\xi_{a_i-2}}\subset\Gamma_{a_i-2}$, define the angle $\alpha'$ as shown on Fig. 17. Then $\alpha'=\bar\alpha_i$ where $\bar\alpha_i$ is the $i$-th angle in the infinite cycle-path (as in Definition \ref{defangles}) for any $\bar\xi\in E_1$ with $\bar\xi_{a_i-1}=\xi_{a_i-2}1$ (i.e., $\bar a_i-\bar a_{i-1}=(m_d-2)+1=m_d-1$.) Therefore we can apply the induction hypothesis to $\alpha'$, and we have $\alpha'=\pi/2^{m_d/2-1}+\sum_{h=1}^{m_d/2-1}x^d_h\pi/2^{m_d/2-h-1}$.\\
\indent  The following picture explains how $\alpha_i$ is computed from $\alpha'$. The two cases correspond to $x_{m_d/2}^d$ being $0$ or $1$.

\begin{center}
\includegraphics[width=0.9\textwidth]{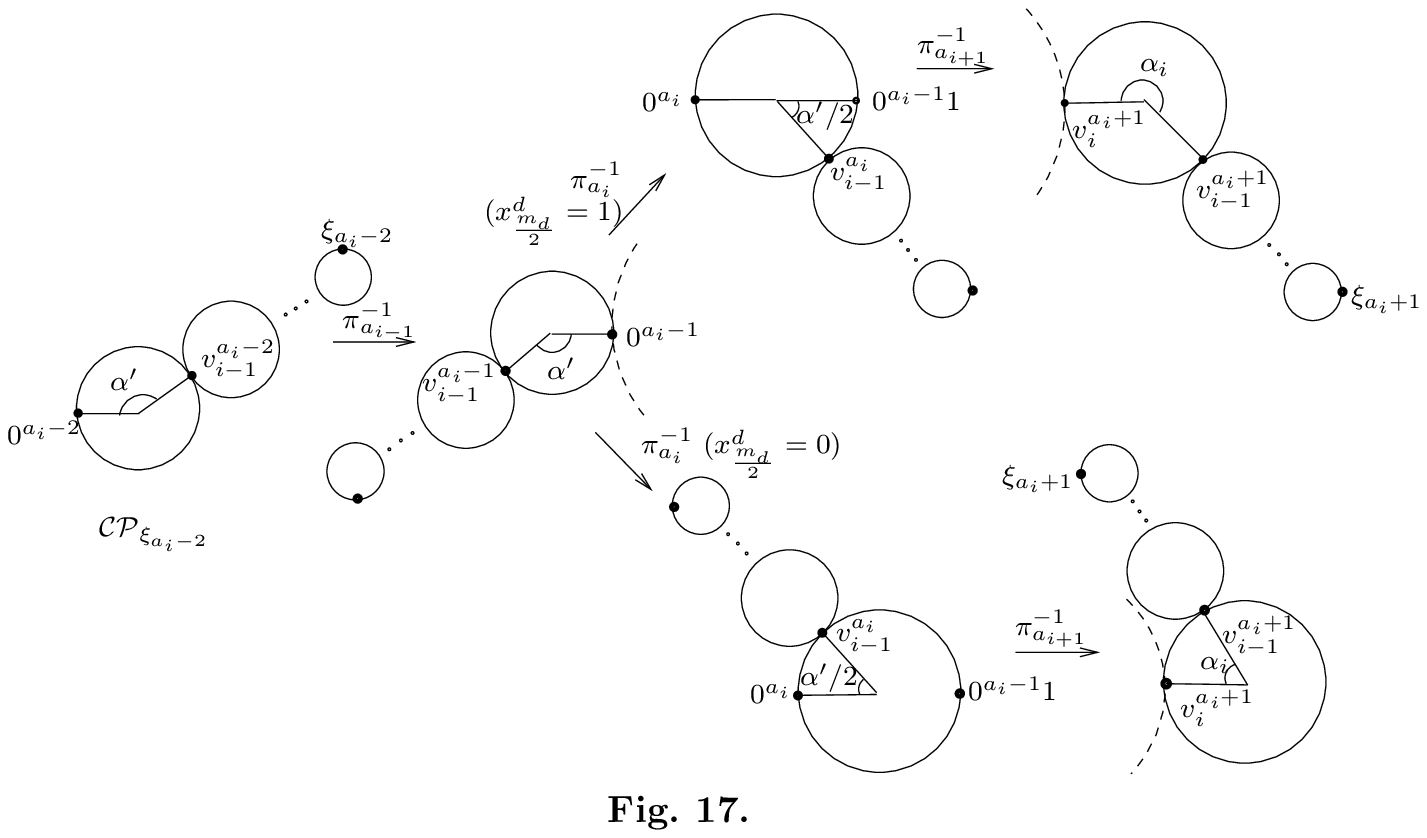}
\end{center}


\noindent We compute, using the induction hypothesis,

\begin{align*}
\alpha_i=\frac{\alpha'}{2} ( + \pi )&=
\frac{\pi}{2^\frac{m_d}{2}}+\sum_{h=1}^{\frac{m_d}{2}-1}\frac{\pi}{2^{\frac{m_d}{2}-h}}x^d_h ( + \pi )\\
&=\frac{\pi}{2^\frac{m_d}{2}}+\sum_{h=1}^\frac{m_d}{2}\frac{\pi}{2^{\frac{m_d}{2}-h}}x^d_h.
\end{align*}

The case $i=1$ can be treated similarly.
\end{proof}

The following statement is easily checked by a direct computation:

\begin{lem}\label{angoliexplementari}
Let
$\alpha_i=\frac{\pi}{2^{\frac{m_d}{2}}}+\sum_{h=1}^{\frac{m_d}{2}}\frac{\pi}{2^{\frac{m_d}{2}-h}}x^d_h$
and
$\beta_i=\frac{\pi}{2^{\frac{m_d}{2}}}+\sum_{h=1}^{\frac{m_d}{2}}\frac{\pi}{2^{\frac{m_d}{2}-h}}y^d_h$ where $m_d$ is even positive and $x_h^d,y_h^d\in\{0,1\}$.
Then $\alpha_i=\beta_i$ if and only if $x^d_h=y^d_h$ for each
$h=1,\ldots, \frac{m_d}{2}$ and $\alpha_i+\beta_i=2\pi$ if and
only if $y^d_h=1-x^d_h$ for each $h=1,\ldots,\frac{m_d}{2}$.
\end{lem}

\begin{es}\rm
Consider $\xi= 1(1100)^{\omega}$. The SEC-sequence associated with $\xi$ is $S\left(E^2(EC)\right)^{\omega}$ (so that
$k=1$, $m_0=0$ and $t_i=m_i=2$ for every $i\geq 1$). The sequence $\{a_i\}$ satisfy then
\begin{displaymath}
a_i= \left\{\begin{array}{cc}
2i-1 & \textrm{if $i$ is odd}, \\
2i-2 & \textrm{if $i$ is even}.
\end{array}\right.
\end{displaymath}
For each $i\geq 1$ one has $x^i_1=0$. Moreover, $a_{2k}-a_{2k-1}=1$ and $a_{2k+1}-a_{2k}= 3$ for each $k\geq 1$. This gives, by Proposition \ref{introductiondesangles}, $\alpha_1=\pi$, $\alpha_{2k}=\pi$ and $\alpha_{2k+1}=\frac{\pi}{2}$ for every $k\geq 1$ (see Fig. 18).

\begin{center}
\psfrag{xi}{$\xi$}\psfrag{Fig.19}{\textbf{Fig. 18.}}
\includegraphics[width=0.5\textwidth]{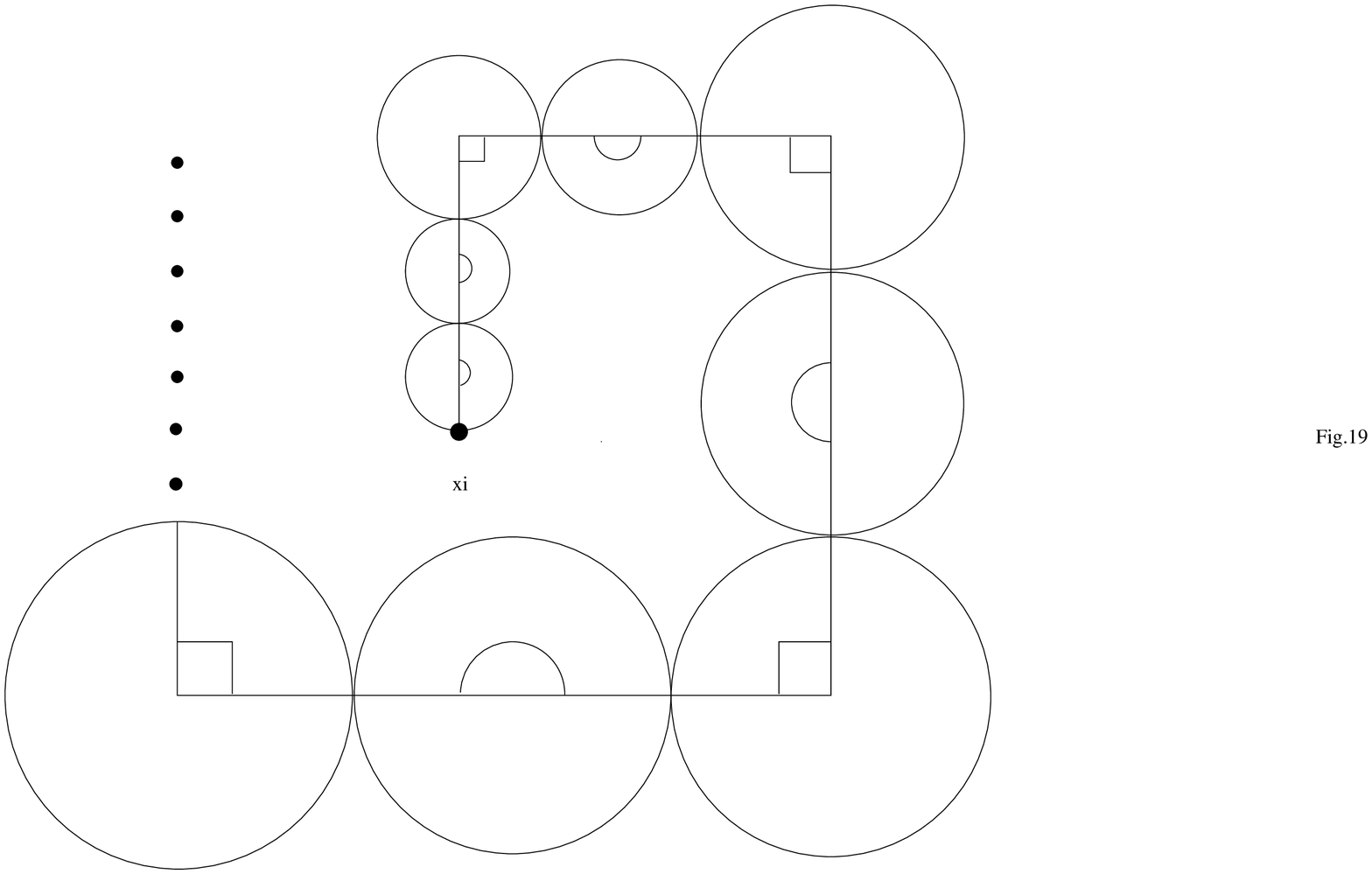}
\end{center}
\end{es}

\begin{proof}[Proof of Theorem \ref{ThmBijection}]
\emph{Part 1.} Suppose that $\xi=v1^{\omega}$ for some $v\in\{0,1\}^\ast$ and consider the associated sequence $\{a_i^\xi\}_{i\geq 1}$. If $\eta = w1^{\omega}$, $w\in\{0,1\}^\ast$, then Proposition \ref{cofinality} implies that $\Gamma_{\xi}=\Gamma_{\eta}$. If $\eta\neq w1^{\omega}$ for any $w\in\{0,1\}^\ast$, then there exist infinitely many indices $i$ such that $a_{i+1}^{\eta}-a_i^{\eta}>1$. Indeed, if $\eta\neq w1^{\omega}$, then by Proposition \ref{PropAlg}, there exist infinitely many indices $l\geq 1$ such that $m_l>0$ in the sequence $\{m_l\}_{l\geq 1}$ associated with $\eta$. By definition of the $a_i$'s (Definition \ref{definizioneveraai}), $a_{T_l+1}-a_{T_{l-1}+t_l}=m_l+1>1$.\\
\indent Hence, there are no indices $i, j$ such that $a^{\xi}_{i+n}=a^{\eta}_{j+n}$ for all $n\geq 0$. Proposition \ref{PropA_i} together with Lemma \ref{lemmaonlyoneend} imply that the graphs $\Gamma_{\xi}$ and $\Gamma_{\eta}$ are not isomorphic.\\
\indent \emph{Part 2.} Let $\xi \neq v1^{\omega}$ and $\eta\neq w1^{\omega}$ and suppose that there exist $r,s$ such that \emph{(a), (b)}, and
\emph{(c)} hold. It follows from Definition \ref{definizioneveraai}, that there exist $i,j$ such that $a_{i+n}^{\xi}=a_{j+n}^\eta$, for all $n\geq 0$. Writing $\Gamma_{\xi}\supset\mathcal{CP}_{\xi}=\{P_h\}_{h\geq 1}$ and $\Gamma_{\eta}\supset\mathcal{CP}_{\eta}=\{P'_h\}_{h\geq 1}$, let $\{v_{i}\}=P_{i}\cap P_{i+1}$ and similarly $\{v'_{j}\}=P'_{j}\cap P'_{j+1}$. Consider the decoration $X_{i}$ (respectively  $X'_{j}$) attached to $v_{i}$ (respectively  $v'_{j}$). Since $a_{i}^\xi=a_{j}^\eta$, these decorations are isomorphic via an isomorphism $\phi_0:X_{i}\longrightarrow X'_{j}$. We can extend the isomorphism $\phi_0:X_{i}\longrightarrow X'_{j}$ to an isomorphism
$\phi_1:X_{i+1}\longrightarrow X'_{j+1}$ where $X_{i+1}$ (respectively  $X'_{j+1}$) is the decoration attached to $v_{i+1}$ (respectively  $v'_{j+1}$). Clearly, $X_i\subset X_{i+1}$ (respectively  $X'_i\subset X'_{i+1}$).\\
\indent Indeed, by Proposition \ref{introductiondesangles} and Lemma \ref{angoliexplementari}, writing $\{v_{i+1}\}=P_{i+1}\cap P_{i+2}$ and
$\{v'_{j+1}\}=P'_{j+1}\cap P'_{j+2}$, one has $d_{\Gamma_{\xi}}(v_i,v_{i+1})=d_{\Gamma_{\eta}}(v'_j,v'_{j+1})$. Note that decorations of
$v_{i+1}$ and $v'_{j+1}$ are both $(a_{i+1}^{\xi}+1)$-decorations which means that they are both isomorphic to the decoration $\mathcal{D}(0^{a_{i+1}^{\xi}+1})$ of $0^{a_{i+1}^{\xi}+1}$ in $\Gamma_{a_{i+1}^{\xi}+1}$. This
decoration is encoded by the diagram $\bar{D}_{a_{i+1}^{\xi}}$ where boundary vertices  correspond to $v_{i+1}$ and $v'_{j+1}$ respectively. By Proposition \ref{PropRCR}, any two vertices situated at the
same distance from the boundary vertices of the diagram have isomorphic neighbourhoods.\\
\indent Repeating the above argument, one can construct a sequence of isomorphisms $\{\phi_n\}_{n\geq 0}$, $\phi_n:X_{i+n}\longrightarrow X'_{j+n}$, such that for every $n'>n$, the restriction of $\phi_{n'}$ to $X_{i+n}$ is equal to $\phi_n$. Thus, there is a well-defined isomorphism $\phi:\Gamma_{\xi}\longrightarrow \Gamma_{\eta}$ given by $\phi:=\lim_{n\to\infty}\phi_n$.\\
\indent Conversely, if we suppose that the conditions \emph{(a), (b)}, and \emph{(c)} are not all satisfied, then three possibilities may occur.
\begin{itemize}
\item There are no indices $i, j$ such that $a^{\xi}_{i+n}=a^{\eta}_{j+n}$ for all $n\geq 0$. By Proposition \ref{PropA_i} and Lemma \ref{lemmaonlyoneend}, the graphs $\Gamma_{\xi}$ and $\Gamma_{\eta}$ are not isomorphic.
\item There exist indices $i, j$ such that $a^{\xi}_{i+n}=a^{\eta}_{j+n}$ for all $n\geq 0$, but there are
infinitely many indices $d$ such that the condition $y^d_h=x^d_h$ or $y^d_h=1-x^d_h$ for each $h=1,\ldots, m_d/2$ is not satisfied. This implies that there are infinitely many indices $n$ such that $d(v_{i+n-1},v_{i+n})\neq
d(v_{j+n-1}',v_{j+n}')$. Indeed,
it follows from Proposition \ref{introductiondesangles}, that for any $d\geq 1$, the choice of the letters $x^d_h$, for $h=1,\ldots, \frac{m_d}{2}$, in the block $(0x^d_1\ldots 0x^d_{\frac{m_d-2}{2}})$ determines how $P_{i-1}$ is attached to $P_i$. In particular, the choices $x^d_h$ and $1-x^d_h$ for each $h=1,\ldots, \frac{m_d}{2}$ correspond to vertices $\{v_{i-1}\}=P_{i-1}\cap P_i$ and $\{v_{i-1}'\}=P_{i-1}\cap P_i$ situated at the same distance from $v_i$.\\
\indent On the other hand, an isomorphism $\phi$ between $\Gamma_{\xi}$ and $\Gamma_{\eta}$
must map $v_{i+n}$ onto $v_{j+n}'$ for each $n\geq 0$. This is a contradiction.
\item Finally, suppose that \emph{(a)} and \emph{(c)} are satisfied but not \emph{(b)}. Let $r_0,s_0>2$ such that $m_{r_0+n-1}=m'_{s_0+n-1}$ and $t_{r_0+n}=t'_{s_0+n}$ for each $n\geq 0$. We prove that, for any $i,j\geq 1$, there exists $q\geq 0$ such that $a^\xi_{i+q}\neq a^\eta_{j+q}$. Without loss of generality, we fix $i,j$ sufficiently large, so that there exist $r>r_0$, $0\leq u<t_r$, $s>s_0$ and $0\leq u'<t'_s$ such that $a^\xi_{i+q}=a^\xi_{T_{r-1}+u+1}$ and $a^\eta_{j+q}=a^\eta_{T'_{s-1}+u'+1}$.
    On the other hand, by Definition \ref{definizioneveraai}, for all $r>r_0$ and for all $0\leq u< t_r$,
    \begin{displaymath}
    a^\xi_{T_{r-1}+u+1}=k+M_{r-1}+T_{r-1}+u=k+M_{r_0-2}+T_{r_0-1}+t_{r_0}+\dots+t_{r-1}+m_{r_0-1}+\dots+m_{r-1}+u;
    \end{displaymath}

    \noindent similarly, for all $s>s_0$ and $0\leq u'< t'_s$,
    \begin{displaymath}
    a^\eta_{T'_{s-1}+u'+1}=k'+M'_{s-1}+T'_{s-1}+u'=k'+M'_{s_0-2}+T'_{s_0-1}+t'_{s_0}+\dots+t'_{s-1}+m'_{s_0-1}+\dots+m'_{s-1}+u'.
    \end{displaymath}

    \noindent Write $\delta:=k'+M'_{s_0-2}+T'_{s_0-1}-k-M_{r_0-2}-T_{r_0-1}$. If $s=r$, then $a^\eta_{T'_{s-1}+u'+1}-a^\xi_{T_{r-1}+u+1}=\delta+u'-u$ whereas if $s>r$, then $a^\eta_{T'_{s-1}+u'+1}-a^\xi_{T_{r-1}+u+1}=\delta+t'_{r}+\dots +t'_{s-1}+m'_r+\dots+m'_{s-1}+u'-u$ where $\delta\neq 0$ by hypothesis. Observe that the sequences $\{m_l\}$ and $\{t_l\}$ (and hence $\{m'_l\}$ and $\{t'_l\}$) cannot be eventually both constant. Otherwise, we may find $r_0,s_0$ such that \emph{(b)} would be satisfied which is a contradiction. Thus, since $u<t_r$, $u'<t'_s$, and $t'_s\geq 1$, $m'_{s-1}\geq 2$ for any $s>1$, it is easy to check that there exists an index $q$ such that $a^\eta_{j+q}-a^\xi_{i+q}=a^\eta_{T'_{s-1}+u'+1}-a^\xi_{T_{r-1}+u+1}\neq 0$. Hence, it follows from Proposition \ref{PropA_i} and Lemma \ref{lemmaonlyoneend}, that the graphs $\Gamma_{\xi}$ and $\Gamma_{\eta}$ are not isomorphic.

\end{itemize}
\end{proof}

\subsection{Random weak limit of $\{\Gamma_n\}_{n\geq 1}$} \label{SubsectionRandomLimit}

Given a sequence of finite connected graphs $\{\Gamma_n\}_{n\geq 1}$, one can consider it as a sequence of random rooted graphs $\{(\Gamma_n,v_n)\}_{n\geq 1}$ by choosing for each $n\geq 1$ a root $v_n\in V(\Gamma_n)$ uniformly at random. Then, one says that $(\Gamma,v)$ is the random weak limit (or distributional limit) of the sequence $\{(\Gamma_n,v_n)\}_{n\geq 1}$ as $n\to\infty$ if, for every $r>0$ and for every finite rooted graph $(H,o)$, the probability that $(H,o)$ is isomorphic to the ball in $\Gamma_n$ centered in $\xi_n$ and of radius $r$ converges to the probability that $(H,o)$ is isomorphic to the ball in $\Gamma$ centered in $\xi$ and of the same radius. This is equivalent to say that the law of $(\Gamma_n,v_n)$ weakly converges to the law of $(\Gamma,v)$ as probability measures on the space $\mathcal{X}$ of connected rooted graphs (see Subsection \ref{subsecSchreierGraphs}).\\
\indent If $\{(\Gamma_n,\xi_n)\}_{n\geq 1}$ are the Schreier graphs of a spherically transitive action of a finitely generated group of automorphisms of a regular rooted tree $T$;  and if the root $\xi_n$ is chosen uniformly at random in each $\Gamma_n$, then the random weak limit of this sequence is the uniform measure on the set $\{(\Gamma_\xi,\xi),\xi\in \partial T\}$ of rooted orbital Schreier graphs of the group action on $\partial T$. \\
\indent Observe that, since the action is spherically transitive, it is ergodic on $\partial T$. This implies that almost all infinite Schreier graphs of a given group have the same number of ends. We discuss this typical value of the number of ends in more detail in the forthcoming paper \cite{BDDN}.
Here, we show that in the case of the Basilica group, the set $\{(\Gamma_\xi,\xi),\xi\in E_1\}$ is of full measure. Moreover, we also show that if we partition $\{(\Gamma_\xi,\xi),\xi\in E_1\}$ into classes of isomorphisms of \emph{unrooted} graphs, then each isomorphism class is of measure $0$.\\
\indent So, we consider again the Basilica group $B$ acting by automorphisms on the binary tree, and we denote by $\Gamma_\xi$, $\xi\in\{0,1\}^\omega$, infinite Schreier graphs of the induced action of $B$ on the boundary of the binary tree identified with the set of infinite binary words and equipped with the uniform measure $\lambda$. Recall that $E_4$ is the set of infinite words defining Schreier graphs with four ends, $E_{2,even}$ and $E_{2,odd}$ are the sets defining Schreier graphs with two ends and, respectively, with odd or even decorations. Finally, $E_1$ denotes the set of infinite words defining Schreier graphs with one end.

\begin{prop}\label{unifmeasure} Almost every graph $\Gamma_\xi$ has one end with respect to
the uniform measure on the set $\{(\Gamma_\xi,\xi),\xi\in \{0,1\}^\omega\}$ of rooted orbital Schreier graphs of the action of the Basilica group $B$ on $\{0,1\}^\omega$.
\end{prop}

\begin{proof} We will show that $\lambda(E_1)=1$, where $\lambda$ is the uniform measure on $\{0,1\}^\omega$.
We have proven that $\{0,1\}^\omega=E_4\sqcup E_{2,even}\sqcup E_{2,odd} \sqcup E_1$. Each part is clearly invariant under $B$.\\
\indent An infinite word $\xi\in \{0,1\}^{\omega}$ is periodic if there exist  $w,u$ in $\{0,1\}^*$ such that $\xi = wu^{\omega}$, where $w$ is possibly empty. We denote by $\{0,1\}^{\omega}_{per}$ the set of all periodic infinite binary words. One checks easily that $\lambda(\{0,1\}^{\omega}_{per})=0$; indeed, $\{0,1\}^{\omega}_{per}$ can be written as the union $\bigcup_{k=0}^{\infty}Y_k$ where $Y_k:=\bigcup_{n=1}^{\infty}Y_{k,n}$ with $Y_{k,n}:=\{wu^{\omega} \ | \ |w|=k, |u|=n\}$. Clearly, $\lambda(Y_{k,n})=0$, since this set contains $2^{k+n}$ words.\\
\indent We have that $E_4\subset \{0,1\}^{\omega}_{per}$ and so $\lambda(E_4)=0$. It is clear that $\lambda(E_{2,even})=\lambda(E_{2,odd})$, and so, by ergodicity, $\lambda(E_{2,even})=\lambda(E_{2,odd})=0$. It follows that $\lambda(E_1)=1$.
\end{proof}

Given $\xi\in\{0,1\}^\omega$, write

$$
I_{\xi} := \{\eta\in \{0,1\}^\omega \ | \ \Gamma_{\xi} \textrm{ and }
\Gamma_{\eta} \ \textrm{are isomorphic as unrooted graphs}\}.
$$

\begin{prop}\label{classemisuranulla}
For any $\xi\in \{0,1\}^\omega$, $\lambda(I_{\xi})=0$.
\end{prop}

\begin{proof}
If $\xi\in E_4\sqcup E_2$, then the statement is obvious by the previous proposition. If $\xi\in E_1$, let $(k,\{m_l\},\{t_l\})$  be the triple provided by Proposition \ref{PropAlg}. Denote by $W_{z,\{t_l\},\{m_l\}}\subset E_1$ the set of words with triples satisfying $k+m_0=z$.
We show that $\lambda(W_{z,\{t_l\},\{m_l\}})=0$. By Proposition \ref{PropAlg}, $W_{z,\{t_l\},\{m_l\}}$ is constituted of words of the form
$$
0^{k-1}1(0x^0_1\ldots 0x^0_{m_0/2})1^{t_1}(0x_1^1\ldots
0x_{m_1/2}^1)\ldots 1^{t_l}(0x_1^l\ldots
0x_{m_l/2}^l)\ldots,
$$
with $x_i^j\in \{0,1\}$. From this, one has
\begin{eqnarray*}
\lambda(W_{z,\{t_l\},\{m_l\}})&=& 1-
\sum_{i=1}^k\frac{1}{2^i}-\sum_{i=1}^{\frac{m_0}{2}}\frac{2^{i-1}}{2^{k+2i-1}}-\sum_{i=1}^{t_1}\frac{2^{\frac{m_0}{2}}}{2^{k+m_0+i}}-
\sum_{i=1}^{\frac{m_1}{2}}\frac{2^{\frac{m_0}{2}+i-1}}{2^{k+m_0+t_1+2i-1}} \\
&-&
\sum_{i=2}^{\infty}\left(\sum_{j=1}^{t_i}\frac{2^{\frac{M_{i-1}}{2}}}{2^{T_{i-1}+M_{i-1}+k+j}}+
\sum_{h=1}^{\frac{m_i}{2}}\frac{2^{\frac{M_{i-1}}{2}+h-1}}{2^{T_i+M_{i-1}+k+2h-1}}\right).
\end{eqnarray*}
Setting $s_i:=T_i+M_i/2+k$, we get
\begin{align} \label{eqLaast}
\lambda(W_{z,\{t_l\},\{m_l\}})&=
\sum_{i=1}^k\frac{1}{2^i}-\sum_{i=1}^{\frac{m_0}{2}}\frac{1}{2^{k+i}}-\sum_{i=1}^{t_1}\frac{1}{2^{k+\frac{m_0}{2}+i}}-
\sum_{i=1}^{\frac{m_1}{2}}\frac{1}{2^{k+\frac{m_0}{2}+t_1+i}}-
\sum_{i=2}^{\infty}\left(\sum_{j=s_{i-1}+1}^{s_i}\frac{1}{2^j}\right)\\
&=1-\sum_{i=1}^{\infty}\frac{1}{2^i} =  0.
\end{align}
For any $i\geq 1$, consider the set $I_i^{\xi}\subset E_1$ defined as follows: $\eta\in I_i^{\xi}$ if and only if $i$ is the smallest integer such that the triple $(k',\{m'_l\},\{t'_l\})$ associated with $\eta$ satisfies $m'_{i+k}=m_{j+k}$, $t'_{i+k+1}=t_{j+k+1}$ and $z'+T_{i}'+M_{i-1}'=z+T_{j}+M_{j-1}$ for some $j\geq 1$ and each $k\geq 0$. By Theorem \ref{ThmBijection}, $I_{\xi}\subset \bigcup_{i=1}^{\infty}I_i^{\xi}$. Since for any $\xi\in\{0,1\}^\omega$, $\lambda(\{w\xi \ | \ w\in \{0,1\}^i\})=0$, (\ref{eqLaast}) implies that $\lambda(I_i^\xi)=0$ for every $i\geq 1$.
\end{proof}

\noindent {\bf Acknowledgments.} We are grateful to Volodymyr
Nekrashevych for pointing out a mistake in a preliminary version
of this paper.


\begin{thebibliography}{99}

\bibitem{BarDud} L. Bartholdi and D. Dudko, Iterated monodromy groups and linearizers, preliminary version, November 2009.

\bibitem{hecketype} L. Bartholdi and R. Grigorchuk, On the spectrum of Hecke type operators related to some fractal
groups, {\it Tr. Mat. Inst. Steklova}, 231 (2000), Din. Sist., Avtom. i Beskon. Gruppy, 5--45; translation in {\it Proc. Steklov Inst. Math.} 2000, no. 4 (231), 1--41.

\bibitem{amenability} L. Bartholdi and B. Vir\'{a}g, Amenability via random
walks, {\it Duke Math Journal}, 130 (2005), no. 1, 39--56.

\bibitem{BenSchr01} I. Benjamini and O. Schramm, Recurrence of distributional limits of finite planar graphs, {\it Electronic Journal of Probability}, 6 (2001), no. 23, 1--13.

\bibitem{bondarenko} I. Bondarenko, {\it Groups generated by bounded automata and their Schreier
graphs}, PhD Thesis Texas A\&M, 2007\\ \texttt{http://txspace.tamu.edu/bitstream/handle/1969.1/85845/Bondarenko.pdf?sequence=1}

\bibitem{BDDN} I. Bondarenko, D. D'Angeli, A. Donno, T. Nagnibeda, Topological invariants of Schreier graphs of self-similar groups, in preparation.

\bibitem{postcritically} I. Bondarenko and V. Nekrashevych, Post-critically finite self-similar groups, {\it Algebra and Discrete Mathematics}, 2 (2003), no. 4, 21--32.

\bibitem{GrigSunik06} R. Grigorchuk and Z. $\check{S}$uni\'c, Asymptotic aspects of Schreier graphs and Hanoi Towers groups, \emph{C. R. Math. Acad. Sci.}, Paris 342 (2006), no. 8, 545--550.

\bibitem{grizuk} R. Grigorchuk and A. \.{Z}uk, On a torsion-free weakly branch group defined by a three-state automaton. {\it International J. Algebra Comput.}, 12 (2002), no. 1, 223--246.

\bibitem{Grom} M. Gromov, Structures métriques pour les variétés riemanniennes, \emph{Textes Mathématiques}, {J}. {L}afontaine and {P}. {P}ansu ({E}ds.), 1. CEDIC, Paris, 1981. iv+152 pp. ISBN: 2-7124-0714-8.

\bibitem{MatNag09} M. Matter and T. Nagnibeda, Abelian sandpile model and self-similar groups, preprint, 2009.

\bibitem{nekrashevyc} V. Nekrashevych, {\it Self-similar Groups}, Mathematical Surveys and Monographs, 117. American Mathematical Society, Providence, RI, 2005. xii+231 pp. ISBN: 0-8218-3831-8

\bibitem{rogers} L. Rogers and A. Teplyaev, Laplacians on the basilica Julia set, to appear in {\it Commun. Pure Appl.
Anal.}, \texttt{http://arxiv.org/abs/0802.3248}

\bibitem{sidki} S. Sidki, Automorphisms of one-rooted trees: growth, circuit structure and acyclicity, {\it J. Math. Sci. (New York)}, 100 (2000), no. 1, 1925--1943.

\bibitem{Ver} A. Vershik, Fully non-free group actions and their characters, in preparation.

\end{thebibliography}
\end{document}